\documentclass[reqno]{amsart}
\usepackage{amsmath,amsthm,amssymb,amsfonts}
\usepackage[margin=1in]{geometry}
\usepackage{enumitem}
\usepackage{graphicx}
\usepackage{tikz}
  \usepackage{epsfig}
\usepackage[pagebackref, colorlinks=true]{hyperref}
\hypersetup{
    colorlinks=true,%
    citecolor=black,%
    filecolor=black,%
    linkcolor=black,%
    urlcolor=black
}

\def\doi#1{   {\href{http://dx.doi.org/#1}
   {{\mdseries\ttfamily DOI}}}}

\usepackage{graphics}
\def\Xint#1{\mathchoice
  {\XXint\displaystyle\textstyle{#1}}%
  {\XXint\textstyle\scriptstyle{#1}}%
  {\XXint\scriptstyle\scriptscriptstyle{#1}}%
  {\XXint\scriptscriptstyle\scriptscriptstyle{#1}}%
  \!\int}
\def\XXint#1#2#3{{\setbox0=\hbox{$#1{#2#3}{\int}$}
  \vcenter{\hbox{$#2#3$}}\kern-.5\wd0}}

\def\dashint{\Xint-}

\newcommand{\al}{\alpha}                
\newcommand{\om}{\Omega}                \newcommand{\pa}{\partial}
\newcommand{\va}{\varepsilon}           
\newcommand{\be}{\begin{equation}}      \newcommand{\ee}{\end{equation}}

\DeclareMathOperator{\dist}{dist}
\DeclareMathOperator{\avg}{avg}
\DeclareMathOperator{\diam}{diam}
\DeclareMathOperator{\BMO}{BMO}
\DeclareMathOperator{\VMO}{VMO}

\def\<{\langle}             \def\>{\rangle}
\def\({\left(}                 \def\){\right)}
\numberwithin{equation}{section}

\theoremstyle{plain}
\newtheorem{thm}{Theorem}[section]
\newtheorem{cor}[thm]{Corollary}

\newtheorem{lem}[thm]{Lemma}

\theoremstyle{definition}
\newtheorem{defn}[thm]{Definition}

\newtheorem{rem}[thm]{Remark}

\title[Homogenization Theory of Elliptic System with Lower Order Terms for Dimension Two]
{Homogenization Theory of Elliptic System with Lower Order Terms for Dimension Two}

\author{Wei Wang}
\address{Department of Mathematics\\Zhejiang University\\Hangzhou 310027, P. R. China}
		\email{wwmath166@outlook.com}

\author{Ting Zhang$^{*}$}
\thanks{This work is partially supported by the National Natural Science Foundation of China 11931010.}
\thanks{* Corresponding author: Ting Zhang.}
\subjclass{Primary: 35B27; Secondary: 35B25, 35B65, 35J25.}
\address{Department of Mathematics\\Zhejiang University\\Hangzhou 310027, P. R. China}
\email{zhangting79@zju.edu.cn}

\date{}

\begin{document}
\bibliographystyle{plain}

\maketitle

\begin{abstract}
In this paper, we consider the homogenization problem for generalized elliptic systems 
$$ 
\mathcal{L}_{\va}=-\operatorname{div}(A(x/\va)\nabla+V(x/\va))+B(x/\va)\nabla+c(x/\va)+\lambda I
$$
with dimension two. Precisely, we will establish the $ W^{1,p} $ estimates, H\"{o}lder estimates, Lipschitz estimates and $ L^p $ convergence results for $ \mathcal{L}_{\va} $ with dimension two. The operator $ \mathcal{L}_{\va} $ has been studied by Qiang Xu with dimension $ d\geq 3 $ in \cite{Xu1,Xu2} and the case $ d=2 $ is remained unsolved. As a byproduct, we will construct the Green functions for $ \mathcal{L}_{\va} $ with $ d=2 $ and their convergence rates.\\
\textbf{Keywords:} Homogenization, elliptic systems, lower order terms, two dimension.
\end{abstract}

\section{Introduction and Main results}

The homogenization theory, which has been gradually established since the 1970s, has high research value both in mathematics and mechanics. Its direct background is the equivalent research scheme of non-uniform material (static loading) or wave or medium (wave or oscillation) forces, elastic waves, and so on. In recent years, considerable advances have been made in the theory of homogenization for second-order linear elliptic systems in the divergence form with rapidly oscillating periodic coefficients,
\begin{align}
L_{\va}=-\operatorname{div}(A(x/\va)\nabla)=-\frac{\pa}{\pa x_i}\left\{a_{ij}^{\alpha\beta}\left(\frac{x}{\va}\right)\frac{\pa}{\pa x_j}\right\},\va>0,\nonumber
\end{align}
in a bounded domain $ \om $ in $ \mathbb{R}^d $, where $ 1\leq i,j\leq d $ and $ 1\leq\alpha,\beta\leq m $.

This paper concerns uniform regularity estimates for second order elliptic systems with lower order terms in a bounded domain in $ \mathbb{R}^2 $. More precisely, we consider
\begin{align}
\mathcal{L}_{\va}=-\operatorname{div}[A(x/\va)\nabla+V(x/\va)]+B(x/\va)\nabla+c(x/\va)+\lambda I,\label{LLva}
\end{align}
for the case that $ x\in\om\subset\mathbb{R}^d $, where $ \om $ is a bounded domain in $ \mathbb{R}^d $ and $ \lambda>0 $ is a sufficiently large constant.

Let $ 1\leq i,j\leq d $ and $ 1\leq\alpha,\beta\leq m $, where $ m\geq 1 $ denotes the number of equations in the system. Assume that the measurable functions $ A=(a_{ij}^{\alpha\beta}(x)):\mathbb{R}^d\to\mathbb{R}^{m^2\times d^2} $, $ V=(V_{i}^{\alpha\beta}(x)):\mathbb{R}^d\to\mathbb{R}^{m^2\times d} $, $ B=(B_i^{\alpha\beta}(x)):\mathbb{R}^d\to\mathbb{R}^{m^2\times d} $ and $ c=(c^{\alpha\beta}(x)):\mathbb{R}^d\to\mathbb{R}^{m^2} $ are coefficients of the elliptic operator $ \mathcal{L}_{\va} $. Furthermore, we can assume that the coefficients satisfy the following conditions:

(1) the uniform elliptic condition
\begin{align}
\mu|\xi|^2\leq a_{ij}^{\alpha\beta}(y)\xi_i^{\alpha}\xi_{j}^{\beta}\leq \mu^{-1}|\xi|^2\label{Ellipticity}
\end{align}
for all $ y\in \mathbb{R}^d $ and $ \xi=(\xi_{i}^{\alpha})\in\mathbb{R}^{m\times d} $, where $ \mu>0 $ is a positive constant;

(2) the periodicity condition
\begin{align}
A(y+z)=A(y),V(y+z)=V(y),B(y+z)=B(y),c(y+z)=c(y)\label{Periodicity}
\end{align}
for all $ y\in \mathbb{R}^d $ and $ z\in \mathbb{Z}^d $;

(3) the boundedness condition
\begin{align}
\max\left\{\left\|V\right\|_{L^{\infty}(\mathbb{R}^d)},\left\|B\right\|_{L^{\infty}(\mathbb{R}^d)},\left\|c\right\|_{L^{\infty}(\mathbb{R}^d)}\right\}\leq\kappa_1,\label{Boundedness}
\end{align}
where $ \kappa_1>0 $ is a positive constant;

(4) the H\"{o}lder regularity
\begin{align}
\max\left\{\left\|A\right\|_{C^{0,\tau}(\mathbb{R}^d)},\left\|V\right\|_{C^{0,\tau}(\mathbb{R}^d)},\left\|B\right\|_{C^{0,\tau}(\mathbb{R}^d)}\right\}\leq\kappa_2,\label{regularity}
\end{align}
where $ \tau\in (0,1) $ and $ \kappa_2>0 $ are positive constants.

 Let $ \kappa=\max\left\{\kappa_1,\kappa_2\right\} $ and we call $ A\in \Lambda(\mu,\tau,\kappa) $ if $ A=A(y) $ satisfies $ \eqref{Ellipticity} $, $ \eqref{Periodicity} $ and $ \eqref{regularity} $.

The uniform regularity estimates of the operator $ \eqref{LLva} $ have been studied by Qiang Xu in \cite{Xu1,Xu2}. In \cite{Xu1}, the author concentrated on the Dirichlet problems and derived uniform $ W^{1,p} $, H\"{o}lder and Lipschitz estimates for the case $ d\geq 3 $. The reason for the assumption $ d\geq 3 $ is that the Sobolev embedding theorem $ W^{1,2}(\mathbb{R}^d)\subset L^{\frac{2d}{d-2}}(\mathbb{R}^d) $, $ (d\geq 3) $ is not valid for the case $ d=2 $. The author also pointed out in \cite{Xu1} that for $ d=1,2 $, results in the paper are still true. For $ d=1 $, we see that it is actually an ODE problem and it should not be hard to solve. Therefore, in this paper, we mainly consider the case that $ d=2 $. The following are the main results of the paper.

\begin{thm}[$ W^{1,p} $ estimates]\label{W1p} Suppose that $ A\in \operatorname{VMO}(\mathbb{R}^2) $ satisfies $ \eqref{Ellipticity} $, $ \eqref{Periodicity} $, other coefficients of $ \mathcal{L}_{\va} $, $ V,B,c $ satisfy $ \eqref{Boundedness} $ and $ \om $ is a  bounded $ C^{1,\eta} $ $ (0<\eta<1) $ domain in $ \mathbb{R}^2 $. Let $ 1<p<\infty $, $ f\in L^p(\om;\mathbb{R}^{m\times 2}) $, $ F\in L^q(\om;\mathbb{R}^m) $ and $ g\in B^{1-\frac{1}{p},p}(\pa\om;\mathbb{R}^m) $, where $ q=\frac{2p}{p+2} $ if $ p>2 $, $ 1<q<\infty $ if $ p=2 $ and $ q=1 $ if $ 1<p<2 $. Then the Dirichlet problem
\begin{align}
\left\{\begin{array}{ccc}
\mathcal{L}_{\va}(u_{\va})=\operatorname{div}(f)+F  &\text{ in }  &\om, \\
 u_{\va}=g & \text{ on } &\pa\om,
\end{array}\right.\label{equation}
\end{align}
has a unique weak solution $ u_{\va}\in W^{1,p}(\om;\mathbb{R}^m) $, whenever $ \lambda\geq\lambda_0 $ and $ \lambda_0=\lambda_0(\mu,\kappa,m) $ is sufficiently large. Furthermore, the solution satisfies the uniform estimate
\begin{align}
\left\|\nabla u_{\va}\right\|_{L^{p}(\om)}\leq C\left\{\left\|f\right\|_{L^p(\om)}+\left\|F\right\|_{L^q(\om)}+\left\|g\right\|_{B^{1-\frac{1}{p},p}(\pa\om)}\right\},\label{W^{1,p} estimates}
\end{align}
where $ C $ depends only on $ \mu,\omega(t),\kappa,\lambda,p,q,m $ and $ \om $.
\end{thm}

Note that $ A\in \operatorname{VMO}(\mathbb{R}^d) $ if $ A $ satisfy
\begin{align}
\sup_{x\in\mathbb{R}^d,0<\rho<t}\dashint_{B(x,\rho)}\left|A(y)-\dashint_{B(x,\rho)}A\right|dy\leq\omega(t),\text{ and }\lim_{t\to 0}\omega(t)=0,\label{VMO condition}
\end{align}
where $ \omega(t) $ is a continuous nondecreasing function. $ B^{\alpha,p}(\pa\om;\mathbb{R}^m) $ denotes the $ L^p $ Besov space with order $ \alpha $ (see \cite{RA}). For the $ W^{1,p} $ estimates of the elliptic operator $ L_{\va} $ without lower order terms, one can refer to \cite{Shen3, Shen4, Shen5,Shen6}. We mention that the proof of $ W^{1,p} $ estimates for the case $ d=2 $ is almost the same as that for $ d\geq 3 $ given in \cite{Xu1}. The only difference is that the indices $ p $ and $ q $ are changed due to different Sobolev embedding theorems.

\begin{thm}\label{Holder}
Suppose that $ A\in \operatorname{VMO}(\mathbb{R}^2) $ satisfies $ \eqref{Ellipticity} $, $ \eqref{Periodicity} $, and other coefficients of $ \mathcal{L}_{\va} $, $ V,B,c $ satisfy $ \eqref{Boundedness} $ and $ \om $ is a bounded $ C^{1,\eta} $ $ (0<\eta<1) $ domain in $ \mathbb{R}^2 $. Let $ f\in L^p(\om;\mathbb{R}^{m\times 2}) $, $ F\in L^q(\om;\mathbb{R}^m) $ and $ g\in C^{0,\sigma}(\pa\om;\mathbb{R}^m) $, where $ 2<p<\infty $, $ q=\frac{2p}{p+2} $ and $ \sigma=1-\frac{2}{p} $. Then the weak solution to $ \eqref{equation} $ satisfies the uniform estimate
\begin{align}
\left\|u_{\va}\right\|_{C^{0,\sigma}(\om)}\leq C\left\{\left\|F\right\|_{L^q(\om)}+\left\|f\right\|_{L^p(\om)}+\left\|g\right\|_{C^{0,\sigma}(\pa\om)}\right\},\label{Holder estimates}
\end{align}
where $ C $ depends only on $ \mu,\omega(t),\kappa,\lambda,m,p,q,\sigma $ and $ \om $.
\end{thm}

Note that the estimate in Theorem \ref{Holder} is sharp. If we change the $ C^{0,\sigma} $ norm of $ g $ to $ C^{0,1} $ norm, $ \eqref{Holder estimates} $ is easy to verify by using Sobolev embedding theorem, since there exists an extension function $ G\in C^{0,1}(\om;\mathbb{R}^m) $ such that $ G=g $ on $ \pa\om $ and $ \left\|G\right\|_{C^{0,1}(\om)}\leq C\left\|g\right\|_{C^{0,1}(\om)} $. To get the sharp estimate, the author of \cite{Xu1} constructed the corresponding Green functions. The main difference between the cases $ d\geq 3 $ and $ d=2 $ is that the Green functions are different from each other. In this paper, we turn to study the Green functions for $ \mathcal{L}_{\va} $ with $ d=2 $. More precisely we can derive the following result.

\begin{thm}[Green functions of $ \mathcal{L}_{\va} $ with $ d=2 $]\label{Green}
Suppose that $ A $ satisfies $ \eqref{Ellipticity} $, $ \eqref{Periodicity} $, $ \operatorname{VMO} $ condition $ \eqref{VMO condition} $, other coefficients of $ \mathcal{L}_{\va} $, $ V,B,c $ satisfy $ \eqref{Boundedness} $ and $ \om $ is a bounded $ C^{1,\eta} $ $ (0<\eta<1) $ domain in $ \mathbb{R}^2 $. If $ \lambda\geq\lambda_0 $, then there exists a unique Green function $ G_{\va}:\om\times\om\to\mathbb{R}^{m^2}\cup\left\{\infty\right\} $ such that for all $ u_{\va} $ being the weak solution for the Dirichlet problem $ \mathcal{L}_{\va}(u_{\va})=F $ in $ \om $ and $ u_{\va}=0 $ on $ \pa\om $, where $ F\in L^p(\om;\mathbb{R}^m) $, $ (p>1) $, we have
\begin{align}
u_{\va}(x)=\int_{\om}G_{\va}(x,y)F(y)dy.\label{Representation formula}
\end{align}
Furthermore, for the Green function $ G_{\va}^*(x,y) $ corresponding to the dual operator of $ \mathcal{L}_{\va} $ denoted as $ \mathcal{L}_{\va}^* $, we have $ G_{\va}^*(x,y)=G_{\va}(y,x) $. For all $ \sigma_1,\sigma_2,\sigma_3,\sigma_4,\sigma\in (0,1) $, we have
\begin{align}
G_{\va}(\cdot,y)\in \operatorname{BMO}(\om) \text{ that is, } \left\|G_{\va}(\cdot,y)\right\|_*\leq C,\quad \text{uniformly for }y\in\om, \label{BMO estimates}
\end{align}
\begin{align}
|G_{\va}(x,y)|\leq\frac{C}{|x-y|^{\sigma}},\quad\text{for any }x,y\in\om,\label{preliminary}
\end{align}
\begin{align}
|G_{\va}(x,y)|\leq\frac{C[\delta(x)]^{\sigma_1}}{|x-y|^{\sigma_1}},   &\text{ if }   \delta(x)<\frac{1}{4}|x-y|,\label{Pointwise estimates for Green functions 1}\\
|G_{\va}(x,y)|\leq\frac{C[\delta(y)]^{\sigma_2}}{|x-y|^{\sigma_2}},   &\text{ if }   \delta(y)<\frac{1}{4}|x-y|,\label{Pointwise estimates for Green functions 2}\\
|G_{\va}(x,y)|\leq\frac{C[\delta(x)]^{\sigma_1}[\delta(y)]^{\sigma_2}}{|x-y|^{\sigma_1+\sigma_2}},   &\text{ if } \delta(x)<\frac{1}{4}|x-y| \text{ or } \delta(y)<\frac{1}{4}|x-y|,\label{Pointwise estimates for Green functions 3}\\
|G_{\va}(x,y)|\leq C\left(1+\ln\left(\frac{\operatorname{diam}(\om)}{|x-y|}\right)\right), &\text{ if }  \delta(x)\geq \frac{1}{4}|x-y|\text{ and } \delta(y)\geq \frac{1}{4}|x-y|,\label{Pointwise estimates for Green functions 4}
\end{align}
and
\begin{align}
|G_{\va}(x,y)-G_{\va}(z,y)|\leq\frac{C|x-z|^{\sigma_3}}{|x-y|^{\sigma_3}},\quad\text{if }|x-z|<\frac{1}{2}|x-y|,\label{Pointwise estimates for Green functions 5}\\
|G_{\va}(x,y)-G_{\va}(x,z)|\leq\frac{C|y-z|^{\sigma_4}}{|x-y|^{\sigma_4}},\quad\text{if }|y-z|<\frac{1}{2}|x-y|,\label{Pointwise estimates for Green functions 6}
\end{align}
where $ \delta(x)=\operatorname{dist}(x,\pa\om) $ denote the distance from $ x $ to the boundary of $ \om $, $ \operatorname{diam}(\om) $ denotes the diameter of $ \om $ and $ C $ is a constant depending only on $ \sigma_1,\sigma_2,\sigma_3,\sigma_4,\sigma,\mu,\omega(t),\kappa,\lambda,m $ and $ \om $.
\end{thm}

Here, the space $ \operatorname{BMO}(\om) $ (functions of bounded mean oscillation), is defined in Definition \ref{BMO atom}. To construct the Green functions with $ d=2 $, we will use the methods in \cite{Dong2, Dong1,Taylor}, which are mainly some arguments related to $ \operatorname{BMO} $ space.

In \cite{Xu1}, the methods for the proof of Lipschitz estimates can be applied to the case $ d=2 $ without any difficulty. In this paper, we do not repeat the proof for the sake of simplicity. The theorem is stated as follows.

\begin{thm}[Lipschitz estimates]\label{Lipschitz estimates}
Suppose that $ A\in\Lambda(\mu,\tau,\kappa),V $ satisfies $ \eqref{Periodicity} $, $ \eqref{regularity} $, $ B,c $ satisfy $ \eqref{Boundedness} $, $ \lambda\geq\lambda_{0} $ and $ \om $ is a bounded $ C^{1,\eta} $ $ (0<\eta<1) $ domain in $ \mathbb{R}^2 $. Let $ p>2 $ and $ 0<\sigma\leq\eta $. Then, for any $ f\in C^{0,\sigma}(\om; \mathbb{R}^{m d}), F\in L^{p}(\om;\mathbb{R}^{m}) $, and $ g\in C^{1,\sigma}(\pa\om;\mathbb{R}^{m}) $, the weak solution to $ \eqref{equation} $ satisfies the uniform estimate
\begin{align}
\left\|\nabla u_{\va}\right\|_{L^{\infty}(\om)} \leq C\left\{\|f\|_{C^{0, \sigma}(\om)}+\|F\|_{L^{p}(\om)}+\|g\|_{C^{1,\sigma}(\pa \om)}\right\},\label{Lipschitz estimates e}
\end{align}
where $ C $ depends only on $ \mu,\tau,\kappa,\lambda,p,m, \sigma,\eta $ and $ \om $.
\end{thm}

The uniform regularity estimates for the operators $ L_{\va} $ without lower order terms are first studied by Fanghua Lin and Marco Avellaneda in \cite{Av}. For the operator with the non-divergence form, see \cite{Av2}. Because of Theorem \ref{Lipschitz estimates}, we can establish the Lipschitz estimates for the Green functions with $ d=2 $ constructed in Theorem \ref{Green}.   To this point, we can obtain the following results.

\begin{thm}[Lipschitz estimates for Green functions]\label{Lipschitz estimates Green}
Suppose that $ A\in\Lambda(\mu,\tau,\kappa), V $ satisfies $ \eqref{Periodicity} $, $ \eqref{regularity} $, $ B,c $ satisfy $ \eqref{Boundedness} $, $ \lambda\geq\lambda_{0} $ and $ \om $ is a bounded $ C^{1,\eta} $ $ (0<\eta<1) $ domain in $ \mathbb{R}^2 $. Then the Green functions $ G_{\va}(x,y) $ constructed in Theorem \ref{Green} satisfy the uniform estimates
\begin{align}
|G_{\va}(x,y)|\leq\frac{C[\delta(x)]}{|x-y|}, &\text{ if } \delta(x)<\frac{1}{4}|x-y|,\label{Pointwise estimates for Green functions 11}\\
|G_{\va}(x,y)|\leq\frac{C[\delta(y)]}{|x-y|}, &\text{ if }\delta(y)<\frac{1}{4}|x-y|,\label{Pointwise estimates for Green functions 21}\\
|G_{\va}(x,y)|\leq\frac{C[\delta(x)][\delta(y)]}{|x-y|^2},   &\text{ if } \delta(x)<\frac{1}{4}|x-y| \text{ or } \delta(y)<\frac{1}{4}|x-y|.\label{Pointwise estimates for Green functions 31}
\end{align}
Furthermore, the Green functions satisfy the following uniformly Lipschitz estimates,
\begin{align}
|\nabla_xG_{\va}(x,y)|&\leq\frac{C}{|x-y|}\min\left\{1,\frac{\delta(y)}{|x-y|}\right\}, \label{Lipschitz estimates Green e 1}\\
|\nabla_yG_{\va}(x,y)|&\leq\frac{C}{|x-y|}\min\left\{1,\frac{\delta(x)}{|x-y|}\right\}, \label{Lipschitz estimates Green e 2}\\
|\nabla_x\nabla_y G_{\va}(x,y)|&\leq\frac{C}{|x-y|^2}, \label{Lipschitz estimates Green e 3}
\end{align}
where $ C $ depends on $ \mu,\tau,\kappa,\lambda,m, \sigma,\eta $ and $ \om $.
\end{thm}

\begin{thm}[Nontangential maximal function estimates]\label{Nontangential maximal function estimates} Suppose that $ A\in\Lambda(\mu,\tau,\kappa),V,B $ satisfy $ \eqref{Periodicity} $ and $ \eqref{regularity} $, $ c $ satisfies $ \eqref{Boundedness} $, $\lambda \geq \lambda_{0} $ and $ \om $ is a bounded $ C^{1,\eta} $ $ (0<\eta<1) $ domain in $ \mathbb{R}^d $ with $ d\geq 2 $. Let $ 1<p<\infty $, and $ u_{\va} $ be the solution of the $ L^{p} $ Dirichlet problem $ \mathcal{L}_{\va}(u_{\va})=0 $ in $ \om $ and $ u_{\va}=g $ on $ \pa\om $ with $ (u_{\va})^{*}\in L^{p}(\pa\om) $, where $ g\in L^{p}(\pa \om;\mathbb{R}^{m})$ and $ (u_{\va})^{*} $ is the nontangential maximal function.  Then
\begin{align}
\left\|(u_{\va})^{*}\right\|_{L^{p}(\pa \om)} \leq C_{p}\|g\|_{L^{p}(\pa \om)},\label{ww31}
\end{align}
where $ C_{p} $ depends on $ \mu,\tau,\kappa, \lambda,d,m,p,\eta $ and $ \om $. Furthermore, if $ g\in L^{\infty}(\pa\om;\mathbb{R}^{m}) $, we have
\begin{align}
\left\|u_{\va}\right\|_{L^{\infty}(\om)} \leq C\|g\|_{L^{\infty}(\pa\om)},\label{ww32}
\end{align}
where $ C $ depends on $ \mu,\tau,\kappa,\lambda,d,m, \eta $ and $ \om $.
\end{thm}
The estimate $ \eqref{ww32} $ is known as the Agmon-Miranda maximum principle, and $ (u_{\va})^* $ is the nontangential maximal function defined as follows.
\begin{align}
(u_{\va})^{*}(x)=\sup\left\{|u_{\va}(y)|:y\in\om\text{ and }|y-x|\leq C_{0}\dist(y,\pa\om)\right\},\label{nontangential maximal function}
\end{align}
for any $ x\in\pa\om $, where $ C_0 $ is a sufficiently large constant number. Because of Theorem \ref{Lipschitz estimates Green}, $ \eqref{ww31} $ follows directly by constructing Poisson kernel $ P_{\va}(x,y) $ associated with $ \mathcal{L}_{\va} $. The Poisson kernel is defined as follows.
\begin{align}
P_{\va}^{\gamma\beta}(x,y)=-n_{j}(y)a_{i j}^{\al\beta}(y/ \va)\pa_{y_j}\left\{G_{\va}^{\alpha \gamma}(x,y)\right\}-n_{j}(y)B_{j}^{\alpha\beta}(y/\va) G_{\va}^{\alpha \gamma}(x,y),\label{Poisson kernel}
\end{align}
with $ n(x)=(n_i(x))_{i=1}^d $ being the outward unit normal to $ \pa\om $. We will omit the proof for Theorem \ref{Nontangential maximal function estimates} since the proof for the case $ d=2 $ is almost the same as the case $ d\geq 3 $, which is already given in \cite{Xu1}. The proof is related to the Lipschitz estimates for the Green functions, for which we assume that $ V $ is divergence-free.

\begin{thm}[Convergence rates]
Suppose that $ A\in\Lambda(\mu,\tau,\kappa),V,B $ satisfy $ \eqref{Periodicity} $ and $ \eqref{regularity} $, $ c $ satisfies $ \eqref{Periodicity} $ and $ \eqref{Boundedness} $, $\lambda \geq \lambda_{0} $ and $ \om $ is a bounded $ C^{1,1} $ domain in $ \mathbb{R}^d $ with $ d\geq 2 $. Let $u_{\va}$ be the weak solution to $\mathcal{L}_{\va}\left(u_{\va}\right)=F$ in $\om$ and $u_{\va}=0$ on $\pa \om$, where $F \in L^{2}\left(\om ; \mathbb{R}^{m}\right)$. Then we have
\begin{align}
\left\|u_{\va}-\Phi_{\va,0}u_0-(\Phi_{\va, k}^{\beta}-P_{k}^{\beta})\pa_k u^{\beta}\right\|_{H_{0}^{1}(\om)}\leq C\va\|F\|_{L^{2}(\om)}\label{Convergence rates L2}
\end{align}
where $ u_0 $ satisfies $\mathcal{L}_{0}(u_0)=F $ in $\om $ and $ u_0=0 $ on $ \pa \om$. Moreover, assume that the coefficients of $\mathcal{L}_{\va}$ satisfy $ \eqref{Ellipticity} $-$ \eqref{regularity}$, then
\begin{align}
\left\|u_{\va}-u_0\right\|_{L^{q}(\om)} \leq C \va\|F\|_{L^{p}(\om)}\label{Convergence rates Lp}
\end{align}
holds for any $F \in L^{p}\left(\om ; \mathbb{R}^{m}\right)$, where $ q=\frac{\text{pd}}{d-p}$ if $1<p<d, q=\infty$ if $p>d$, and $C$ depends on $\mu, \tau, \kappa, \lambda, m, d, p$ and $\om$.
\end{thm}

Here, $ \Phi_{\va,k}=(\Phi_{\va,k}^{\alpha\beta}),0\leq k\leq d $ is the Drichlet corrector, associated with $ \mathcal{L}_{\va} $ as follows:
\begin{align}
L_{\va}\left(\Phi_{\va,0}\right)=\operatorname{div}\left(V_{\va}\right)\quad\text{in }\om,\quad\Phi_{\va, k}=I\quad\text{on }\pa\om,\label{Drichlet correctors 1}
\end{align}
and
\begin{align}
L_{\va}(\Phi_{\va,k}^{\beta})=0\quad\text{in }\om,\quad \Phi_{\va,k}^{\beta}=P_{k}^{\beta}\quad\text{on }\pa\om,\label{Drichlet correctors 2}
\end{align}
for $ 1\leq k\leq d $, where $ V_{\va}(x)=V(x/\va), \Phi_{\va,k}^{\beta}=(\Phi_{\va,k}^{1 \beta}, \cdots, \Phi_{\va, k}^{m \beta})\in W^{1,2}(\om;\mathbb{R}^{m})$, and $ P_{k}^{\beta}=x_{k}e^{\beta} $. Here, $ e^{\beta}=(0,...,1,...,0) $ with 1 in the $ \beta^{\operatorname{th }} $ position and $ 0 $ otherwise.

We remark that $ \eqref{Convergence rates L2} $ is a generalization for Theorem 2.4 in \cite{Kenig3}, which is established of the operator $ L_{\va} $. For $ d\geq 3 $, it is proved in \cite{Xu1}. For the case that $ d=2 $, the proof is almost the same and we will omit it for the sake of simplicity. For $ \eqref{Convergence rates Lp} $, it is a generalization of Theorem 3.4 in \cite{Kenig2}, which is about the elliptic operator $ L_{\va} $. In \cite{Xu1}, the author gave the simplified proof for $ d\geq 3 $ without using the convergence rates of the Green functions. In this paper, we will complete the proof of the theorem about convergence rates for the Green functions for $ \mathcal{L}_{\va} $, which is mentioned in \cite{Xu1} without proof. The method is from \cite{Kenig2}.
\section{Preliminaries}
Before the proofs of the main theorems of the paper, we will first introduce some results of the homogenization problems for the operator $ \mathcal{L}_{\va} $ with $ d\geq 2 $. The homogenized operator of $ \mathcal{L}_{\va} $ is defined as follows.
\begin{align}
\mathcal{L}_{0}=-\operatorname{div}(\widehat{A}\nabla+\widehat{V})+\widehat{B}\nabla+\widehat{c}+\lambda I.\nonumber
\end{align}
The coefficients of the homogenized operator $ \mathcal{L}_0 $, $ \widehat{A}=(\widehat{a}_{ij}^{\alpha\beta}) $, $ \widehat{V}=(\widehat{V}_{i}^{\alpha\beta}) $, $ \widehat{B}=(\widehat{B}_{i}^{\alpha\beta}) $ and $ \widehat{c}=(\widehat{c}^{\alpha\beta}) $ are given by
\begin{align}
\widehat{a}_{ij}^{\alpha\beta}=\int_{Y}[a_{ij}^{\alpha\beta}(y)+a_{ik}^{\alpha\gamma}(y)\pa_k\chi_{j}^{\gamma\beta}(y)]dy,\quad \widehat{V}_{i}^{\alpha\beta}=\int_{Y}[V_{i}^{\alpha\beta}(y)+a_{ij}^{\alpha\gamma}(y)\pa_j\chi_{0}^{\gamma\beta}(y)]dy,\nonumber\\
\widehat{B}_{i}^{\alpha\beta}=\int_{Y}[B_{i}^{\alpha\beta}(y)+B_{j}^{\alpha\gamma}(y)\pa_j\chi_{i}^{\gamma\beta}(y)]dy,\quad \widehat{c}^{\alpha\beta}=\int_{Y}[c^{\alpha\beta}(y)+c_{i}^{\alpha\gamma}(y)\pa_i\chi_{0}^{\gamma\beta}(y)]dy,\label{Homogenization coefficients}
\end{align}
where $ \chi_{k} $ are the corresponding correctors for homogenization problems, defined as $ \chi_{k}=(\chi_k^{\alpha\beta}) $ with $ 0\leq k\leq d $, satisfying
\begin{align}
\left\{\begin{matrix}
L_1(\chi_0)=\operatorname{div}(V)&\text{in}&\mathbb{R}^d, \\
\chi_0\in W_{\operatorname{per}}^{1,2}(Y;\mathbb{R}^{m^2})&\text{and}& \int_{Y}\chi_0dy=0,
\end{matrix}\right.\label{Corrector 1}
\end{align}
and
\begin{align}
\left\{\begin{array}{lcl}
L_1(\chi_k^{\beta}+P_k^{\beta})=0  & \text{ in } &\mathbb{R}^d, \\
\chi_k\in W_{\operatorname{per}}^{1,2}(Y;\mathbb{R}^{m})  & \text{ and } & \int_{Y}\chi_k^{\beta}dy=0,\ 1\leq k\leq d.
\end{array}\right.\label{Corrector 2}
\end{align}
 Here, $ Y=[0,1)^d=\mathbb{R}^d/\mathbb{Z}^d $, $ P_k^{\beta}=x_ke^{\beta} $, and $ W_{\operatorname{per}}^{1,2}(Y;\mathbb{R}^{m}) $ denotes the closure of $ C_{\operatorname{per}}^{\infty}(\om;\mathbb{R}^m) $ in $ W^{1,2}(\om;\mathbb{R}^m) $. Note that $ C_{\operatorname{per}}^{\infty}(\om;\mathbb{R}^m) $ is the subset of $ C^{\infty}(Y;\mathbb{R}^m) $, which collects all $ Y $-periodic vector-valued functions.

Next, we will give some simple conclusions. These conclusions are basic and important for subsequent proofs.
\begin{lem}[\cite{Xu1}, Lemma 2.4]\label{el}
Let $ \om $ be a Lipschitz domain in $ \mathbb{R}^d $, $ d\geq 2 $. Suppose that $ A $ satisfies the ellipticity condition $ \eqref{Ellipticity} $, and other coefficients of $ \mathcal{L}_{\va} $, $ V,B,c $ satisfy $ \eqref{Boundedness} $. Then we have the following properties: for any $ u,v\in W_0^{1,2}(\om;\mathbb{R}^m) $,
\begin{align}
|\langle \mathcal{L}_{\va}(u),v\rangle|\leq C\left\|u\right\|_{W_0^{1,2}(\om)}\left\|v\right\|_{W_0^{1,2}(\om)},\quad c_0\left\|u\right\|_{W_0^{1,2}(\om)}^2\leq\langle \mathcal{L}_{\va}(u),u\rangle, \label{Lax-Milgram bound}
\end{align}
whenever $ \lambda\geq \lambda_0 $, where $ \lambda_0=\lambda_0(\mu,\kappa,m,d) $ is sufficiently large.  Note that $ C $ depends only on $ \mu,\kappa,\lambda,m,d,\om $, and $ c_0 $ depends only on $ \mu,\kappa,m,d,\lambda,\om $.
\end{lem}
Based on the lemma above and by using the Lax-Milgram theorem, we can obtain the following theorem.
\begin{thm}[\cite{Xu1}, Lemma 2.5]\label{energy}
Under the conditions of Lemma \ref{el}, suppose that $ F\in H^{-1}(\om;\mathbb{R}^m) $ and $ g\in W^{\frac{1}{2},2}(\om;\mathbb{R}^m) $. Then the Dirichlet boundary value problem $ \mathcal{L}_{\va}(u_{\va})=F $ in $ \om $ and $ u_{\va}=g $ on $ \pa\om $, has a unique weak solution $ u_{\va}\in W^{1,2}(\om;\mathbb{R}^m) $, whenever $ \lambda\geq\lambda_0 $, and the solution satisfies the uniform estimate
\begin{align}
\left\|u_{\va}\right\|_{W^{1,2}(\om)}\leq C\left\{\left\|F\right\|_{H^{-1}(\om)}+\left\|g\right\|_{W^{\frac{1}{2},2}(\pa\om)}\right\}, \label{Energy estimates}
\end{align}
where $ C $ depends only on $ \mu,\kappa,m,d $ and $ \om $. Moreover, with one more the periodicity condition $ \eqref{Periodicity} $ on the coefficients of $ \mathcal{L}_{\va} $, we then have $ u_{\va}\rightharpoonup  u $ weakly in $ W^{1,2}(\om;\mathbb{R}^m) $ and strongly in $ L^2(\om;\mathbb{R}^m) $ as $ \va\to 0 $. where $ u $ satisfies $ \mathcal{L}_{0}(u_0)=F $ in $ \om $, and $ u_0=g $ on $ \pa\om $.
\end{thm}

To simplify the notations, we define that for $ x\in\om $ and $ 0<r<\operatorname{diam}(\om) $,
\begin{align}
\om(x,r)=\om\cap B(x,r)\quad\text{and}\quad\Delta(x,r)=\pa\om\cap B(x,r).\label{om x r}
\end{align}
Also, for $ u\in L^p(E) $ with $ E $ being measurable and $ 1\leq p<\infty $,
\begin{align}
\left\|u\right\|_{L_{\avg}^p(E)}=\left(\dashint_{E}|u(x)|^pdx\right)^{\frac{1}{p}}=\left(\frac{1}{|E|}\int_{E}|u(x)|^pdx\right)^{\frac{1}{p}}.\label{average int}
\end{align}

\begin{lem}[Caccioppoli's inequality]\label{Caccioppoli} Assume that $ A $ satisfies $ \eqref{Ellipticity} $, other coefficients of $ \mathcal{L}_{\va} $, $ V,B,c $ satisfy $ \eqref{Boundedness} $, $ \om $ is a $ C^1 $ bounded domain in $ \mathbb{R}^d $ $ (d\geq 2) $, $ x_0\in\om $ and $ 0<r<\operatorname{diam}(\om) $. If $ \pa\om\cap B(x_0,2r)\neq\emptyset $, assume that $ u_{\va}\in W^{1,2}(\om(x_0,2r);\mathbb{R}^m) $ is a weak solution to
\begin{align}
\left\{\begin{matrix}
\mathcal{L}_{\va}(u_{\va})=\operatorname{div}(f)+F&\text{in}&\om(x_0,2r),\\
u_{\va}=0&\text{on}&\Delta(x_0,2r),
\end{matrix}\right.\nonumber
\end{align}
with $ f\in L^2(\om(x_0,2r);\mathbb{R}^{m\times d}) $, $ F\in L^q(\om(x_0,2r);\mathbb{R}^m) $, $ q=\frac{2d}{d+2} $ if $ d\geq 3 $ and $ q>1 $ if $ d=2 $. If $ \pa\om\cap B(x_0,2r)=\emptyset $, assume that $ u_{\va}\in W^{1,2}(B(x_0,2r);\mathbb{R}^m) $ is a weak solution to
\begin{align}
\mathcal{L}_{\va}(u_{\va})=\operatorname{div}(f)+F\quad\text{in}\quad B(x_0,2r),\nonumber
\end{align}
with the same data $ f $, $ F $ and $ q $. Then there exists $ \lambda_0=\lambda_0(\mu,d,m,\kappa) $, such that for $ \lambda\geq\lambda_0 $, we have the uniform estimate
\begin{align}
\left(\dashint_{\om_r}|\nabla u_{\va}|^2dx\right)^{\frac{1}{2}}\leq\frac{C}{r}\left(\dashint_{\om_{2r}}|u_{\va}|^2dx\right)^{\frac{1}{2}}+C\left(\dashint_{\om_{2r}}|f|^2dx\right)^{\frac{1}{2}}+Cr\left(\dashint_{\om_{2r}}|F|^qdx\right)^{\frac{1}{q}},\label{Caccioppoli inequality}
\end{align}
where $ \om_r=\om(x_0,r) $ and $ C $ depends only on $ \mu,\kappa,\lambda,\mu,m,d,\om $.
\end{lem}
\begin{proof}
For the case $ d\geq 3 $, one can find the proof in \cite{Xu1}. For $ d=2 $, the proof is almost the same. We only need to adjust the proof for $ d\geq 3 $ by changing the Sobolev embedding theorem $ W^{1,2}(\mathbb{R}^d)\subset L^{\frac{2d}{d-2}}(\mathbb{R}^d) $ when $ d\geq 3 $ to $ W^{1,2}(\mathbb{R}^2)\subset L^q(\mathbb{R}^2) $ with $ q\geq 2 $ when $ d=2 $. For the sack of completeness, we give the proof for $ d=2 $ as follows. By translation and rescaling, we may assume that $ r=1 $ and $ x_0=0 $. Let $ \varphi\in C_0^{\infty}(B(0,2)) $ be a cut-off function satisfying $ \varphi\equiv 1 $ in $ B(0,1) $, $ \varphi\equiv 0 $ in $ (B(0,\frac{3}{2}))^c $, and $ |\nabla\varphi|\leq C $. Then, by choosing the test function as $ \phi=\varphi^2u_{\va} $, we can obtain
\begin{align}
& \int_{\om(0,2)}[A(x/\va)\nabla u_{\va}+V(x/\va)u_{\va}]\nabla u_{\va}\varphi^2 dx+2\int_{\om(0,2)}[A(x/\va)\nabla u_{\va}+V(x/\va)u_{\va}]\nabla \varphi u_{\va}\varphi dx\nonumber\\
&\quad\quad\quad+\int_{\om(0,2)}B(x/\va)\nabla u_{\va}u_{\va}\varphi^2 dx+\int_{\om(0,2)}c(x/\va)u_{\va}u_{\va}\varphi^2+\lambda|u_{\va}|^2\varphi^2dx\nonumber\\
&\quad\quad =\int_{\om(0,2)}Fu_{\va}\varphi^2dx-\int_{\om(0,2)}f\nabla u_{\va}\varphi^2 dx-2\int_{\om(0,2)}f\nabla\varphi u_{\va}\varphi dx.\nonumber
\end{align}
From the ellipticity condition $ \eqref{Ellipticity} $ and Young's inequality, we have
\begin{align}
& \frac{\mu}{4}\int_{\om(0,2)}\varphi^2|\nabla u_{\va}|^2dx
 \leq (C'-\lambda)\int_{\om(0,2)}\varphi^2|u_{\va}|^2dx\nonumber\\
&\quad\quad+C\int_{\om(0,2)}|\nabla\varphi|^2|u_{\va}|^2dx+C\int_{\om(0,2)}\varphi^2|f|^2dx+\int_{\om(0,2)}\varphi^2|F||u_{\va}|dx,\nonumber
\end{align}
where $ C'=C'(\mu,\kappa,m) $ is a constant. We can derive $ \lambda_1 $, such that for any $ \lambda\geq\lambda_0 $, $ C'-\lambda<0 $. By the Sobolev embedding theorem, we can choose $ 1<q\leq 2 $ (if we prove that the results are true for $ 1<q\leq 2 $, the case for $ q>2 $ is trivial)
\begin{align}
\int_{\om(0,2)}\varphi^2|F||u_{\va}|dx&\leq \left(\int_{\om(0,2)}(\varphi|u_{\va}|)^{\frac{q}{q-1}}dx\right)^{\frac{q-1}{q}}\left(\int_{\om(0,2)}(\varphi|F|)^qdx\right)^{\frac{1}{q}}\nonumber\\
&\leq  C\left(\int_{\om(0,2)}|\nabla(\varphi u_{\va})|^2dx\right)^{\frac{1}{2}}\left(\int_{\om(0,2)}|(\varphi |F|)|^qdx\right)^{\frac{1}{q}}\nonumber\\
&\leq \frac{\mu}{8}\int_{\om(0,2)}|\nabla u_{\va}|^2\varphi^2dx+\frac{\mu}{8}\int_{\om(0,2)}|u_{\va}|^2|\nabla\varphi|^2dx+C\left(\int_{\om(0,2)}|\varphi F|^qdx\right)^{\frac{2}{q}}.\nonumber
\end{align}
Here we use the Sobolev embedding theorem $ W_0^{1,2}(\om(0,2);\mathbb{R}^m)\subset L^{\frac{q}{q-1}}(\om(0,2);\mathbb{R}^m) $ with $ 1<q\leq 2 $. By the definition of $ \varphi $, $ \eqref{Caccioppoli inequality} $ is true.
\end{proof}
\begin{rem}
More precisely, we can obtain the inequality from $ \om(x_0,tr) $ to $ \om(x_0,sr) $, where $ 0<t<s<1 $. In the proof of this, we can choose the cut-off function $ \varphi $ satisfying $ \varphi\equiv 1 $ in $ B(x_0,tr) $, $ \varphi\equiv 0 $ in $ (B(x_0,sr))^c $ and $ |\nabla\varphi|\leq \frac{C}{(s-t)r} $. Using almost the same arguments, we can obtain
\begin{align}
\left(\dashint_{\om_{tr}}|\nabla u_{\va}|^2dx\right)^{\frac{1}{2}}\leq \frac{C}{(s-t)r}\left(\dashint_{\om_{sr}}|u_{\va}|^2dx\right)^{\frac{1}{2}}+C\left(\dashint_{\om_{sr}}|f|^2dx\right)^{\frac{1}{2}}+Cr\left(\dashint_{\om_{sr}}|F|^qdx\right)^{\frac{1}{q}},\label{28gj}
\end{align}
where $ \om_{tr}=\om(x_0,tr) $ and $ C $ depends only on $ \mu,\kappa,\lambda,m,d,\om $.
\end{rem}
\begin{rem}
Under the same conditions of Lemma \ref{Caccioppoli}, assume that $ J\in\mathbb{R}^m $ is an arbitrary vector and $ v_{\va}=u_{\va}-J $, the equation corresponding to $ v_{\va} $ is stated as follows
\begin{align}
\mathcal{L}_{\va}(v_{\va})=\operatorname{div}(f+V(x/\va)J)+F-c(x/\va)J-\lambda J.\nonumber
\end{align}
Then, if $ \pa\om\cap B(x,2r)=\emptyset $, for $ v_{\va} $, we can use the Caccioppoli's inequality $ \eqref{Caccioppoli inequality} $ to obtain
\begin{align}
\left(\dashint_{B_{r}}|\nabla u_{\va}|^2\right)^{\frac{1}{2}}&\leq \frac{C}{r}\left(\dashint_{B_{2r}}|u_{\va}-J|^2\right)^{\frac{1}{2}}+C\left(\dashint_{B_{2r}}|f|^2\right)^{\frac{1}{2}}+Cr\left(\dashint_{B_{2r}}|F|^q\right)^{\frac{1}{q}}+C(|J|+|J|r),\label{Caccioppoli rem}
\end{align}
where $ B_r=B(x_0,r) $ and $ C $ depends only on $ \mu,\kappa,\lambda,m,p,q,d $.
\end{rem}

\begin{lem}[\cite{Xu1}, Remark 3.4]\label{interpolation}
Let $ \om $ be a $ C^1 $ bounded domain in $ \mathbb{R}^d $$ (d\geq 2) $, and  $ u\in W^{1,p}(\om;\mathbb{R}^m) $ with $ 2\leq p<\infty $. Then, for all $ \delta>0 $, there exists $ C_{\delta}>0 $ depending only on $ \delta,m,p,d,\om $ such that
\begin{align}
\left\|u\right\|_{L^p(\om)}\leq\delta\left\|\nabla u\right\|_{L^p(\om)}+C_{\delta}\left\|u\right\|_{L^2(\om)}.\label{Little lemma 1}
\end{align}
Moreover, when $ 1\leq p<\infty $ and $ q=\frac{pd}{p+d} $, we have
\begin{align}
\left\|u\right\|_{L^p(\om)}\leq C\left\{\left\|\nabla u\right\|_{L^q(\om)}+\left\|u\right\|_{L^2(\om)}\right\},\label{Little lemma 2}
\end{align}
where $ C $ depends only on $ m,p,d,\om $.
\end{lem}

\begin{rem}
Obviously, the second constant $ C $ in the above Lemma \ref{interpolation} is related to the diameter of $ \om $. Here, for the convenience of later calculation, we need to quantify this. For simplicity, let us assume $ \om=B(0,1) $. Then we can get
\begin{align}
\left\|u\right\|_{L^p(B(0,1))}\leq C_1\left\|\nabla u\right\|_{L^q(B(0,1))}+C_2\left\|u\right\|_{L^2(B(0,1))}.\nonumber
\end{align}
By choosing $ v(x)=u(rx) $ and using the above results, we have
\begin{align}
\left\|u(rx)\right\|_{L^p(B(0,1))}\leq C_1\left\|r\nabla u(rx)\right\|_{L^q(B(0,1))}+C_2\left\|u(rx)\right\|_{L^2(B(0,1))}.\nonumber
\end{align}
After changing the variables, it follows that
\begin{align}
r^{-\frac{d}{p}}\left\|u(x)\right\|_{L^p(B(0,r))}\leq C_1 r^{1-\frac{d}{q}}\left\|\nabla u(x)\right\|_{L^q(B(0,r))}+C_2 r^{-\frac{d}{2}}\left\|u(x)\right\|_{L^2(B(0,r))}.\nonumber
\end{align}
This implies that $ C_1(r)=C_1 $ and $ C_2(r)=C_2r^{\frac{d}{p}-\frac{d}{2}} $. Rewrite the inequality above, we have
\begin{align}
\left\|u\right\|_{L^p(B(0,r))}\leq C_1\left\|\nabla u\right\|_{L^q(B(0,r))}+C_2 r^{\frac{d}{p}-\frac{d}{2}}\left\|u\right\|_{L^2(B(0,r))},\label{interp}
\end{align}
where $ C_1 $, $ C_2 $ depends only on $ m,d,p $.
\end{rem}
\section{$ W^{1,p} $ estimates for $ \mathcal{L}_{\va} $ with $ d=2 $}
\begin{thm}\label{W1p L}
Assume that $ A\in\operatorname{VMO}(\mathbb{R}^2) $ satisfies $ \eqref{Ellipticity} $, $ \eqref{Periodicity} $, $ \om $ is a $ C^{1,\eta} $ $ (0<\eta<1) $ bounded domain in $ \mathbb{R}^2 $, and $ 1<p<\infty $. Let $ f\in L^p(\om;\mathbb{R}^{m\times 2}) $, $ F\in L^q(\om;\mathbb{R}^m) $, where $ q=\frac{2p}{p+2} $ if $ p>2 $, $ 1<q<\infty $ if $ p=2 $ and $ q=1 $ if $ 1<p<2 $. Then there exists a unique weak solution $ u_{\va}\in W^{1,p}(\om;\mathbb{R}^m) $ of the Dirichlet problem
\begin{align}
\left\{\begin{matrix}
L_{\va}(u_{\va})=\operatorname{div}(f)+F&\text{in}&\om,\\
u_{\va}=0&\text{on}&\pa\om,
\end{matrix}\right.\nonumber
\end{align}
satisfying the uniform estimate
\begin{align}
\left\|\nabla u_{\va}\right\|_{L^{p}(\om)}\leq C\left\{\left\|F\right\|_{L^q(\om)}+\left\|f\right\|_{L^p(\om)}\right\},\label{W^{1,p} estimates 4}
\end{align}
where $ C $ depends only on $ \mu,\omega(t),m,d,p,q,\om $.
\end{thm}
\begin{proof}
See Theorem 5.3.1 in \cite{Shen1}.
\end{proof}

Using Theorem \ref{W1p L} and some iteration arguments, we can prove the $ W^{1,p} $ estimates for the operator $ \mathcal{L}_{\va} $.

\begin{lem}\label{ww2}
Assume that $ A\in\operatorname{VMO}(\mathbb{R}^2) $ satisfies $ \eqref{Ellipticity} $, $ \eqref{Periodicity} $, and other coefficients of $ \mathcal{L}_{\va} $, $ V,B,c $ satisfy $ \eqref{Boundedness} $, $ \om $ is a $ C^{1,\eta} $ $ (0<\eta<1) $ bounded domain in $ \mathbb{R}^2 $, and $ 1<p<\infty $. Let $ f\in L^p(\om;\mathbb{R}^{m\times 2}) $. Then, for $ \lambda\geq\lambda_0 $, there exists a unique weak solution $ u_{\va}\in W^{1,p}(\om;\mathbb{R}^m) $ of the Dirichlet problem
\begin{align}
\left\{\begin{matrix}
\mathcal{L}_{\va}(u_{\va})=\operatorname{div}(f)&\text{in}&\om,\\
u_{\va}=0&\text{on}&\pa\om,
\end{matrix}\right.\nonumber
\end{align}
satisfying the uniform estimate
\begin{align}
\left\|\nabla u_{\va}\right\|_{L^{p}(\om)}\leq C\left\|f\right\|_{L^p(\om)},\label{ww3}
\end{align}
where $ C $ depends only on $ \mu,\omega(t),\kappa,\lambda,m,p $ and $ \om $.
\end{lem}
\begin{proof}
If $ p=2 $, we can derive the estimate $ \eqref{ww3} $ by the energy inequality $ \eqref{Energy estimates} $. It is easy to obtain a unique weak solution $ u_{\va}\in W_0^{1,2}(\om;\mathbb{R}^m) $ such that $ \left\|\nabla u_{\va}\right\|_{L^{2}(\om)}\leq C\left\|f\right\|_{L^2(\om)} $. For $ p>2 $, the uniqueness and the existence are trivial. Then, we only need to show the uniform estimate $ \eqref{ww3} $. Firstly, $ u_{\va} $ is the weak solution of the Dirichlet problem
\begin{align}
\left\{\begin{matrix}
L_{\va}(u_{\va})=\operatorname{div}(f+V_{\va}u_{\va})-B_{\va}\nabla u_{\va}-(c_{\va}+\lambda)u_{\va}&\text{in}&\om,\\
u_{\va}=0&\text{on}&\pa\om.
\end{matrix}\right.\nonumber
\end{align}
From $ \eqref{W^{1,p} estimates 4} $, we have
\begin{align}
\left\|\nabla u_{\va}\right\|_{L^p(\om)}\leq C\left\{\left\|f\right\|_{L^p(\om)}+\left\|u_{\va}\right\|_{L^p(\om)}+\left\|u_{\va}\right\|_{L^{\frac{2p}{p+2}}(\om)}+\left\|\nabla u_{\va}\right\|_{L^{\frac{2p}{p+2}}(\om)}\right\}.\nonumber
\end{align}
Then, in view of $ \eqref{Energy estimates} $, H\"{o}lder's inequality and $ \eqref{Little lemma 1} $, it follows that
\begin{align}
\left\|\nabla u_{\va}\right\|_{L^p(\om)}\leq C\left\{\left\|u_{\va}\right\|_{L^2(\om)}+\left\|\nabla u_{\va}\right\|_{L^2(\om)}+\left\|f\right\|_{L^p(\om)}\right\}\leq C\left\|f\right\|_{L^p(\om)}.\nonumber
\end{align}
If $ 1<p<2 $, we can derive $ \eqref{ww3} $ by the duality arguments.
\end{proof}

\begin{lem}\label{ww5}
Assume that $ A\in\operatorname{VMO}(\mathbb{R}^2) $ satisfies $ \eqref{Ellipticity} $, $ \eqref{Periodicity} $, and other coefficients of $ \mathcal{L}_{\va} $, $ V,B,c $ satisfy $ \eqref{Boundedness} $, $ \om $ is a $ C^{1,\eta} $ $ (0<\eta<1) $ bounded domain in $ \mathbb{R}^2 $, and $ 1<p<\infty $. Let $ F\in L^q(\om;\mathbb{R}^m) $, where $ q=\frac{2p}{p+2} $ if $ p>2 $, $ 1<q<\infty $ if $ p=2 $, $ q=1 $ if $ 1<p<2 $. Then, for $ \lambda\geq\lambda_0 $, there exists a unique weak solution $ u_{\va}\in W^{1,p}(\om;\mathbb{R}^m) $ of the Dirichlet problem \begin{align}
\left\{\begin{matrix}
\mathcal{L}_{\va}(u_{\va})=F&\text{in}&\om,\\
u_{\va}=0&\text{on}&\pa\om,
\end{matrix}\right.\label{fd}
\end{align}
satisfying the uniform estimate
\begin{align}
\left\|\nabla u_{\va}\right\|_{L^{p}(\om)}\leq C\left\|F\right\|_{L^q(\om)},\label{ww4}
\end{align}
where $ C $ depends only on $ \mu,\omega(t),\kappa,\lambda,m,p $ and $ \om $.
\end{lem}
\begin{proof}
After getting the estimate $ \eqref{ww4} $, the uniqueness of $ \eqref{fd} $ is trivial to establish. Therefore, we only need to show the existence of the solution. In fact, given $ 1<p<\infty $, and $ q $ defined above, we can first assume that $ F\in L^q(\om;\mathbb{R}^m)\cap L^2(\om;\mathbb{R}^m) $ and use Theorem \ref{energy} to establish the existence of $ \eqref{fd} $. Then, by using the standard density arguments, we can prove that the result is valid for all $ F\in L^q(\om;\mathbb{R}^m) $ with the help of the a priori estimate. Now, we will prove the estimate $ \eqref{ww4} $, where, we use the duality arguments and Lemma \ref{ww2}. For any $ f\in C_0^1(\om;\mathbb{R}^{m\times 2}) $, there exists a unique $ v_{\va}\in W_0^{1,2}(\om;\mathbb{R}^m) $ such that $ \mathcal{L}_{\va}^*(v_{\va})=\operatorname{div}(f) $ in $ \om $, and $ v_{\va}=0 $ on $ \pa\om $. According to $ \eqref{ww3} $, we have $ \left\|\nabla v_{\va}\right\|_{L^{p'}(\om)}\leq C\left\|f\right\|_{L^{p'}(\om)} $, where $ p'=\frac{p}{p-1} $ denoted as the conjugate number of $ p $. Then, by using the definition of $ u_{\va} $ and $ v_{\va} $, we can obtain
\begin{align}
\int_{\om}\nabla u_{\va}fdx=-\int_{\om}\mathcal{L}_{\va}(u_{\va})v_{\va}dx=-\int_{\om}Fv_{\va}dx.\nonumber
\end{align}
(1) If $ 2<p<\infty $, we choose $ q=\frac{2p}{p+2} $. Using H\"{o}lder's inequality, Poincar\'{e}-Sobolev inequality and Sobolev embedding theorem, we obtain
\begin{align}
\left|\int_{\om}\nabla u_{\va}fdx\right|&\leq  \left\|F\right\|_{L^{\frac{2p}{p+2}}(\om)}\left\|v_{\va}\right\|_{L^{\frac{2p}{p-2}}(\om)}\leq C\left\|F\right\|_{L^{\frac{2p}{p+2}}(\om)}\left\|\nabla v_{\va}\right\|_{L^{\frac{p}{p-1}}(\om)}\leq C \left\|F\right\|_{L^{\frac{2p}{p+2}}(\om)}\left\|f\right\|_{L^{\frac{p}{p-1}}(\om)}.\nonumber
\end{align}
Then,   using the duality methods, we can obtain that $ \left\|\nabla u_{\va}\right\|_{L^{p}(\om)}\leq C\left\|F\right\|_{L^{\frac{2p}{p+2}}(\om)} $.\\
(2) If $ p=2 $, choose $ 1<q<\infty $, then, we have
\begin{align}
\left|\int_{\om}\nabla u_{\va}fdx\right|&\leq  \left\|F\right\|_{L^{q}(\om)}\left\|v_{\va}\right\|_{L^{\frac{q}{q-1}}(\om)}\leq C\left\|F\right\|_{L^{q}(\om)}\left\|\nabla v_{\va}\right\|_{L^{2}(\om)}\leq C \left\|F\right\|_{L^{q}(\om)}\left\|f\right\|_{L^{2}(\om)}.\nonumber
\end{align}
Then, we can obtain that $ \left\|\nabla u_{\va}\right\|_{L^{2}(\om)}\leq C\left\|F\right\|_{L^{q}(\om)} $.\\
(3) If $ 1<p<2 $, choose $ q=1 $, then, we have
\begin{align}
\left|\int_{\om}\nabla u_{\va}fdx\right|&\leq  \left\|F\right\|_{L^1(\om)}\left\|v_{\va}\right\|_{L^{\infty}(\om)}\leq C\left\|F\right\|_{L^1(\om)}\left\|\nabla v_{\va}\right\|_{L^{\frac{p}{p-1}}(\om)}\leq C \left\|F\right\|_{L^{1}(\om)}\left\|f\right\|_{L^{\frac{p}{p-1}}(\om)}.\nonumber
\end{align}
  Because of   cases (1), (2) and (3), we can complete the proof.
\end{proof}

\begin{rem}
Furthermore, when $ p>2 $ and $ F\in \dot{W}^{-1,p'}(\om) $ $($where $ \dot{W}^{-1,p'}(\om) $ denotes the dual space for the homogeneous Sobolev space $ \dot{W}_0^{1,p}(\om) $$)$, we have
\begin{align}
\left\|u_{\va}\right\|_{L^{q'}(\om)}\leq C\left\|F\right\|_{\dot{W}^{-1,p'}(\om)}. \label{Dual estimates}
\end{align}
The conclusion is proved by using the dual method. For all $ f\in C_0^1(\om;\mathbb{R}^m) $ we can choose $ v_{\va} $ such that $ \mathcal{L}_{\va}^*(v_{\va})=f $ in $ \om $, and $ v_{\va}=0 $ on $ \pa\om $, and
\begin{align}
\int_{\om}u_{\va}fdx=\int_{\om}u_{\va}\mathcal{L}_{\va}^*(v_{\va})=\int_{\om}\mathcal{L}_{\va}(u_{\va})v_{\va}dx=\langle F,v_{\va}\rangle. \nonumber
\end{align}
This implies that
\begin{align}
\left|\int_{\om}u_{\va}fdx\right|=\left|\langle F,v_{\va}\rangle\right|\leq \left\|F\right\|_{\dot{W}^{-1,p'}(\om)}\left\|v_{\va}\right\|_{\dot{W}_0^{1,p}(\om)}\leq\left\|F\right\|_{\dot{W}^{-1,p'}(\om)}\left\|f\right\|_{L^q(\om)}.\nonumber
\end{align}
According to the duality property, we have $ \left\|u_{\va}\right\|_{L^{q'}(\om)}\leq C\left\|F\right\|_{\dot{W}^{-1,p'}(\om)} $.
\end{rem}

\begin{proof}[Proof of Theorem \ref{W1p}]
In the case of $ g=0 $, we write $ v_{\va}=u_{\va,1}+u_{\va,2} $, where $ u_{\va,1} $ and $ u_{\va,2} $ are the solution in Lemma \ref{ww2} and \ref{ww5}, respectively. Then we have
\begin{align}
\left\|\nabla v_{\va}\right\|_{L^p(\om)}\leq \left\|\nabla u_{\va,1}\right\|_{L^p(\om)}+\left\|\nabla u_{\va,1}\right\|_{L^p(\om)}\leq C\left\{\left\|f\right\|_{L^p(\om)}+\left\|F\right\|_{L^q(\om)}\right\}.\label{ww6}
\end{align}

For $ g\neq 0 $, consider the homogeneous Dirichlet problem $ \mathcal{L}_{\va}(w_{\va})=0 $ in $ \om $ and $ w_{\va}=g $ on $ \pa\om $, where $ g\in B^{1-\frac{1}{p},p}(\pa\om;\mathbb{R}^m) $. By the properties of boundary Besov space, there exists $ G\in W^{1,p}(\om;\mathbb{R}^m) $ such that $ G=g $ on $ \pa\om $ and $ \left\|G\right\|_{W^{1,p}(\om)}\leq C\left\|g\right\|_{B^{1-\frac{1}{p},p}(\pa\om)} $. For $ h_{\va}=w_{\va}-G $, we have
\begin{align}
\left\{\begin{matrix}
\mathcal{L}_{\va}(h_{\va})=\operatorname{div}(A(x/\va)\nabla G+V(x/\va)G)-B(x/\va)\nabla G-(c(x/\va)+\lambda)G&\text{in}&\om,\\
h_{\va}=0&\text{on}&\pa\om.
\end{matrix}\right.\nonumber
\end{align}
Recall the case of $ g=0 $, in which there exists the unique weak solution $ h_{\va}\in W_0^{1,p}(\om;\mathbb{R}^m) $, satisfying the uniform estimate
\begin{align}
\left\|\nabla h_{\va}\right\|_{L^p(\om)}\leq C\left\|G\right\|_{W^{1,p}(\om)}+C\left\|\nabla G\right\|_{L^q(\om)}+C\left\|G\right\|_{L^q(\om)}.\nonumber
\end{align}
For $ p>2 $, choose $ q=\frac{2p}{p+2}<p $, for $ p=2 $, choose $ q<2=p $ and for $ p<2 $, choose $ q=1<p $, then
\begin{align}
\left\|\nabla h_{\va}\right\|_{L^p(\om)}&\leq C\left\|G\right\|_{W^{1,p}(\om)}+C\left\|\nabla G\right\|_{L^q(\om)}+C\left\|G\right\|_{L^q(\om)}
\leq C\left\|G\right\|_{W^{1,p}(\om)}\leq C\left\|g\right\|_{B^{1-\frac{1}{p},p}(\pa\om)},\nonumber
\end{align}
for any $ 1<p<\infty $. This implies
\begin{align}
\left\|\nabla w_{\va}\right\|_{L^p(\om)}\leq\left\|\nabla h_{\va}\right\|_{L^p(\om)}+C\left\|\nabla G\right\|_{L^p(\om)}\leq C\left\|g\right\|_{B^{1-\frac{1}{p},p}(\pa\om)}.\label{ww7}
\end{align}
Finally, let $ u_{\va}=v_{\va}+w_{\va} $. Combining $ \eqref{ww6} $ and $ \eqref{ww7} $, we can complete the proof.
\end{proof}

\begin{thm}[The Localization of $ W^{1,p} $ estimates for $ \mathcal{L}_{\va} $ with $ d=2 $]Let $ 2\leq p<\infty $. Assume that $ A\in\operatorname{VMO}(\mathbb{R}^2) $ satisfies $ \eqref{Ellipticity} $, $ \eqref{Periodicity} $, other coefficients of $ \mathcal{L}_{\va} $, $ V,B,c $ satisfy $ \eqref{Boundedness} $, $ \lambda\geq\lambda_0 $, $ \om $ is a $ C^{1,\eta} $ $ (0<\eta<1) $ bounded domain in $ \mathbb{R}^2 $, $ x_0\in\om $ and $ 0<r<\operatorname{diam}(\om) $. If $ \pa\om\cap B(x_0,2r)\neq\emptyset $, assume that $ u_{\va}\in W^{1,2}(\om(x_0,2r);\mathbb{R}^m) $ is the weak solution to
\begin{align}
\left\{\begin{matrix}
\mathcal{L}_{\va}(u_{\va})=\operatorname{div}(f)+F&\text{in}&\om(x_0,2r),\\
u_{\va}=0&\text{on}&\Delta(x_0,2r),
\end{matrix}\right.\nonumber
\end{align}
with $ f\in L^p(\om(x_0,2r);\mathbb{R}^{m\times 2}) $, $ F\in L^q(\om(x_0,2r);\mathbb{R}^m) $, $ q=\frac{2p}{p+2} $ if $ 2<p<\infty $ and $ q>1 $ if $ p=2 $. If $ \pa\om\cap B(x_0,2r)=\emptyset  $, assume that $ u_{\va}\in W^{1,2}(B(x_0,2r);\mathbb{R}^m) $ is a weak solution to
\begin{align}
\mathcal{L}_{\va}(u_{\va})=\operatorname{div}(f)+F\quad\text{in}\quad B(x_0,2r),\nonumber
\end{align}
with the same data $ f $, $ F $, $ p $ and $ q $. Then $ \nabla u_{\va}\in L^p(\om(x_0,r);\mathbb{R}^m) $ and
\begin{align}
\left\|\nabla u_{\va}\right\|_{L_{\avg}^p(\om_{tr})}\leq\frac{C}{(s-t)^2tr}\left\|u_{\va}\right\|_{L_{\avg}^2(\om_{sr})}+\frac{C}{(s-t)t^{\frac{2}{p}}}\left\{\left\|f\right\|_{L_{\avg}^p(\om_{sr})}+r\left\|F\right\|_{L_{\avg}^q(\om_{sr})}\right\},\label{W^{1,p} interior estimates}
\end{align}
where $ 0<t<s<1 $, $ \om_{tr}=\om(x_0,tr) $ and $ C $ depends only on $ \mu,\omega(t),\kappa,\lambda,m,p,q,\om $.
\end{thm}
\begin{proof}
If $ p=2 $, we can derive this inequality by Caccioppoli's inequality $ \eqref{Caccioppoli inequality} $. If $ p>2 $, by using rescaling and translation, we can assume that $ x_0=0 $ and $ r=1 $. We can also assume that $ \pa\om\cap B(x_0,2r)\neq\emptyset $, since the other case is almost the same. For $ 0<t<s<1 $, we see that $ 0<t<\frac{t+s}{2}<s<1 $. We can choose $ \varphi\in C_0^{\infty}(B(0,\frac{t+s}{2})) $ as a cut-off function such that $ \varphi\equiv 1 $ in $ B(0,t) $, $ \varphi\equiv 0 $ in $ (B(0,\frac{t+s}{2}))^c $ and $ |\nabla \varphi|\leq \frac{C}{(s-t)} $. Then by letting $ w_{\va}=\varphi u_{\va} $, we have
\begin{align}
\left\{\begin{matrix}
-\operatorname{div}(A_{\va}\nabla w_{\va})=\operatorname{div}(f\varphi)-f\nabla\varphi+F\varphi+\widetilde{F}&\text{in}&\om(0,2),\\
w_{\va}=0&\text{on}&\pa[\om(0,2)],
\end{matrix}\right.\nonumber
\end{align}
where $ A_{\va}=A(x/\va) $, $ V_{\va}=V(x/\va) $, $ B_{\va}=B(x/\va) $, $ c_{\va}=c(x/\va) $, and
\begin{align}
\widetilde{F}^{\alpha}=\operatorname{div}(V_{\va}^{\alpha\beta}w_{\va}^{\beta}-A_{\va}^{\alpha\beta}\nabla\varphi u_{\va}^{\beta})-A_{\va}^{\alpha\beta}\nabla\varphi\nabla u_{\va}^{\beta}-V_{\va}^{\alpha\beta}\nabla\varphi u_{\va}^{\beta}-B_{\va}^{\alpha\beta}\nabla u_{\va}^{\beta}\varphi-c_{\va}^{\alpha\beta}w_{\va}^{\beta}-\lambda w_{\va}^{\alpha}.\nonumber
\end{align}
Let $ \widetilde{f}=f\varphi+V_{\va}^{\alpha\beta}w_{\va}^{\beta}-A_{\va}^{\alpha\beta}\nabla\varphi u_{\va}^{\beta} $ and
\begin{align}
G=-f\nabla\varphi+F\varphi-A_{\va}^{\alpha\beta}\nabla\varphi\nabla u_{\va}^{\beta}-V_{\va}^{\alpha\beta}\nabla\varphi u_{\va}^{\beta}-B_{\va}^{\alpha\beta}\nabla u_{\va}^{\beta}\varphi-c_{\va}^{\alpha\beta}w_{\va}^{\beta}-\lambda w_{\va}^{\alpha}. \nonumber
\end{align}
Then, by using the estimate $ \eqref{W^{1,p} estimates 4} $, we can obtain
\begin{align}
\left\|\nabla w_{\va}\right\|_{L^p(\om(0,{\overline{t}}))}\leq C\left\{\|\widetilde{f}\|_{L^p(\om(0,{\overline{t}})}+\left\|G\right\|_{L^\frac{2p}{p+2}(\om(0,{\overline{t}})}\right\},\nonumber
\end{align}
where we define $ \overline{t}=\frac{t+s}{2} $. This implies that
\begin{align}
\left\|\nabla u_{\va}\right\|_{L^p(\om(0,t))}&\leq  C\left\|\varphi u_{\va}\right\|_{L^p(\om(0,\overline{t}))}+C\left\|\nabla \varphi u_{\va}\right\|_{L^p(\om(0,\overline{t}))}+C\left\|\nabla \varphi\nabla u_{\va}\right\|_{L^{\frac{2p}{p+2}}(\om(0,\overline{t}))}\nonumber\\
&\quad+C\left\|\nabla \varphi u_{\va}\right\|_{L^{\frac{2p}{p+2}}(\om(0,\overline{t}))}+C\left\|\varphi u_{\va}\right\|_{L^{\frac{2p}{p+2}}(\om(0,\overline{t}))}+C\left\|f\nabla \varphi\right\|_{L^{\frac{2p}{p+2}}(\om(0,\overline{t}))}\nonumber\\
&\quad+C\left\|\nabla u_{\va}\varphi\right\|_{L^{\frac{2p}{p+2}}(\om(0,\overline{t}))}+C\left\|f \varphi\right\|_{L^p(\om(0,\overline{t}))}+C\left\|F\varphi\right\|_{L^{\frac{2p}{p+2}}(\om(0,\overline{t}))}.\nonumber
\end{align}
Using the definition of the cut-off function $ \varphi $, we have
\begin{align}
\left\|\nabla u_{\va}\right\|_{L^p(\om(0,t))}&\leq C\left\|u_{\va}\right\|_{L^p(\om(0,\overline{t}))}+\frac{C}{(s-t)}\left\|u_{\va}\right\|_{L^p(\om(0,\overline{t}))}+\frac{C}{(s-t)}\left\|\nabla u_{\va}\right\|_{L^{\frac{2p}{p+2}}(\om(0,\overline{t}))}\nonumber\\
&\quad+\frac{C}{(s-t)}\left\|u_{\va}\right\|_{L^{\frac{2p}{p+2}}(\om(0,\overline{t}))}+C\left\|\nabla u_{\va}\right\|_{L^{\frac{2p}{p+2}}(\om(0,\overline{t}))}+C\left\| u_{\va}\right\|_{L^{\frac{2p}{p+2}}(\om(0,\overline{t}))}\nonumber\\
&\quad+\frac{C}{(s-t)}\left\|f\right\|_{L^{\frac{2p}{p+2}}(\om(0,\overline{t}))}+C\left\|f\right\|_{L^p(\om(0,\overline{t}))}+C\left\|F\right\|_{L^{\frac{2p}{p+2}}(\om(0,\overline{t}))}\nonumber\\
&\leq  \frac{C}{(s-t)}\left\{\left\|u_{\va}\right\|_{L^p(\om(0,\overline{t}))}+\left\|\nabla u_{\va}\right\|_{L^{\frac{2p}{p+2}}(\om(0,\overline{t}))}+\left\|f\right\|_{L^{p}(\om(0,\overline{t}))}\right\}+C\left\|F\right\|_{L^{\frac{2p}{p+2}}(\om(0,\overline{t}))}.\nonumber
\end{align}
This, together with $ \eqref{interp} $ for $ p>2 $, $ q=\frac{2p}{p+2} $ and $ d=2 $, i.e. 
\begin{align}
\left\|u_{\va}\right\|_{L^p(\om(0,\overline{t}))}\leq C\left\|\nabla u_{\va}\right\|_{L^{\frac{2p}{p+2}}(\om(0,\overline{t}))}+C\overline{t}^{\frac{2}{p}-1}\left\|u_{\va}\right\|_{L^2(\om(0,\overline{t}))},\nonumber
\end{align}
gives that
\begin{align}
&\left\|\nabla u_{\va}\right\|_{L^p(\om(0,t))}\nonumber\\
&\quad\quad\leq \frac{C}{(s-t)}\left\|\nabla u_{\va}\right\|_{L^{\frac{2p}{p+2}}(\om(0,\overline{t}))}+\frac{C}{(s-t)(t+s)^{1-\frac{2}{p}}}\left\|u_{\va}\right\|_{L^2(\om(0,\overline{t}))}\nonumber\\
&\quad\quad\quad\quad+\frac{C}{(s-t)}\left\|f\right\|_{L^{p}(\om(0,\overline{t}))}+\left\|F\right\|_{L^{\frac{2p}{p+2}}(\om(0,\overline{t}))}\nonumber\\
&\quad\quad\leq \frac{C\overline{t}^{\frac{2}{p}}}{(s-t)}\left\|\nabla u_{\va}\right\|_{L^{2}(\om(0,\overline{t}))}+\frac{C}{(s-t)(t+s)^{1-\frac{2}{p}}}\left\|u_{\va}\right\|_{L^2(\om(0,\overline{t}))}
\nonumber\\
&\quad\quad\quad\quad+\frac{C}{(s-t)}\left\|f\right\|_{L^{p}(\om(0,\overline{t}))}+\left\|F\right\|_{L^{\frac{2p}{p+2}}(\om(0,\overline{t}))}.\label{ww8}
\end{align}
For the second inequality of $\eqref{ww8} $, we have used the H\"{o}lder's inequality. Finally we can use $ \eqref{28gj} $ to derive
\begin{align}
\left\|\nabla u_{\va}\right\|_{L_{\avg}^{2}(\om(0,\overline{t}))}\leq \frac{C}{(s-t)}\left\|u_{\va}\right\|_{L_{\avg}^2(\om(0,s))}+C\left\|f\right\|_{L_{\avg}^2(\om(0,s))}+C\left\|F\right\|_{L_{\avg}^{\frac{2p}{p+2}}(\om(0,s))},\nonumber
\end{align}
i.e.
\begin{align}
\left\|\nabla u_{\va}\right\|_{L^{2}(\om(0,\overline{t}))}\leq \frac{C}{(s-t)}\left\|u_{\va}\right\|_{L^2(\om(0,s))}+C\left\|f\right\|_{L^2(\om(0,s))}+\frac{C}{s^{\frac{2}{p}}}\left\|F\right\|_{L^{\frac{2p}{p+2}}(\om(0,s))}.\label{ww9}
\end{align}
Combining inequalities $ \eqref{ww8} $, $ \eqref{ww9} $ and the fact that $ \overline{t}/s\leq 1 $, we can obtain
\begin{align}
\left\|\nabla u_{\va}\right\|_{L^p(\om(0,t))}\leq\frac{C}{(s-t)^2t^{1-\frac{2}{p}}}\left\|u_{\va}\right\|_{L^2(\om(0,s))}+\frac{C}{(s-t)}\left\|f\right\|_{L^{p}(\om(0,s))}+\frac{C}{(s-t)}\left\|F\right\|_{L^{\frac{2p}{p+2}}(\om(0,s))},\nonumber
\end{align}
where $ C $ depends only on $ \mu,\omega(t),\kappa,\lambda,m,p,q $. Taking the average and using $ t,s<1 $, we have
\begin{align}
\left\|\nabla u_{\va}\right\|_{L_{\avg}^p(\om(0,t))}\leq \frac{C}{(s-t)^2t}\left\|u_{\va}\right\|_{L_{\avg}^2(\om(0,s))}+\frac{C}{(s-t)t^{\frac{2}{p}}}\left\{\left\|f\right\|_{L_{\avg}^p(\om(0,s))}+\left\|F\right\|_{L_{\avg}^{\frac{2p}{p+2}}(\om(0,s))}\right\}.\nonumber
\end{align}
This completes the proof.
\end{proof}
\begin{cor}
Let $ 2<p<\infty $. Assume that $ A\in\operatorname{VMO}(\mathbb{R}^2) $ satisfies $ \eqref{Ellipticity} $, $ \eqref{Periodicity} $, other coefficients of $ \mathcal{L}_{\va} $ satisfy $ \eqref{Boundedness} $, $ \lambda\geq\lambda_0 $, $ \om $ is a $ C^{1,\eta} $ $ (0<\eta<1) $ bounded domain in $ \mathbb{R}^2 $, $ x_0\in\om $ and $ 0<r<\operatorname{diam}(\om) $. If $ \pa\om\cap B(x_0,2r)\neq\emptyset $, assume that $ u_{\va}\in W^{1,2}(\om(x_0,2r);\mathbb{R}^m) $ is the weak solution to
\begin{align}
\left\{\begin{matrix}
\mathcal{L}_{\va}(u_{\va})=\operatorname{div}(f)+F&\text{in}&\om(x_0,2r),\\
u_{\va}=0&\text{on}&\Delta(x_0,2r),
\end{matrix}\right.\nonumber
\end{align}
with $ f\in L^p(\om(x_0,2r);\mathbb{R}^{m\times 2}) $, $ F\in L^{q}(\om(x_0,2r);\mathbb{R}^m) $ and $ q=\frac{2p}{p+2} $. If $ \pa\om\cap B(x_0,2r)=\emptyset  $, assume that $ u_{\va}\in W^{1,2}(B(x_0,2r);\mathbb{R}^m) $ is a weak solution to
\begin{align}
\mathcal{L}_{\va}(u_{\va})=\operatorname{div}(f)+F\quad\text{in}\quad\om(x_0,2r),\nonumber
\end{align}
with the same data $ f $, $ F $ and $ p $. Then, for $ \sigma=1-\frac{2}{p} $,
\begin{align}
[u_{\va}]_{C^{0,\sigma}(\om_r)}\leq Cr^{-\sigma}\left\{\left(\dashint_{\om_{2r}}|u_{\va}|^2dx\right)^{\frac{1}{2}}+r\left(\dashint_{\om_{2r}}|f|^pdx\right)^{\frac{1}{p}}
+r^2\left(\dashint_{\om_{2r}}|F|^qdx\right)^{\frac{1}{q}}\right\}.\label{Holder interior estimates}
\end{align}
In particular, for all $ \overline{p}>0 $, we have
\begin{align}
\left\|u_{\va}\right\|_{L^{\infty}(\om_{r})}\leq C\left\{\left(\dashint_{\om_{2r}}|u_{\va}|^{\overline{p}}dx\right)^{\frac{1}{\overline{p}}}+r\left(\dashint_{\om_{2r}}|f|^pdx\right)^{\frac{1}{p}}+r^2\left(\dashint_{\om_{2r}}|F|^qdx\right)^{\frac{1}{q}}\right\},\label{infty interior estimates}
\end{align}
where $ \om_r=\om(x_0,r) $ and $ C $ depends only on $ \mu,\omega(t),\kappa,\lambda,m,p,\overline{p},q,\om $.
\end{cor}
\begin{proof}
To obtain more precise scale estimates, we need to make a clear exploration of the relationship between the constant of Morrey's inequality and the domain's radius. It can be seen that on a ball with a radius of $ 1 $, for $  u\in W^{1,p}(B(0,1)) $ and $ p>d $, where $ d\geq 2 $ is the dimension, then according to Morrey's theorem, we can obtain that
\begin{align}
[u]_{C^{0,1-\frac{d}{p}}(B(0,1))}\leq C\left\|\nabla u\right\|_{L^p(B(0,1))}+C\left\|u\right\|_{L^p(B(0,1))}.\nonumber
\end{align}
Then, Lemma \ref{interpolation} implies that
\begin{align}
[u]_{C^{0,1-\frac{d}{p}}(B(0,1))}\leq C\left\|\nabla u\right\|_{L^p(B(0,1))}+C\left\|u\right\|_{L^2(B(0,1))}.\nonumber
\end{align}
Taking $ v(x)=u(rx) $ and using the inequality on $ v $, we have
\begin{align}
r^{1-\frac{d}{p}}[u]_{C^{0,1-\frac{d}{p}}(B(0,r))}\leq Cr^{1-\frac{d}{p}}\left\|\nabla u\right\|_{L^p(B(0,r))}+Cr^{-\frac{d}{2}}\left\|u\right\|_{L^2(B(0,r))},\nonumber
\end{align}
\begin{align}
[u]_{C^{0,1-\frac{d}{p}}(B(0,r))}\leq C\left\|\nabla u\right\|_{L^p(B(0,r))}+Cr^{\frac{d}{p}-\frac{d}{2}-1}\left\|u\right\|_{L^2(B(0,r))}.\label{ww10}
\end{align}
Note that if we change $ B(0,r) $ to $ \om(0,r) $, $ \eqref{ww10} $ is still true. Therefore, for $ \om(x_0,r) $ such that $ x_0\in\om $, $ 0<r<\diam(\om) $, $ d=2 $, $ \sigma=1-\frac{2}{p} $ and $ 0<t<s\leq 1 $, we have
\begin{align}
[u_{\va}]_{C^{0,\sigma}(\om(x_0,tr))}&\leq  C\left\|\nabla u_{\va}\right\|_{L^p(\om(x_0,tr))}+C(tr)^{-1-\sigma}\left\|u_{\va}\right\|_{L^2(\om(x_0,tr))}\nonumber\\
&\leq C(tr)^{\frac{2}{p}}\left\|\nabla u_{\va}\right\|_{L_{\avg}^p(\om(x_0,tr))}+C(tr)^{-\sigma}\left\|u_{\va}\right\|_{L_{\avg}^2(\om(x_0,tr))}\nonumber\\
&\leq C(tr)^{-\sigma}\left\|u_{\va}\right\|_{L_{\avg}^2(\om(x_0,tr))}\label{ww12}\\
&\quad+C(tr)^{\frac{2}{p}}\left\{\frac{C}{(s-t)^2tr}\left\|u_{\va}\right\|_{L_{\avg}^2(\om(x_0,sr))}+\right.\nonumber\\
&\quad\quad+\left.\frac{C}{(s-t)t^{\frac{2}{p}}}\left\{\left\|f\right\|_{L_{\avg}^p(\om(x_0,sr))}+r\left\|F\right\|_{L_{\avg}^q(\om(x_0,sr))}\right\}\right\}\nonumber\\
&\leq\frac{Cr^{-\sigma}}{(s-t)^2t^{\sigma}}\left\|u_{\va}\right\|_{L_{\avg}^2(\om(x_0,sr))}+\frac{Cr^{1-\sigma}}{(s-t)}\left\{\left\|f\right\|_{L_{\avg}^p(\om(x_0,sr))}+r\left\|F\right\|_{L_{\avg}^{\frac{2p}{p+2}}(\om(x_0,sr))}\right\},\nonumber
\end{align}
where, for the second inequality, we use the $ W^{1,p} $ estimate $ \eqref{W^{1,p} interior estimates} $ and $ C $ depends only on $ \mu,\omega(t),\kappa,\lambda,m,p $. By choosing special $ t,s $, we can get the first inequality $ \eqref{Holder interior estimates} $. For any $ x\in \om(x_0,tr) $, by using $ \eqref{ww12} $, we have
\begin{align}
|u_{\va}(x)|&\leq \left|u_{\va}(x)-\dashint_{\om(x,\frac{(s-t)r}{2})}u_{\va}\right|+\left|\dashint_{\om(x,\frac{(s-t)r}{2})}u_{\va}\right|\nonumber\\
&\leq C[u_{\va}]_{C^{0,\sigma}(\om(x_0,\frac{(s+t)r}{2}))}r^{\sigma}+\left\|u_{\va}\right\|_{L_{\avg}^{2}(\om(x_0,\frac{(s-t)r}{2}))}\nonumber\\
&\leq \frac{C}{(s-t)^2t^{\sigma}}\left\|u_{\va}\right\|_{L_{\avg}^{2}(\om(x_0,sr))}+\frac{Cr}{(s-t)}\left\{\left\|f\right\|_{L_{\avg}^{p}(\om(x_0,sr))}+r\left\|F\right\|_{L_{\avg}^{\frac{2p}{p+2}}(\om(x_0,sr))}\right\}.\label{ww11}
\end{align}
Taking special $ t $ and $ s $, we can prove that the second inequality holds when $ \overline{p}=2 $. Next, we will discuss the case of $  0<\overline{p}<2 $. In fact, this is a standard convexity improvement. What is different from \cite{Xu1} is only the difference of a series of inequality indices when $ d=2 $. They are not the essential differences. For the sake of completeness, we will prove this conclusion here. We adopt the methods in \cite{Shen2}. Since the index $ \frac{2d}{d-2} $ appears in the proof of \cite{Shen2}, we need to make a slight adjustment when $ d=2 $. The process of proof involves the iteration from $ \om(x_0,tr) $ to $ \om(x_0,sr) $, which is why we intend to make the estimates more precise than the previous estimates. Firstly, for the proof of the second inequality, we have to do some simple observations to transform it into a homogeneous problem.

Using the dilation and translation, we can assume that $ x_0=0 $ and $ r=1 $. For $ \om(0,1) $, we can choose $ v_{\va}\in W^{1,2}(\om(0,2);\mathbb{R}^m) $ satisfying $ \mathcal{L}_{\va}(v_{\va})=0 $ in $ \om(0,2) $ and $ v_{\va}=u_{\va} $ on $ \pa(\om(0,2)) $. Let $ w_{\va}=u_{\va}-v_{\va} $, then $ \mathcal{L}_{\va}(w_{\va})=\operatorname{div}(f)+F $ in $ \om(0,2) $, and $ w_{\va}=0 $ on $ \pa(\om(0,2)) $. According to $ \eqref{infty interior estimates} $ for $ \overline{p}=2 $, we have
\begin{align}
\left\|v_{\va}\right\|_{L^{\infty}(\om(0,1))}\leq C\left\|v_{\va}\right\|_{L_{\avg}^{2}(\om(0,\frac{3}{2}))},\label{ww13}
\end{align}
and
\begin{align}
\left\|w_{\va}\right\|_{L^{\infty}(\om(0,1))}\leq C\left\{\left\|w_{\va}\right\|_{L_{\avg}^{2}(\om(0,2))}+\left\|f\right\|_{L_{\avg}^{p}(\om(0,2))}+\left\|F\right\|_{L_{\avg}^{\frac{2p}{p+2}}(\om(0,2))}\right\}.\label{ww14}
\end{align}
Moreover, since $ w_{\va}=0 $ on $ \pa(\om(0,2)) $,  using  Theorem \ref{W1p}, we get
\begin{align}
\left\|w_{\va}\right\|_{L_{\avg}^{2}(\om(0,2))}\leq C\left\|w_{\va}\right\|_{L_{\avg}^{p}(\om(0,2))}\leq C\left\|f\right\|_{L_{\avg}^{p}(\om(0,2))}+C\left\|F\right\|_{L_{\avg}^{\frac{2p}{p+2}}(\om(0,2))}.\label{ww15}
\end{align}
By   $ \eqref{ww13} $, $ \eqref{ww14} $ and $ \eqref{ww15} $, we obtain
\begin{align}
\left\|u_{\va}\right\|_{L^{\infty}(\om(0,1))}&\leq \left\|w_{\va}\right\|_{L^{\infty}(\om(0,1))}+\left\|v_{\va}\right\|_{L^{\infty}(\om(0,1))}\nonumber\\
&\leq C\left\|v_{\va}\right\|_{L_{\avg}^{2}(\om(0,\frac{3}{2}))}+C\left\|f\right\|_{L_{\avg}^{p}(\om(0,2))}+C\left\|F\right\|_{L_{\avg}^{\frac{2p}{p+2}}(\om(0,2))}.\label{ww17}
\end{align}
Since $ u_{\va}=0 $ on $ \Delta(0,2) $ and $ v_{\va}=u_{\va} $ on $ \pa(\om(0,2)) $, we have $ \mathcal{L}_{\va}(v_{\va})=0 $ in $ \om(0,2) $ and $ v_{\va}=0 $ on $ \Delta(0,2) $. We claim that,
\begin{align}
\left\|v_{\va}\right\|_{L_{\avg}^{2}(\om(0,\frac{3}{2}))}\leq C\left\|v_{\va}\right\|_{L_{\avg}^{\overline{p}}(\om(0,2))}.\label{ww16}
\end{align}
If we prove the claim, we can get the  conclusions. This is because by using $ \eqref{ww17} $ and $ \eqref{ww16} $,
\begin{align}
\left\|u_{\va}\right\|_{L^{\infty}(\om(0,1))}&\leq C\left\|v_{\va}\right\|_{L_{\avg}^{\overline{p}}(\om(0,2))}+C\left\|f\right\|_{L_{\avg}^{p}(\om(0,2))}+C\left\|F\right\|_{L_{\avg}^{\frac{2p}{p+2}}(\om(0,2))}\nonumber\\
&\leq C\left\|u_{\va}\right\|_{L_{\avg}^{\overline{p}}(\om(0,2))}+C\left\|f\right\|_{L_{\avg}^{p}(\om(0,2))}+C\left\|F\right\|_{L_{\avg}^{\frac{2p}{p+2}}(\om(0,2))}+\left\|w_{\va}\right\|_{L_{\avg}^{\overline{p}}(\om(0,2))}\nonumber\\
&\leq C\left\|u_{\va}\right\|_{L_{\avg}^{\overline{p}}(\om(0,2))}+C\left\|F\right\|_{L_{\avg}^{p}(\om(0,2))}+C\left\|f\right\|_{L_{\avg}^{\frac{2p}{p+2}}(\om(0,2))}+\left\|w_{\va}\right\|_{L_{\avg}^{2}(\om(0,2))}\nonumber\\
&\leq C\left\|u_{\va}\right\|_{L_{\avg}^{\overline{p}}(\om(0,2))}+C\left\|f\right\|_{L_{\avg}^{p}(\om(0,2))}+C\left\|F\right\|_{L_{\avg}^{\frac{2p}{p+2}}(\om(0,2))}.\nonumber
\end{align}
In this way, the proof of $ \eqref{infty interior estimates} $ comes down to the proof of $ \eqref{ww16} $. Next, we will prove the claim. We can assume that $ \left\|v_{\va}\right\|_{L_{\avg}^{\overline{p}}(\om(0,2))}=1 $,  for otherwise, setting $ \widetilde{v}_{\va}=v_{\va}/\left\|v_{\va}\right\|_{L_{\avg}^{\overline{p}}(\om(0,2))} $, we can obtain that $ \left\|\widetilde{v}_{\va}\right\|_{L_{\avg}^{\overline{p}}(\om(0,2))}=1 $ and
$$
\mathcal{L}_{\va}(\widetilde{v}_{\va})=0\text{ in }\om(0,2)\quad\text{and}\quad \widetilde{v}_{\va}=0\text{ on }\Delta(0,2).
$$
By using the results for the case that $ \left\|v_{\va}\right\|_{L_{\avg}^{\overline{p}}(\om(0,2))}=1 $, we can prove the general case. Choosing $ p_1>2 $, and using $ \eqref{ww11} $ with $ \overline{p}=2 $, we have, for $ 0<t<s\leq 1 $
\begin{align}
\left\|v_{\va}\right\|_{L_{\avg}^{p_1}(\om(0,t))}&\leq \left\|v_{\va}\right\|_{L^{\infty}(\om(0,t))}\leq \frac{C}{(s-t)^2t^{\sigma}}\left\|v_{\va}\right\|_{L_{\avg}^{2}(\om(0,s))}.\nonumber
\end{align}
For $ 0<\alpha<1 $, such that $ \frac{\alpha}{p_1}+\frac{1-\alpha}{\overline{p}}=\frac{1}{2} $, i.e. $ 0<\al=(\frac{1}{\overline{p}}-\frac{1}{2})/(\frac{1}{\overline{p}}-\frac{1}{p_1})<1 $,
\begin{align}
\left\|v_{\va}\right\|_{L_{\avg}^{2}(\om(0,t))}&\leq  \left\|v_{\va}\right\|_{L_{\avg}^{p_1}(\om(0,t))}^{\al}\left\|v_{\va}\right\|_{L_{\avg}^{\overline{p}}(\om(0,t))}^{1-\al}\leq \left(\frac{C}{(s-t)^2t^{\sigma}}\right)^{\alpha}\left\|v_{\va}\right\|_{L_{\avg}^{2}(\om(0,s))}^{\al}\left\|v_{\va}\right\|_{L_{\avg}^{\overline{p}}(\om(0,t))}^{1-\al}\nonumber\\
&\leq \frac{C}{(s-t)^{2\alpha}t^{\sigma\alpha+2\frac{(1-\alpha)}{\overline{p}}}}\left\|v_{\va}\right\|_{L_{\avg}^{2}(\om(0,s))}^{\al}\left\|v_{\va}\right\|_{L^{\overline{p}}(\om(0,1))}^{1-\al}.\nonumber
\end{align}
Let $ I(t)=\left\|v_{\va}\right\|_{L_{\avg}^{2}(\om(0,t))} $, we can obtain the iteration formula as follows
\begin{align}
I(t)\leq \frac{C}{(s-t)^{2\alpha}t^{\sigma\alpha+2\frac{(1-\alpha)}{\overline{p}}}}\left\|v_{\va}\right\|_{L^{\overline{p}}(\om(0,1))}^{1-\al}(I(s))^{\alpha}.\nonumber
\end{align}
Take the logarithm on both sides, multiply them by $ \frac{1}{s} $, choose $ t=s^b $  with $ b=\frac{1+\al}{2\al} $ and integrate $ s $ from $ \frac{1}{2} $ to $ 1 $. Then we have
\begin{align}
\int_{\frac{1}{2}}^{1}\ln(I(t))\frac{ds}{s}\leq C+\alpha\int_{\frac{1}{2}}^{1}\ln(I(s))\frac{ds}{s}+C\ln\left\|v_{\va}\right\|_{L^{\overline{p}}(\om(0,1))}.\nonumber
\end{align}
If there exists $ t_0\in [(\frac{1}{2})^b,\frac{1}{2}] $ such that $ I(t_0)\leq 1 $,  then it is easy to get that
\begin{align}
\left(\dashint_{\om(0,(\frac{1}{2})^b)}|v_{\va}|^2\right)^{\frac{1}{2}}\leq CI(t_0)\leq C\leq C\left(\dashint_{\om(0,2)}|v_{\va}|^{\overline{p}}\right)^{\frac{1}{\overline{p}}},\label{gj1}
\end{align}
where $ C $ depends only on $ \mu,\omega(t),\kappa,\lambda,m,p,\overline{p},q,\om $. If for any $ t\in [(\frac{1}{2})^b,\frac{1}{2}] $, $ I(t)\geq 1 $, then 
\begin{align}
\int_{\frac{1}{2}}^{1}\ln(I(s^b))\frac{ds}{s}=\frac{1}{b}\int_{(\frac{1}{2})^b}^{1}\ln(I(s))\frac{ds}{s}\geq \frac{1}{b}\int_{\frac{1}{2}}^{1}\ln(I(s))\frac{ds}{s}.\nonumber
\end{align}
Hence $ \frac{1}{b}-\alpha=\frac{\al+1}{2\al}-\al>0 $, and
\begin{align}
\left(\frac{1}{b}-\alpha\right)\int_{\frac{1}{2}}^{1}\ln(I(s))\frac{ds}{s}\leq C+C\ln\left\|v_{\va}\right\|_{L^{\overline{p}}(\om(0,1))}.\label{gj2}
\end{align}
Moreover, for any $ t\in [\frac{1}{2},1] $, there exists $ C $ depending only on $ \mu,\omega(t),\kappa,\lambda,m,p,\overline{p},q,\om $, such that
\begin{align}
I(t)\geq C\left(\dashint_{\om(0,(\frac{1}{2})^b)}|v_{\va}|^2\right)^{\frac{1}{2}}.\label{gj3}
\end{align}
In view of $ \eqref{gj2} $ and $ \eqref{gj3} $, it can be got that
\begin{align}
\ln\left\{\left(\dashint_{\om(0,(\frac{1}{2})^b)}|v_{\va}|^2\right)^{\frac{1}{2}}\right\}\leq C+C\ln\left\|v_{\va}\right\|_{L^{\overline{p}}(\om(0,1))},\nonumber
\end{align}
and then
\begin{align}
\left(\dashint_{\om(0,(\frac{1}{2})^b)}|v_{\va}|^2\right)^{\frac{1}{2}}\leq C\left(\dashint_{\om(0,1)}|v_{\va}|^{\overline{p}}\right)^{\frac{1}{\overline{p}}}\leq C\left(\dashint_{\om(0,2)}|v_{\va}|^{\overline{p}}\right)^{\frac{1}{\overline{p}}}.\nonumber
\end{align}
This, together with $ \eqref{gj1} $ and almost the same arguments, implies that
there is a constant $ r_0=\min\{\frac{1}{16},(\frac{1}{2})^b\} $ such that for any $ x\in \om(0,\frac{3}{2}) $,
\begin{align}
\left(\dashint_{\om(x,r_0)}|v_{\va}|^2\right)^{\frac{1}{2}}\leq C\left(\dashint_{\om(0,2)}|v_{\va}|^{\overline{p}}\right)^{\frac{1}{\overline{p}}}.\nonumber
\end{align}
By covering $ \om(0,\frac{3}{2}) $ with $ \om(x,r_0) $ defined above, we can complete the proof of claim $ \eqref{ww16} $.
\end{proof}

\begin{cor}
Suppose that $ A\in \operatorname{VMO}(\mathbb{R}^2) $ satisfies $ \eqref{Ellipticity} $, $ \eqref{Periodicity} $, and other coefficients of $ \mathcal{L}_{\va} $,$ V,B,c $ satisfy $ \eqref{Boundedness} $. $ \om $ is a $ C^{1,\eta} $ $ (0<\eta<1) $  bounded domain in $ \mathbb{R}^2 $, $ f\in L^p(\om;\mathbb{R}^{m\times 2}) $, $ F\in L^q(\om;\mathbb{R}^m) $ and $ g\in C^{0,1}(\pa\om;\mathbb{R}^m) $, where $ 2<p<\infty $, $ q=\frac{2p}{p+2} $ and $ \sigma=1-\frac{2}{p} $. Then the weak solution to $ \eqref{equation} $ satisfies the uniform estimate
\begin{align}
\left\|u_{\va}\right\|_{C^{0,\sigma}(\om)}\leq C\left\{\left\|F\right\|_{L^q(\om)}+\left\|f\right\|_{L^p(\om)}+\left\|g\right\|_{C^{0,1}(\pa\om)}\right\},\label{Holder estimates elementary}
\end{align}
where $ C $ depends only on $ \mu,\omega(t),\kappa,\lambda,m,p,\sigma $ and $ \om $.
\end{cor}
\begin{proof}
One can use Sobolev embedding theorem and Theorem \ref{W1p} to finish the proof. For details, one can see Corollary 3.8 in \cite{Xu1}.
\end{proof}

If we do not assume that $ A\in \VMO(\mathbb{R}^2) $, we have the following theorem given by Theorem 4.3.3 in \cite{Shen1}.

\begin{thm}\label{ww29}
Suppose that A satisfies $ \eqref{Ellipticity} $, and other coefficients of $ \mathcal{L}_{\va} $,$ V,B,c $ satisfy $ \eqref{Boundedness} $. Let $ \om $ be a bounded Lipschitz domain in $ \mathbb{R}^d $ with $ d\geq 2 $. Then there exists $ \delta\in(0,\frac{1}{2}) $, depending only on $ \mu $ and $ \om $, such that for any $ F\in W^{-1,p}(\om;\mathbb{R}^{m})$ and $g \in B^{1-\frac{1}{p}, p}\left(\pa\om;\mathbb{R}^{m}\right)$ with $ |\frac{1}{p}-\frac{1}{2}|<\delta $, there exists a unique solution in $ W^{1,p}(\om ; \mathbb{R}^{m})$ to the Dirichlet problem: $ L_{\va}(u_{\va})=F $ in $ \om $ and $ u_{\va}=g $ on $ \pa\om $. Moreover, the solution satisfies the estimate
\begin{align}
\left\|u_{\va}\right\|_{W^{1,p}(\om)}\leq C\left\{\left\|F\right\|_{W^{-1,p}(\om)}+\left\|g\right\|_{B^{1-\frac{1}{p},p}(\pa\om)}\right\},\nonumber
\end{align}
with constant $ C $ depending only on $ p,\mu $ and $\om $.
\end{thm}
From Theorem \ref{ww29} and the proof of Theorem \ref{W1p}, we can obtain the $ W^{1,p} $ estimates for the operator $ \mathcal{L}_{\va} $ without the assumption that $ A\in \VMO(\mathbb{R}^2) $.

\begin{thm}\label{W1p 2} Suppose that $ A $ satisfies $ \eqref{Ellipticity} $, other coefficients of $ \mathcal{L}_{\va} $, $ V,B,c $ satisfy $ \eqref{Boundedness} $ and $ \om $ is a $ C^{1,\eta} $ $ (0<\eta<1) $ bounded domain in $ \mathbb{R}^2 $. Then, there exists $ \delta\in(0,\frac{1}{2}) $, depending only on $ \mu $ and $ \om $, such that for any $ |\frac{1}{p}-\frac{1}{2}|<\delta $, the Dirichlet problem $ \mathcal{L}_{\va}(u_{\va})=\operatorname{div}(f)+F $ in $ \om $, $ u_{\va}=g $ on $ \pa\om $ with $ f\in L^{p}(\om;\mathbb{R}^{m\times 2})$, $ F\in L^{q}(\om;\mathbb{R}^{m})$ and $ g\in B^{1-\frac{1}{p}, p}\left(\pa\om;\mathbb{R}^{m}\right)$ has a unique weak solution $ u_{\va}\in W^{1,p}(\om;\mathbb{R}^m) $, whenever $ \lambda\geq\lambda_0 $. Furthermore, the solution satisfies the uniform estimate
\begin{align}
\left\|\nabla u_{\va}\right\|_{L^{p}(\om)}\leq C\left\{\left\|f\right\|_{L^p(\om)}+\left\|F\right\|_{L^q(\om)}+\left\|g\right\|_{B^{1-\frac{1}{p},p}(\pa\om)}\right\},\label{W^{1,p} estimates 2}
\end{align}
where $ q=\frac{2p}{p+2} $ if $ p>2 $, $ 1<q<\frac{1}{1-2\delta} $ if $ p=2 $, $ q=1 $ if $ \frac{1}{1+2\delta}<p<2 $ and $ C $ depends only on $ \mu,\omega(t),\kappa,\lambda,p,q,m $ and $ \om $.
\end{thm}

By using Sobolev embedding theorem for the case that $ d=2 $, we can deduce that $ u_{\va} $ in Theorem \ref{W1p 2} is actually H\"{o}lder continuous.

\section{Green functions for the operator $ \mathcal{L}_{\va} $}
The following definitions about $ \operatorname{BMO} $ and atom functions are from \cite{Taylor} and is essential for the construction of the Green functions for the elliptic operator $ \mathcal{L}_{\va} $.

\begin{defn}[$ \operatorname{BMO} $ space and atom functions]\label{BMO atom}
For $ x_0\in\mathbb{R}^d $, $ r>0 $ and $ \om $ a domain in $ \mathbb{R}^d $. We denote
\begin{align}
\om(x_0,r):=\om\cap B(x_0,r),\nonumber
\end{align}
as before. We define that the $ \operatorname{BMO}(\om) $ is a space containing functions such that
\begin{align}
\left\|u\right\|_*=\sup\left\{\dashint_{\om(x_0,r)}|u-\overline{u}_{x_0,r}|:x_0\in\overline{\om},r>0\right\} \nonumber
\end{align}
is finite, where we define
\begin{align}
\overline{u}_{x_0,r}:=\left\{\begin{array}{ccc}
0 & \text{ if }  & r\geq \operatorname{dist}(x_0,\pa\om), \\
\dashint_{\om(x_0,r)}u & \text{ if } & r< \operatorname{dist}(x_0,\pa\om).
\end{array}\right.\label{ww35}
\end{align}
We call the bounded measurable function $ a $ atom function in $ \om $ if $ \operatorname{supp}(a)\subset\om(x_0,r) $ with $ x_0\in\overline{\om} $ for $ r>0 $ and
\begin{align}
\left\|a\right\|_{\infty}\leq\frac{1}{|\om(x_0,r)|},\quad \overline{a}_{x_0,r}=0. \nonumber
\end{align}
\end{defn}

\begin{lem}[Generalized Morrey's inequality]
Let $ \om $ be a bounded $ C^{1,\eta} $ domain in $ \mathbb{R}^d $($ d\geq 2 $) with $ 0<\eta<1 $. For $ d<p<\infty $, $ \rho>0 $, $ u\in W_0^{1,p}(\om;\mathbb{R}^m) $, and $ x_0\in \overline{\om} $, we have
\begin{align}
\left\|u-\overline{u}_{x_0,\rho}\right\|_{L^{\infty}(\om(x_0,\rho))}\leq C\rho^{1-\frac{d}{p}}\left(\int_{\om(x_0,2\rho)}|\nabla u|^pdy\right)^{\frac{1}{p}},\label{Generalized Morrey inequality}
\end{align}
where $ C $ depends only on $ m,d,p,\om $.
\end{lem}
\begin{proof}
If $ \rho\geq\operatorname{dist}(x_0,\pa\om) $, we have $ \overline{u}_{x,\rho}=0 $. Then, we can choose $\overline{x}\in\pa\om $ such that $ \operatorname{dist}(x_0,\pa\om)=|x-\overline{x}| $. For $ x\in \om(x_0,\rho) $, we have
\begin{align}
|u(x)|\leq |x-\overline{x}|^{1-\frac{d}{p}}[u]_{C^{0,1-\frac{d}{p}}(\om(x_0,2\rho))}+|u(\overline{x})|\leq C\rho^{1-\frac{d}{p}}\left(\int_{\om(x_0,2\rho)}|\nabla u|^pdy\right)^{\frac{1}{p}}.\nonumber
\end{align}
If $ \rho<\operatorname{dist}(x_0,\pa\om) $ and $ x\in \om(x_0,\rho)) $, we have
\begin{align}
|u(x)-\overline{u}_{x_0,\rho}|\leq C\rho^{1-\frac{d}{p}}[u]_{C^{0,1-\frac{d}{p}}(\om(x_0,2\rho))}\leq C\rho^{1-\frac{d}{p}}\left(\int_{\om(x_0,2\rho)}|\nabla u|^pdy\right)^{\frac{1}{p}}.\nonumber
\end{align}
This proves the conclusion.
\end{proof}
\begin{lem}
Assume that $ A\in\operatorname{VMO}(\mathbb{R}^2) $ satisfies $ \eqref{Ellipticity} $ and $ \eqref{Periodicity} $, other coefficients of $ \mathcal{L}_{\va} $ satisfy $ \eqref{Boundedness} $ and $ \om $ is a $ C^{1,\eta} $ $ (0<\eta<1) $ bounded domain in $ \mathbb{R}^2 $. Let $ a $ be an atom function in $ \om $. If $ u_{\va} $ is the unique weak solution for the Dirichlet problem $ \mathcal{L}_{\va}(u_{\va})=a $ in $ \om $ and $ u_{\va}=0 $ on $ \pa\om $, then there exists a constant $ C $ depending only on $ \mu,\omega(t)\kappa,m,\lambda,p,q, $ and $ \om $ such that
\begin{align}
\left\|u_{\va}\right\|_{\infty}\leq C.\label{L^{infty} estimates with atom}
\end{align}
\end{lem}
\begin{proof}
For atom function $ a $, we can assume that
\begin{align}
\operatorname{supp}(a)\subset\om(x_0,\rho)\quad\text{and}\quad\left\|a\right\|_{\infty}\leq\frac{1}{|\om(x_0,\rho)|}\nonumber
\end{align}
with $ x_0\in\om $ and $ 0<\rho<\operatorname{diam}(\om) $. Fix $ z\in\om $, we can choose $ p>2 $. Using the Sobolev embedding theorem, we have
$$
|u(z)| \leq  |u(z)-\overline{u}_{z,\rho}|+|\overline{u}_{z,\rho}|
 \leq C\left\{\rho^{1-\frac{2}{p}}\left\|\nabla u_{\va}\right\|_{L^p(\om)}+\rho^{\frac{2}{p}-1}\left\|u_{\va}\right\|_{L^{\frac{2p}{p-2}}(\om)}\right\}.
$$
From the $ W^{1,p} $ estimates and $ \eqref{Dual estimates} $, we have
\begin{align}
\left\|\nabla u_{\va}\right\|_{L^p(\om)}&\leq C\left\|a\right\|_{L^{\frac{2p}{p+2}}(\om)}\leq C\rho^{\frac{2}{p}-1}, \nonumber\\
\left\|u_{\va}\right\|_{L^{\frac{2p}{p-2}}(\om)}&\leq C\left\|a\right\|_{\dot{W}^{-1,p'}(\om)}.\nonumber
\end{align}
Then, we only need to prove that $ \left\|a\right\|_{\dot{W}^{-1,p'}(\om)}\leq \rho^{1-\frac{2}{p}} $. It is because for all $ v\in \dot{W}_0^{1,p}
(\om;\mathbb{R}^m) $, we have
\begin{align}
\left|\int_{\om}a^{\alpha}v^{\alpha}dy\right|\leq\left|\int_{\om}a^{\alpha}(v^{\al}-\overline{v}_{x,\rho}^{\al})dy\right|\leq C\rho^{1-\frac{2}{p}}\left\|\nabla v\right\|_{L^p(\om)},\nonumber
\end{align}
where we have used $ \eqref{Generalized Morrey inequality}$. This completes the results.
\end{proof}
\begin{defn}[Hardy space] Let $ \om $ is a $ C^1 $ domain in $ \mathbb{R}^d $. A function $ f $ is an element in the Hardy space $ \mathcal{H}^1 $, if there exist a sequence of atoms $ \left\{a_i\right\}_{i=1}^{\infty} $ and a sequence of real numbers $ \left\{\lambda_{i}\right\}\in l^1 $ such that $ f=\sum_{i=1}^{\infty}\lambda_ia_i $. We define the norm in this space as
\begin{align}
\left\|f\right\|_{\mathcal{H}^1(\om)}=\inf\left\{\sum_{i=1}^{\infty}|\lambda_i|:f=\sum_{i=1}^{\infty}\lambda_ia_i\right\}.\nonumber
\end{align}
We notice the expression
\begin{align}
\sup\left\{\int_{\om}a(y)u(y)dy:a\text{ is an atom in }\om\right\},\nonumber
\end{align}
gives the equivalent norm of $ \operatorname{BMO}(\om) $. This is because that $ \operatorname{BMO}(\om) $ space can be regarded as the dual space of $ \mathcal{H}^1 $.
\end{defn}

Next, we will give a vital property for $ \BMO $ space. For the complete proof of the theorem, we refer to Corollary 6.22 of \cite{MG}.

\begin{thm}
Let $ u\in \BMO (\om) $ with $ \om $ being a bounded domain in $ \mathbb{R}^d $, then $ u\in L^{p}(\om)$ for all $ 1\leq p<\infty $ and there is a $ C $ depending only on $ p $ and $ d $, such that
\begin{align}
\left(\dashint_{\om(x_0,r)}|u-\overline{u}_{x_0,r}|^{p} d x\right)^{\frac{1}{p}} \leq C\left\|u\right\|_{*},\label{Corollary 6.22}
\end{align}
for any $ x_0\in\om $, $ 0<r<\diam(\om) $.
\end{thm}

\begin{proof}[Proof of Theorem \ref{Green}]
Here we follow the proof in \cite{Dong2}, \cite{Dong1}  and \cite{Taylor} to construct the Green function and get its pointwise estimates. For $ y\in \om $, there exists a matrix valued function $ G_{\rho,\va}^{\alpha\beta}(\cdot,y):\om\to\mathbb{R}^{m\times m} $, such that for any $ 1\leq\gamma\leq m $, $ u_{\rho}=(u_{\rho}^{\delta})=(G_{\rho,\va}^{\delta\gamma}(\cdot,y))\in W_0^{1,2}(\om;\mathbb{R}^m) $ satisfies
\begin{align}
\langle\mathcal{L}_{\va}(u_{\rho}),\varphi\rangle=\langle\mathcal{L}_{\va}(G_{\rho,\va}^{\delta\gamma}(\cdot,y)),\varphi^{\delta}\rangle=\dashint_{\om(y,\rho)}\varphi^{\gamma}dx,\quad\forall\varphi\in W_0^{1,2}(\om;\mathbb{R}^m).\nonumber
\end{align}
Here, we call $ G_{\rho,\va}^{\alpha\beta}(\cdot,y) $ is the average Green matrix for the operator $ \mathcal{L}_{\va} $. Then for all atom functions in $ \om $ denoted as $ a $ such that
\begin{align}
\operatorname{supp}(a)\subset\om(y,\rho)\quad\text{and}\quad\left\|a\right\|_{\infty}\leq\frac{1}{|\om(y,\rho)|},\nonumber
\end{align}
we can obtain $ v_{\va}\in W_0^{1,2}(\om;\mathbb{R}^m) $ such that $ \mathcal{L}_{\va}^*(v_{\va})=a $.  Because of the properties of dual operators, we have
\begin{align}
\dashint_{\om(y,\rho)}v_{\va}^{\gamma}dx=\langle \mathcal{L}_{\va}(u_{\rho}),v_{\va}\rangle=\langle u_{\rho},\mathcal{L}_{\va}^*(v_{\va})\rangle=\int_{\om}G_{\rho,\va}^{\alpha\gamma}(\cdot,y)a^{\alpha}(\cdot).\nonumber
\end{align}
From $ \eqref{L^{infty} estimates with atom} $, we have $ \left\|v_{\va}\right\|_{\infty}\leq C $, where $ C $ depends only on $ \mu,\omega(t),\kappa,m,\lambda $ and $ \om $. Then we have
\begin{align}
\left|\int_{\om}G_{\rho,\va}(x,y)a(x)dx\right|\leq C.\nonumber
\end{align}
According to the fact that $ \mathcal{H}^1 $ is the dual space of $ \operatorname{BMO} $ space, we can derive that $ G_{\rho,\va}(\cdot,y) $ has a uniform boundedness $ C $ in $ \operatorname{BMO} $ space, where $ C $ depends only on $ \mu,\omega(t),\kappa,m,\lambda $ and $ \om $. From Banach-Alaoglu theorem, we have, for all $ y\in\om $, there exists a sequence $ \rho_j $ such that $ \rho_j\to 0 $ and functions $ G_{\va}^{\alpha\beta}(\cdot,y)\in \operatorname{BMO}(\om) $ such that $ G_{\rho_j,\va}^{\alpha\beta}(\cdot,y) $ converge $ G_{\va}^{\alpha\beta}(\cdot,y) $ in $ \operatorname{BMO}(\om) $ space in the sense of the weak-* topology. For $ F\in L^q(\om;\mathbb{R}^m) $ where $ q>1 $, we can choose $ 1<q_1<q $, $ p=\frac{2q_1}{2-q_1}>2 $ and $ u_{\va}\in W_0^{1,p}(\om;\mathbb{R}^m) $, such that $ \mathcal{L}_{\va}^*(u_{\va})=F $. Then, we have
\begin{align}
\dashint_{\om(y,\rho)}u_{\va}(x)dx=\int_{\om}G_{\rho,\va}(x,y)F(x)dx,\label{ww18}
\end{align}
and $ \left\|u_{\va}\right\|_{W^{1,p}(\om)}\leq C\left\|F\right\|_{L^q(\om)} $, where $ C $ is a constant depending only on $ \mu,\omega(t),\kappa,m,\lambda,p,q $. Using Sobolev embedding theorem, we see that $ u_{\va} $ is a continuous function. Letting $ \rho_j\to 0 $, the left side of $ \eqref{ww18} $ converges to $ u_{\va}(y) $. On the other hand, for the right-hand side, according to $ L^p\subset \mathcal{H}^1 $, we have
\begin{align}
u_{\va}(y)=\int_{\om}G_{\va}(x,y)F(x)dx.\label{ww19}
\end{align}
Of course, this is still a certain distance from the representation theorem. Next, we will prove the representation theorem and the uniqueness of the Green function. Set $ G_{\widetilde{\rho},\va}^{*\alpha\beta}(\cdot,x) $ is the average Green matrix for $ \mathcal{L}_{\va}^* $. Then, by the definition of the average Green matrix for $ \mathcal{L}_{\va} $, we have
\begin{align}
&\dashint_{\om(y,\rho)}G_{\widetilde{\rho},\va}^{*\alpha\beta}(z,x)dz=\langle \mathcal{L}_{\va}(G_{\rho,\va}^{\alpha}(\cdot,y)),G_{\widetilde{\rho},\va}^{*\beta}(\cdot,x)\rangle=\langle G_{\rho,\va}^{\alpha}(\cdot,y),\mathcal{L}_{\va}^*(G_{\widetilde{\rho},\va}^{*\beta}(\cdot,x))\rangle=\dashint_{\om(x,\widetilde{\rho})}
G_{\rho,\va}^{\beta\alpha}(z,y)dz.\nonumber
\end{align}
Choose $ \rho_n $ and $ \widetilde{\rho}_n $ are two sequences such that the two average Green matrices converge. Letting $ n\to\infty $ at the same time we have
\begin{align}
G_{\va}^{*\alpha\beta}(y,x)=G_{\va}^{\beta\alpha}(x,y)\label{adjoint p}
\end{align}
and the representation theorem can be obtained by this and $ \eqref{ww19} $. It is easy to verify the uniqueness of the Green function. If we assume that $ \widetilde{G}_{\va}(x,y) $ is another Green function that satisfies the properties above, we can take $ F\in C_0^{\infty}(\om) $ and the Dirichlet problem $ \mathcal{L}_{\va}(u_{\va})=F $ in $ \om $ with $ u_{\va}=0 $ on $ \pa\om $ has the unique solution (taking by the representation theorem) $ u_{\va}(x)=\int_{\om}G_{\va}(x,y)F(y)dy $ and $ u_{\va}(x)=\int_{\om}\widetilde{G}_{\va}(x,y)F(y)dy $. Then we have
\begin{align}
\int_{\om}(G_{\va}(x,y)-\widetilde{G}_{\va}(x,y))F(y)dy=0.\nonumber
\end{align}
According to the arbitrariness of $ F $, the uniqueness of the Green function can be obtained.

Finally, we will prove the pointwise estimates for the Green functions. Namely, we will prove $ \eqref{preliminary} $-$ \eqref{Pointwise estimates for Green functions 6} $. Letting $ x_0,y_0\in\om $ and assuming that $ \delta(x_0)<\frac{1}{2}|x_0-y_0|=\frac{1}{2}r $, we have $ \om(x_0,\frac{1}{2}r)\subset\om\backslash\left\{y_0\right\} $. According to the definition of $ G_{\va}^{\cdot\gamma}(\cdot,y_0) $, we have $ \mathcal{L}_{\va}(G_{\va}^{\cdot\gamma}(\cdot,y_0))=0 $ in $ \om(x_0,\frac{1}{2}r) $ and $ G_{\va}^{\cdot\gamma}(\cdot,y_0)=0 $ on $ \om\cap B(x_0,\frac{1}{2}r) $. Then, by using $\eqref{infty interior  estimates}$, we have
\begin{align}
|G(x_0,y_0)|\leq C\left\|G(\cdot,y_0)\right\|_{L^{\infty}(\om(x_0,\frac{r}{4}))}\leq C\dashint_{\om(x_0,\frac{1}{2}r)}|G(z,y_0)|dz.\nonumber
\end{align}
According to the definition of $ \operatorname{BMO}(\om) $ and noticing that $ \overline{G(z,y_0)}_{x_0,\frac{r}{2}}=0 $ by $ \eqref{ww35} $, we have
\begin{align}
|G(x_0,y_0)|\leq C\dashint_{\om(x_0,\frac{r}{2})}|G(z,y_0)-\overline{G(z,y_0)}_{x_0,\frac{r}{2}}|dz\leq C\left\|G(\cdot,y_0)\right\|_{*}\leq C,\nonumber
\end{align}
where $ C $ depends only on $ \mu,\omega(t),\kappa,m,\lambda $ and $ \om $. Then for $ \delta(x)<\frac{1}{2}|x-y| $, we can obtain
\begin{align}
|G_{\va}(x,y)|\leq C,\label{Boundedness estimates}
\end{align}
where $ C $ depends only on $ \mu,\kappa,m,\lambda,\omega(t) $ and $ \om $. Assume that $ \delta(x_0)<\frac{1}{4}|x_0-y_0|=\frac{1}{4}r $, and $ z_0 $ is chosen such that $ |x_0-z_0|=\delta(x_0) $, we have that $ \mathcal{L}_{\va}(G_{\va}^{\cdot\gamma}(\cdot,y_0))=0 $ in $ \om(x_0,\frac{r}{2}) $ and $ G_{\va}^{\cdot\gamma}(\cdot,y_0)=0 $ on $ \pa\om\cap B(z_0,\frac{r}{2}) $. According to the localized boundary H\"{o}lder estimates, $\eqref{Holder interior estimates}$, we have, for all $ \sigma_1\in (0,1) $ by letting $ v_{\va}(x)=G_{\va}^{\cdot\gamma}(x,y_0) $, there is
\begin{align}
|v_{\va}(x_0)|&= |v_{\va}(x_0)-v_{\va}(z_0)|+|v_{\va}(z_0)| 
 \leq |x_0-z_0|^{\sigma_1}[u_{\va}]_{C^{0,\sigma_1}(\om(z_0,\frac{3r}{8}))}\nonumber\\
&\leq C\left(\frac{[\delta(x_0)]}{r}\right)^{\sigma_1}\left(\dashint_{\om(z_0,\frac{r}{2})}|v_{\va}|^2dx\right)^{\frac{1}{2}}.\nonumber
\end{align}
Here, we notice that for all $ x\in \om(z_0,\frac{r}{2}) $, $ \delta(x)<\frac{1}{2}r $. Then by using $ \eqref{Boundedness estimates} $, we have
\begin{align}
|v_{\va}(x_0)|\leq C\left(\frac{[\delta(x_0)]}{r}\right)^{\sigma_1}\leq C\frac{[\delta(x_0)]^{\sigma_1}}{|x_0-y_0|^{\sigma_1}},\nonumber
\end{align}
which proves $ \eqref{Pointwise estimates for Green functions 1} $. On the other hand, the second inequality $ \eqref{Pointwise estimates for Green functions 2} $ can be obtained naturally according to the Green function of the adjoint operator. In addition, we can see from the above proof that $ \frac{1}{4} $ is not an essential constant in $ \eqref{Pointwise estimates for Green functions 1} $-$ \eqref{Pointwise estimates for Green functions 2} $. That is, $ \eqref{Pointwise estimates for Green functions 1} $-$ \eqref{Pointwise estimates for Green functions 2} $ are also true if we change $ \frac{1}{4} $ to any constants $ 0<C_0<\frac{1}{2} $ by using same arguments. Next, let us prove $ \eqref{Pointwise estimates for Green functions 3} $. Firstly, we can assume that $ \delta(x_0)<\frac{1}{4}|x_0-y_0| $ and $ \delta(y_0)<\frac{1}{4}|x_0-y_0| $. Otherwise, $ \eqref{Pointwise estimates for Green functions 3} $ follows from $ \eqref{Pointwise estimates for Green functions 1} $-$ \eqref{Pointwise estimates for Green functions 2} $ directly. Using the same arguments, we can obtain
\begin{align}
|v_{\va}(x_0)|&\leq C\left(\frac{[\delta(x_0)]}{r}\right)^{\sigma_1}\left(\dashint_{\om(z_0,\frac{7r}{16})}|v_{\va}|^2dx\right)^{\frac{1}{2}}.\nonumber
\end{align}
For all $ y\in B(z_0,\frac{7r}{16})\cap\om $, $ |y-y_0|\geq |x_0-y_0|-|x_0-y|\geq\frac{9}{16}r $ and $ \delta(y_0)\leq\frac{1}{4}|x_0-y_0|=\frac{1}{4}r\leq \frac{4}{9}r $. So according to the above description, we can also obtain that
\begin{align}
|v_{\va}(y)|\leq\frac{C[\delta(y_0)]^{\sigma_2}}{|y-y_0|^{\sigma_2}}\text{ for any }y\in \om(z_0,\frac{7r}{16}).\label{ww20}
\end{align}
Then in view of $ \eqref{ww20} $,
\begin{align}
|v_{\va}(x_0)|&\leq C\left(\frac{[\delta(x_0)]}{r}\right)^{\sigma_1}\left(\dashint_{B(z_0,\frac{7r}{16})\cap\om}\frac{[\delta(y_0)]^{2\sigma_2}}{|y-y_0|^{2\sigma_2}}dx\right)^{\frac{1}{2}}\leq\frac{C[\delta(x_0)]^{\sigma_1}[\delta(y_0)]^{\sigma_2}}{|x_0-y_0|^{\sigma_1+\sigma_2}}.\nonumber
\end{align}

At last, when $ \delta(x_0)\geq\frac{1}{4}|x_0-y_0| $ and $ \delta(y_0)\geq\frac{1}{4}|x_0-y_0| $, we can choose $ F\in C_0^{\infty}(\om(x_0,\frac{r}{4});\mathbb{R}^m) $. Obviously, here, $ \om(x_0,\frac{r}{4})=B(x_0,\frac{r}{4}) $. Let $ w_{\va} $ satisfies $ \mathcal{L}_{\va}^*(w_{\va})=F $ in $ \om $ and $ w_{\va}=0 $ on $ \pa\om $. From the representation theorem, we can easily obtain that $ w_{\va}(y)=\int_{\om}G_{\va}(z,y)F(z)dz $. Since $ F\equiv 0 $ in $ \om\backslash\om(x_0,\frac{r}{4}) $, we have, $ \mathcal{L}_{\va}^*(w_{\va})=0 $ in $ \om\backslash\om(x_0,\frac{r}{4}) $. For all $ p>2 $ and $ 1<q<2 $, by using the $ W^{1,p} $ estimates $ \eqref{W^{1,p} interior estimates} $ and Sobolev embedding theorem, we have
\begin{align}
|w_{\va}(y_0)|&\leq  C\left(\dashint_{\om(y_0,\frac{r}{4})}|w_{\va}(z)|^2dz\right)^{\frac{1}{2}}\leq C\left(\dashint_{\om(y_0,\frac{r}{4})}|w_{\va}(z)|^pdz\right)^{\frac{1}{p}}\nonumber\\
&=Cr^{-\frac{2}{p}}\left(\int_{\om}|\nabla w_{\va}(z)|^2dz\right)^{\frac{1}{2}}\leq Cr^{-\frac{2}{p}}\left\|F\right\|_{L^q(\om)}\leq Cr^{-\frac{2}{p}+\frac{2}{q}-1}\left(\int_{\om(x_0,\frac{r}{4})}|F|^2dz\right)^{\frac{1}{2}}.\nonumber
\end{align}
This implies that
\begin{align}
\left|\int_{\om(x_0,\frac{r}{4})}G_{\va}(z,y_0)F(z)dz\right|\leq Cr^{-\frac{2}{p}+\frac{2}{q}-1}\left(\int_{\om(x_0,\frac{r}{4})}|F|^2dz\right)^{\frac{1}{2}}.\label{ww21}
\end{align}
From   \eqref{ww21}, we have
\begin{align}
\left(\dashint_{\om(x_0,\frac{r}{4})}|G_{\va}(z,y)|^2dz\right)^{\frac{1}{2}}\leq Cr^{-\frac{2}{p}+\frac{2}{q}-2},\label{ww22}
\end{align}
by using the duality arguments. For all $ \sigma\in(0,1) $, we can choose special $ p,q $ such that $ -\frac{2}{p}+\frac{2}{q}-2=-\sigma $. $ \eqref{ww21} $, together with $ \eqref{infty interior estimates} $ gives the proof of $ \eqref{preliminary} $. Of course, there is still a certain distance between this and $ \eqref{Pointwise estimates for Green functions 4} $, so we need to make more precise estimates.

For $ x_0,y_0\in\om $, letting $ r_1=\frac{1}{2}|x_0-y_0| $, we see that if $ \delta(x_0)<\frac{1}{2}|x_0-y_0| $, according to $ \eqref{Boundedness estimates} $, we can obtain $ |G_{\va}(x_0,y_0)|\leq C $. If $ \delta(x_0)\geq\frac{1}{2}|x_0-y_0| $, we can consider the sequence of subsets of $ \om $ denoted as $ \om_{j}=\om(x_0,2^jr_1) $ with $ j=0,1,...,N $ such that $ 2^Nr_1\geq\operatorname{diam}(\om) $. Note that $ N\leq C\left(1+\ln\left(\frac{\operatorname{diam}(\om)}{|x_0-y_0|}\right)\right) $. According to the fact that $ G_{\va}(\cdot,y_0)\in \operatorname{BMO}(\om) $, we obtain
\begin{align}
\left|\dashint_{\om_j}G_{\va}(x,y_0)-\dashint_{\om_{j+1}}G_{\va}(x,y_0)dx\right|\leq C\left\|G_{\va}(\cdot,y_0)\right\|_{*}\leq C.\label{Average difference estimates}
\end{align}
With the choice of $ N $, we get
\begin{align}
\dashint_{\om_N}|G_{\va}(x,y_0)|dx=\dashint_{\om}|G_{\va}(x,y_0)|dx\leq C\left\|G_{\va}(\cdot,y_0)\right\|_{*}\leq C.\label{Large scale integral}
\end{align}
Setting $ J=\dashint_{\om_0}G_{\va}^{\cdot\gamma}(x,y_0)dx $ and $ \theta_{\va}=G_{\va}(x,y_0)-J $, we have
\begin{align}
\mathcal{L}_{\va}(\theta_{\va})=\operatorname{div}(V(x/\va)J)+(-c(x/\va)J-\lambda J)\quad\text{in }B(x_0,r_1).\nonumber
\end{align}
In the view of $ \eqref{infty interior estimates} $, we can choose $ p>2 $, $ q=\frac{2p}{p+2} $, $ f=V(x/\va)J $ and $ F=-c(x/\va)J-\lambda J $. Then, we can obtain
\begin{align}
|\theta_{\va}(x_0)|&\leq  C\left\{\dashint_{B(x_0,r_1)}|\theta_{\va}|dx+r_1\left(\dashint_{B(x_0,r_1)}|f|^pdx\right)^{\frac{1}{p}}+r_1^2\left(\dashint_{B(x_0,r_1)}|F|^qdx\right)^{\frac{1}{q}}\right\}\nonumber\\
&\leq C\left\{\dashint_{B(x_0,r_1)}\left|G_{\va}(x,y_0)-\dashint_{\om_0}G_{\va}(z,y_0)dz\right|dx+|J|r_1+|J|r_1^2\right\},\nonumber\\
&\leq C(1+|J|r_1+|J|r_1^2).\label{The second important estimates}
\end{align}
The third inequality in $ \eqref{The second important estimates} $ is derived from $ \eqref{Average difference estimates} $ and
\begin{align}
\dashint_{B(x_0,r_1)}\left|G_{\va}(x,y_0)-\dashint_{\om_0}G_{\va}(z,y_0)dz\right|dx\leq C\left\|G_{\va}(\cdot,y_0)\right\|_{*}\leq C.\nonumber
\end{align}
Then we can see that
\begin{align}
\left|G_{\va}(x_0,y_0)-\dashint_{\om_0}G_{\va}(z,y_0)dz\right|\leq C(1+|J|r_1+|J|r_1^2).\nonumber
\end{align}
According to $ \eqref{Average difference estimates} $ and $\eqref{Large scale integral} $, we have
\begin{align}
|G_{\va}(x_0,y_0)|&\leq  C\sum_{j=-1}^{N}\left|\dashint_{\om_j}G_{\va}(x,y_0)dx-\dashint_{\om_{j+1}}G_{\va}(x,y_0)dx\right|+\dashint_{\om_N}|G_{\va}(x,y_0)|dx\nonumber\\
&\leq  C\left( 1+\ln\left(\frac{\operatorname{diam}(\om)}{|x_0-y_0|}\right)\right)+C\left|G_{\va}(x_0,y_0)-\dashint_{\om_0}G_{\va}(x,y_0)dx\right|,\nonumber
\end{align}
where we denote $ \dashint_{\om_{-1}}G_{\va}(x,y_0)=G_{\va}(x_0,y_0) $, and $ C $ depends only on $ \mu,\omega(t),\kappa,\lambda,m,p $. Set $ \sigma\in (0,1) $, by   $ \eqref{preliminary} $, we get
\begin{align}
|J|&= \left|\dashint_{\om_0}G_{\va}(x,y_0)dx\right|\leq\dashint_{\om_0}|G_{\va}(x,y_0)|dx\leq C\dashint_{B(x_0,r_1)}|x-y|^{-\sigma}dx\leq Cr_1^{-\sigma}.\nonumber
\end{align}
From \eqref{The second important estimates}, we can obtain $ |\theta_{\va}(x_0)|\leq C(1+r_1^{1-\sigma}+r_1^{2-\sigma})\leq C $ since $ r_1\leq\operatorname{diam}(\om) $. Then
\begin{align}
\left|G_{\va}(x_0,y_0)-\dashint_{\om_0}G_{\va}(x,y_0)dx\right|\leq C.\nonumber
\end{align}
Therefore, we can obtain the proof of $ \eqref{Pointwise estimates for Green functions 4} $, that is $ |G_{\va}(x,y)|\leq C\left(1+\ln\left(\frac{\operatorname{diam}(\om)}{|x-y|}\right)\right) $, where $ C $ depends only on $ \mu,\omega(t),\kappa,\lambda,m,p,\om $. Finally, we need to prove $ \eqref{Pointwise estimates for Green functions 5} $-$ \eqref{Pointwise estimates for Green functions 6} $. We only need to prove $ \eqref{Pointwise estimates for Green functions 5} $ and $ \eqref{Pointwise estimates for Green functions 6} $ follows directly from $ \eqref{adjoint p} $. For $ 0<\sigma_3<1 $, $ x_0,y_0,z_0\in\om $ such that $ |x_0-z_0|<\frac{1}{2}|x_0-y_0| $, we can define $ r_2=|x_0-y_0| $ and have
\begin{align}
|G_{\va}(x_0,y_0)-G_{\va}(z_0,y_0)|&\leq[G_{\va}(x,y_0)]_{C^{0,\sigma_3}(\om(x_0,\frac{1}{2}r))}|x_0-z_0|^{\sigma_3}.\nonumber
\end{align}
If $ \delta(x_0)<\frac{2}{3}r_2 $, noticing that $ \mathcal{L}_{\va}(G_{\va}^{\cdot\gamma}(x,y_0))=0 $ in $ \om(x_0,\frac{2}{3}r_2) $ and $ G_{\va}^{\cdot\gamma}(\cdot,y_0)=0 $ on $ \pa\om\cap B(x_0,\frac{2}{3}r_2) $, we get
\begin{align}
|G_{\va}(x_0,y_0)-G_{\va}(z_0,y_0)|&\leq [G_{\va}(x,y_0)]_{C^{0,\sigma_3}(\om(x_0,\frac{1}{2}r_2))}|x_0-z_0|^{\sigma_3}\nonumber\\
&\leq C|x_0-z_0|^{\sigma_3}r_2^{-\sigma_3}\left(\dashint_{\om(x_0,\frac{2}{3}r_2)}|G_{\va}(x,y_0)|^2dx\right)^{\frac{1}{2}}\nonumber\\
&\leq C|x_0-z_0|^{\sigma_3}r_2^{-\sigma_3}\left(\dashint_{\om(x_0,\frac{2}{3}r)}|G_{\va}(x,y_0)-\overline{G_{\va}(x,y_0)}_{x_0,\frac{2}{3}r_2}|^2dx\right)^{\frac{1}{2}}\nonumber\\
&\leq \frac{C|x_0-z_0|^{\sigma_3}}{|x_0-y_0|^{\sigma_3}},\label{ww36}
\end{align}
where we have used $ \eqref{Holder interior estimates} $, $ \eqref{Corollary 6.22} $ and $ \eqref{ww35} $. If $ \delta(x_0)>\frac{2}{3}r_2 $, we can choose $ J_1=\overline{G_{\va}(x,y_0)}_{x_0,\frac{2}{3}r} $ and have
\begin{align}
\mathcal{L}_{\va}(G_{\va}^{\cdot\gamma}(x,y_0)-J_1)=\operatorname{div}(V(x/\va)J_1)-c(x/\va)J_1-\lambda J_1.\nonumber
\end{align}
From $ \eqref{preliminary} $ and the fact that $ r_2\leq\diam(\om) $, we can obtain
\begin{align}
r_2|J_1|+r_2^2|J_1|\leq C.\label{ww37}
\end{align}
From \eqref{Holder interior estimates}, we have
\begin{align}
[G_{\va}(x,y_0)]_{C^{0,\sigma_3}(\om(x_0,\frac{1}{2}r))}\leq Cr^{-\sigma_3}\left\{\left(\dashint_{\om(x_0,\frac{2}{3}r)}|G_{\va}(x,y_0)-J_1|^2dx\right)^{\frac{1}{2}}+r_2|J_1|+r_2^2|J_1|\right\}.\nonumber
\end{align}
This, together with $ \eqref{Corollary 6.22} $ and $ \eqref{ww37} $, we can obtain
\begin{align}
[G_{\va}(x,y_0)]_{C^{0,\sigma_3}(\om(x_0,\frac{1}{2}r))}\leq C|x_0-y_0|^{-\sigma_3}.\label{ww38}
\end{align}
Combining $ \eqref{ww37} $ and $ \eqref{ww38} $, we can obtain  $ \eqref{Pointwise estimates for Green functions 5} $ and finish the proof.
\end{proof}

\begin{rem}
From $ \eqref{Pointwise estimates for Green functions 1} $-$ \eqref{Pointwise estimates for Green functions 4} $, we can obtain that for any $ 0<\sigma,\sigma_1,\sigma_2<1 $,
\begin{align}
|G_{\va}(x,y)|\leq \frac{C}{|x-y|^{\sigma}}\min\left\{1,\frac{[\delta(x)]^{\sigma_1}}{|x-y|^{\sigma_1}},\frac{[\delta(y)]^{\sigma_2}}{|x-y|^{\sigma_2}}\frac{[\delta(x)]^{\sigma_1}[\delta(y)]^{\sigma_2}}{|x-y|^{\sigma_1+\sigma_2}}\right\},\label{ww26}
\end{align}
where $ C $ depends only on $ \sigma_1,\sigma_2,\sigma,\mu,\omega(t),\kappa,\lambda,m $ and $ \om $
\end{rem}

\begin{rem}
For $ R\leq\delta(y) $, $ B(y,R) $ being a ball with center $ y $, radius $ R $, and $ F\in C_{0}^{\infty}(\om;\mathbb{R}^m) $, we can consider $ \mathcal{L}_{\va}^*(u_{\va})=F $ in $ \om $ and $ u_{\va}=0 $ on $ \pa\om $. Then, there exists a unique $ u_{\va}\in H_0^1(\om;\mathbb{R}^m) $ such that
\begin{align}
u_{\va}(y)=\int_{\om}G_{\va}(x,y)F(x)dx.\nonumber
\end{align}
If $ F\in C_0^{\infty}(\om;\mathbb{R}^m) $ and $ \operatorname{supp}(F)\subsetneq B(y,R)\subsetneq\om $, then, by $ \eqref{infty interior estimates} $, we have
\begin{align}
\left|\int_{\om}G_{\va}(x,y)F(x)dx\right|\leq\left\|u_{\va}\right\|_{L^{\infty}(B(y,\frac{1}{4}R))}\leq C\left(\dashint_{B(y,\frac{1}{2}R)}|u_{\va}|^2dx\right)^{\frac{1}{2}}+CR^2\left(\dashint_{B(y,\frac{1}{2}R)}|F|^2dx\right)^{\frac{1}{2}}.\nonumber
\end{align}
Note that $ u_{\va}=0 $ on $ \pa\om $, we can use Poincar\'{e}'s inequality and obtain
\begin{align}
\left(\dashint_{B(y,\frac{1}{2}R)}|u_{\va}|^2dx\right)^{\frac{1}{2}}&\leq  CR^{-1}\left(\int_{B(y,\frac{1}{2}R)}|u_{\va}|^2dx\right)^{\frac{1}{2}}\leq  CR^{-1}\left(\int_{B(y,2\delta(y))\cap\om}|u_{\va}|^2dx\right)^{\frac{1}{2}}\nonumber\\
&\leq  C\delta(y)R^{-1}\left(\int_{B(y,2\delta(y))\cap\om}|\nabla u_{\va}|^2dx\right)^{\frac{1}{2}}\leq  C\delta(y)R^{-1}\left(\int_{\om}|\nabla u_{\va}|^2dx\right)^{\frac{1}{2}}.\nonumber
\end{align}
From Theorem \ref{energy}, we have
\begin{align}
\left(\dashint_{B(y,\frac{1}{2}R)}|u_{\va}|^2dx\right)^{\frac{1}{2}}\leq C\delta(y)R^{-1}\left(\int_{\om}|\nabla u_{\va}|^2dx\right)^{\frac{1}{2}}\leq C\delta(y)R^{-1}\left\|F\right\|_{H^{-1}(\om)}.\label{h-1}
\end{align}
On the other hand, since $ F\in L^2(\om;\mathbb{R}^m) $, we note that actually, $ \langle F,\varphi\rangle_{H^{-1}(\om)\times H_0^1(\om)}=\int_{\om}F(x)\varphi(x)dx $. Then, by using Poincar\'{e}'s inequality,
\begin{align}
\left\|F\right\|_{H^{-1}(\om)}&=\sup_{\|\varphi\|_{H_0^1(\om)}=1}\left|\int_{\om}F(x)\varphi(x)dx\right|\leq \|F\|_{L^2(B(y,R))}\|\varphi\|_{L^2(B(y,R))}\nonumber\\
&\leq CR\|F\|_{L^2(B(y,R))}\|\nabla\varphi\|_{L^2(B(y,R))}\leq CR\|F\|_{L^2(B(y,R))}.\nonumber
\end{align}
This, together with $ \eqref{h-1} $ implies that
\begin{align}
\left(\dashint_{B(y,\frac{1}{2}R)}|u_{\va}|^2dx\right)^{\frac{1}{2}}\leq CR\delta(y)\left(\dashint_{B(y,R)}|F|^2dx\right)^{\frac{1}{2}}.\nonumber
\end{align}
Then, $ \left|\int_{\om}G_{\va}(x,y)F(x)dx\right|\leq CR\delta(y)\left(\dashint_{B(y,R)}|F|^2dx\right)^{\frac{1}{2}} $, where $ C $ depends only on $ \mu,\kappa,\lambda,m,\omega(t) $ and $ \om $. Using the duality methods, we have
\begin{align}
\left(\dashint_{B(y,R)}|G_{\va}(x,y)|^2dx\right)^{\frac{1}{2}}\leq C\frac{\delta(y)}{R},\nonumber
\end{align}
where $ C $ depends only on $ \mu,\omega(t),\kappa,\lambda,m, $ and $ \om $. If $ \delta(y)\leq CR $, we have
\begin{align}
\left(\dashint_{B(y,R)}|G_{\va}(x,y)|^2dx\right)^{\frac{1}{2}}\leq C.\label{Big ball integration}
\end{align}
\end{rem}
The following lemma is essential for the proof of Theorem \ref{Holder}. This lemma is from \cite{Shen1}. The author gave the proof for the elliptic operator without lower order terms. What is remarkable is that, in \cite{Shen1}, for the elliptic operator $ L_{\va} $ without lower order terms, the author does not give the proof for the case $ d=2 $.
\begin{lem}
Suppose that $ A $ satisfies $ \eqref{Ellipticity} $, $ \eqref{Periodicity} $ and $ \operatorname{VMO} $ condition $ \eqref{VMO condition} $. Other coefficients of $ \mathcal{L}_{\va} $, $ V,B,c $ satisfy $ \eqref{Boundedness} $. $ \om $ is a $ C^{1,\eta} $ $ (0<\eta<1) $ bounded domain in $ \mathbb{R}^2 $. Then, for all $ \sigma\in(0,1) $, we have
\begin{align}
\int_{\om}|\nabla_y G_{\va}(x,y)|[\delta(y)]^{\sigma-1}dy\leq C[\delta(x)]^{\sigma},\label{Green function integration estimates}
\end{align}
where $ C $ depends only on $ \mu,\omega(t),\kappa,\lambda,m,\sigma $ and $ \om $.
\end{lem}
\begin{proof}
Set $ x_0\in \om $ and $ r=\frac{1}{2}\delta(x_0) $. For all $ R<\frac{1}{4}\delta(x_0) $, we consider the annulus $ B(x_0,2R)\backslash B(x_0,R) $. We can use small two dimensional balls with radius of $ \frac{5}{8}R $ whose center are on the circle $ B(x_0,\frac{3R}{2}) $ to cover the annulus $ B(x_0,2R)\backslash B(x_0,R) $. We denote these small balls as $ \left\{B(x_i,\frac{5}{8}R)\right\}_{i=1}^{N} $. Obviously, we have $ N\leq C\frac{\pi(4R^2-R^2)}{\pi(\frac{5}{8} R)^2}\leq C $, where $ C $ is a constant, independent of $ R $. The specific positional relationship is shown in Figure 1, where $ r_1=\frac{2}{3}R $, $ r_2=\frac{5}{8}R $ and $ r_3=2R $.

\centerline{
\begin{tikzpicture}
\filldraw[color=black, fill=gray!50] (0,0) circle (4);
\filldraw[color=black, fill=black!0] (0,0) circle (2);
\filldraw[color=black, fill=black](2.4,1.8) circle (0.05);
\filldraw[color=black, fill=black](0,0) circle (0.05);
\filldraw[color=black, fill=gray!60] (2.4,1.8) circle (1.7);
\filldraw[color=black, fill=gray!70] (2.4,1.8) circle (1.25);
\node[color=black] at (-0.2,0.2) {$ x_0 $};
\node[color=black] at (2.4,2.0) {$ x_i $};
\node[color=black] at (-1.7,-1.6) {$ r_3 $};
\node[color=black] at (3.0,1.3) {$ r_1 $};
\node[color=black] at (1.9,1.5) {$ r_2 $};
\draw [color=black, thick][->](0,0) -- (-2.4,-3.2);
\draw [color=black, thick][->](2.4,1.8) -- (1.65,0.8);
\draw [color=black, thick][->](2.4,1.8) -- (1.65,0.8);
\draw [color=black, thick][->](2.4,1.8) -- (3.42,0.44);
\filldraw[color=black, fill opacity=0] (2.4,1.8) circle (1.7);
\filldraw[color=black, fill opacity=0] (2.4,1.8) circle (1.25);
\filldraw[color=black, fill opacity=0] (0,0) circle (4);
\filldraw[color=black, fill opacity=0] (0,0) circle (2);
\end{tikzpicture}
}
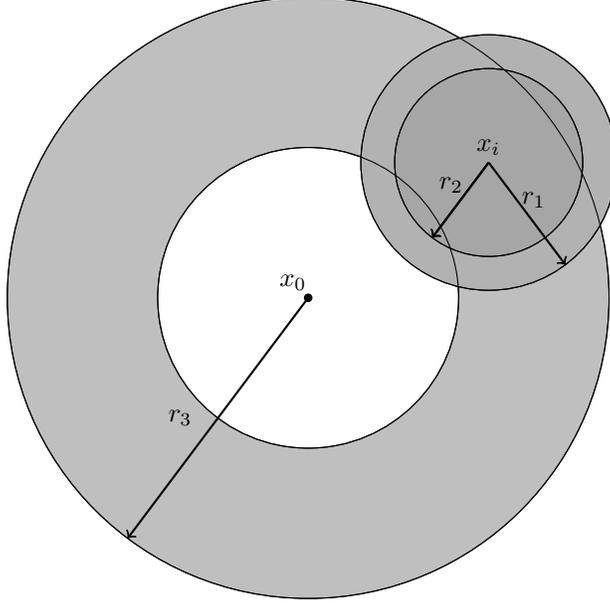
\centerline{Figure 1.Balls $ B(x_i.\frac{5}{8}R) $ construct a covering of $ B(x_0,2R)\backslash B(x_0,R) $.}
Then, for $ u_{\va}(y)=G_{\va}(x_0,y) $, we have $ \mathcal{L}_{\va}^*(u_{\va})=0 $ in $ B(x_0,3R)\backslash B(x_0,\frac{1}{2} R) $. Using H\"{o}lder's inequality, we have
\begin{align}
\int_{B(x_0,2R)\backslash B(x_0,R)}|\nabla_y G_{\va}(x_0,y)|dy&\leq C\left(\int_{B(x_0,2R)\backslash B(x_0,R)}|\nabla_y G_{\va}(x_0,y)|^2dy\right)^{\frac{1}{2}}R\nonumber\\
&\leq C\sum_{i=1}^{N}\left(\int_{B(x_i,\frac{5}{8} R)}|\nabla_y G_{\va}(x_0,y)|^2dy\right)^{\frac{1}{2}}R.\nonumber
\end{align}
Using the previous annotation of Caccioppoli's inequality $ \eqref{Caccioppoli rem} $, we have
\begin{align}
&\left(\int_{B(x_i,\frac{5}{8} R)}|\nabla_y G_{\va}(x_0,y)|^2dy\right)^{\frac{1}{2}}\leq\frac{C}{R}\left(\int_{B(x_i,\frac{2}{3} R)}\left| G_{\va}(x_0,y)-\dashint_{B(x_i,\frac{2}{3} R)}G_{\va}(x_0,z)dz\right|^2dy\right)^{\frac{1}{2}}\nonumber\\
&\quad\quad\quad\quad+C\left|\dashint_{B(x_i,\frac{2}{3} R)}G_{\va}(x_0,z)dz\right|R+C\left|\dashint_{B(x_i,\frac{2}{3} R)}G_{\va}(x_0,z)dz\right|R^2.\nonumber
\end{align}
For any $ \sigma_1\in(0,1) $, by using $ \eqref{preliminary} $, we have
\begin{align}
\left|\dashint_{B(x_i,\frac{2}{3} R)}G_{\va}(x_0,z)dz\right|\leq\dashint_{B(x_i,\frac{2}{3} R)}|G_{\va}(x_0,z)|dz\leq CR^{-\sigma_1}.\label{ww23}
\end{align}
Since $ G_{\va}(x_0,y)\in \BMO(\om) $, in view of $ \eqref{Corollary 6.22} $, we have
\begin{align}
\left(\dashint_{B(x_i,\frac{2}{3} R)}\left| G_{\va}(x_0,y)-\dashint_{B(x_i,\frac{2}{3} R)}G_{\va}(x_0,z)dz\right|^2dy\right)^{\frac{1}{2}}\leq C\left\|G_{\va}(x_0,y)\right\|_{*}\leq C.\label{ww24}
\end{align}
Combining it with $ \eqref{ww23} $, $ \eqref{ww24} $ and the fact that $ R\leq \diam(\om) $, we can obtain
\begin{align}
\left(\int_{B(x_i,\frac{5}{8} R)}|\nabla_y G_{\va}(x_0,y)|^2dy\right)^{\frac{1}{2}}\leq C,\nonumber
\end{align}
where $ C $ depends only on $ \mu,\omega(t),\kappa,\lambda,m,\sigma_1 $ and $ \om $. Then
\begin{align}
\int_{B(x_0,2R)\backslash B(x_0,R)}|\nabla_y G_{\va}(x_0,y)|dy\leq CR.\nonumber
\end{align}
Therefore, we have
\begin{align}
\int_{B(x_0,r)}|\nabla_y G_{\va}(x_0,y)|dy&\leq \sum_{j=0}^{\infty}\int_{B(x_0,2^{-j}r)\backslash B(x_0,2^{-j-1}r)}|\nabla_y G_{\va}(x_0,y)|dy\nonumber\\
&\leq C\sum_{j=1}^{\infty}2^{-j}r+\int_{B(x_0,r)\backslash B(x_0,\frac{1}{2} r)}|\nabla_y G_{\va}(x_0,y)|dy\nonumber\\
&\leq Cr+\left(\int_{B(x_0,r)\backslash B(x_0,\frac{1}{2} r)}|\nabla_y G_{\va}(x_0,y)|^2dy\right)^{\frac{1}{2}}r\nonumber\\
&\leq Cr+\left(\dashint_{B(x_0,\frac{3}{2} r)\backslash B(x_0,\frac{1}{4}r)}|G_{\va}(x_0,y)|^2dy\right)^{\frac{1}{2}}r\leq Cr.\nonumber
\end{align}
For the last inequality, we have used $ \eqref{Big ball integration} $. Then, we have
\begin{align}
\int_{B(x_0,r)}|\nabla_y G_{\va}(x_0,y)|[\delta(y)]^{\sigma-1}dy\leq Cr^{\sigma}.\label{Improvement of Shen1}
\end{align}
Next, to estimate the integral on $ \om\backslash B(x_0,r) $, we observe that if $ Q $ is a cube in $ \mathbb{R}^2 $ with the property $ 3Q\subset\om\backslash\left\{x_0\right\} $ and its side length $ l(Q)\sim \operatorname{dist}(Q,\pa\om) $, then
\begin{align}
\int_{Q}|\nabla_y G_{\va}(x_0,y)|[\delta(y)]^{\sigma-1}dy&\leq C[l(Q)]^{\sigma-1}|Q|\left(\dashint_{Q}|\nabla_y G_{\va}(x_0,y)|^2dy\right)^{\frac{1}{2}}\nonumber\\
&\leq C[l(Q)]^{\sigma-2}|Q|\left(\dashint_{2Q}|G_{\va}(x_0,y)|^2dy\right)^{\frac{1}{2}},\nonumber
\end{align}
where we have used Caccioppoli's inequality $ \eqref{Caccioppoli inequality} $ for the last step. This, together with the pointwise estimates $ \eqref{Pointwise estimates for Green functions 1} $-$ \eqref{Pointwise estimates for Green functions 3} $, gives
\begin{align}
\int_{Q}|\nabla_y G_{\va}(x_0,y)|[\delta(y)]^{\sigma-1}dy&\leq Cr^{\sigma_1}[l(Q)]^{\sigma+\sigma_2-2}|Q|\left(\dashint_{2Q}\frac{dy}{|x_0-y|^{2(\sigma_1+\sigma_2)}}\right)^{\frac{1}{2}}\nonumber\\
&\leq Cr^{\sigma_1}[l(Q)]^{\sigma+\sigma_2-2}|Q|\dashint_{2Q}\frac{dy}{|x_0-y|^{(\sigma_1+\sigma_2)}}\nonumber\\
&\leq Cr^{\sigma_1}\int_{Q}\frac{[\delta(y)]^{\sigma+\sigma_2-2}}{|x_0-y|^{\sigma_1+\sigma_2}}dy,\label{Integration on cubes}
\end{align}
where $ 0<\sigma_1,\sigma_2<1 $, and we have used the observation that $ \delta(y)\sim l(Q) $ for $ y\in 2Q $ and $ |x-y|\sim |x-z| $ for any $ y,z\in 2Q $. Finally, we perform a Witney decomposition on $ \om $(see \cite{EM}). This gives $ \om=\cup_{j}Q_j $ where $ \left\{Q_j\right\} $ is a sequence of (closed) non-overlapping cubes with the property that $ 4Q_j\subset \om $ and $ l(Q_j)\sim\operatorname{dist}(Q_j,\pa\om) $. Let
\begin{align}
\mathcal{O}=\cup_{3Q_j\subset\om\backslash\left\{x_0\right\}}Q_j.\nonumber
\end{align}
Note that if $ y\in\om\backslash\mathcal{O} $, then  $ y\in Q_j $ for some $ Q_j $ such that $ x_0\in 3Q_j $. It follows that $ |y-x_0|\leq Cl(Q_j)\leq C\delta(x_0) $. Hence
\begin{align}
\int_{\om\backslash\mathcal{O}}|\nabla_y G_{\va}(x_0,y)|[\delta(y)]^{\sigma-1}dy\leq Cr^{\sigma}\nonumber
\end{align}
by $ \eqref{Big ball integration} $ and $ \eqref{Improvement of Shen1} $. By the summation, the estimate $ \eqref{Integration on cubes} $ leads to
\begin{align}
\int_{\mathcal{O}\backslash B(x_0,r)}|\nabla_y G_{\va}(x_0,y)|[\delta(y)]^{\sigma-1}dy&\leq  Cr^{\sigma_1}\int_{\om}\frac{[\delta(y)]^{\sigma+\sigma_2-2}}{(|x_0-y|+r)^{\sigma_1+\sigma_2}}dy,\label{ww25}
\end{align}
where  we have used the fact that $ |x_0-y|\geq c_1\delta(x_0)=c_1r $ for any $ y\in\mathcal{O}\backslash B(x_0,r) $ with some positive constant $ c_1$.  Since $ \om $ is $ C^{1,\eta} $, the integral in the RHS of $ \eqref{ww25} $ is bounded by
\begin{align}
  C\int_{0}^{\infty}\int_{\mathbb{R}}\frac{t^{\sigma+\sigma_2-2}}{(|y'|+|t-r|+r)^{\sigma_1+\sigma_2}}dy'dt&\leq C\int_{0}^{\infty}\frac{t^{\sigma+\sigma_2-2}}{(|t-r|+r)^{\sigma_1+\sigma_2-1}}dt \nonumber\\
 &\leq Cr^{\sigma-\sigma_1}\int_{0}^{\infty}\frac{t^{\sigma+\sigma_2-2}}{(|t-1|+1)^{\sigma_1+\sigma_2-1}}dt\leq Cr^{\sigma-\sigma_1},\nonumber
\end{align}
where we have chosen $ \sigma_1,\sigma_2\in (0,1) $ so that $ \sigma_1+\sigma_2>1 $ and $ \sigma<\sigma_1<1 $. This, together with $ \eqref{Improvement of Shen1} $ and $ \eqref{ww25} $, completes the proof.
\end{proof}
Using the above lemma, we will prove Theorem \ref{Holder}.
\begin{proof}[The proof of Theorem \ref{Holder}]
First we assume that $ u_{\va,1} $ satisfies the equation
\begin{align}
\mathcal{L}_{\va}(u_{\va,1})=0\quad\text{in }\om\quad\text{and}\quad u_{\va,1}=g\quad\text{on }\pa\om. \nonumber
\end{align}
Choose $ v $ such that $ \Delta v^{\alpha}=0 $ in $ \om $ and $ v^{\alpha}=g^{\alpha} $ on $ \pa\om $. By the classical interior Lipschitz estimates for harmonic functions, Caccioppoli's inequality $ \eqref{Caccioppoli inequality} $ and H\"{o}lder estimate: $ [v]_{C^{0,\sigma}(\om)}\leq C\left\|g\right\|_{C^{0,\sigma}(\pa\om)} $, we have
\begin{align}
|\nabla v(x)|&\leq  C\left(\dashint_{B(x,\frac{1}{4}\delta(x))}|\nabla v|^2dy\right)^{\frac{1}{2}}\leq\frac{C}{\delta(x)}\left(\dashint_{B(x,\frac{1}{2}\delta(x))}|v(y)-v(x)|^2dy\right)^{\frac{1}{2}}\nonumber\\
&\leq C[\delta(x)]^{\sigma-1}[v]_{C^{0,\sigma}(\om)}\leq C[\delta(x)]^{\sigma-1}\left\|g\right\|_{C^{0,\sigma}(\pa\om)},\nonumber
\end{align}
for all $ \sigma\in (0,1) $. Let $ w_{\va}=u_{\va,1}-v $, we get
\begin{align}
\mathcal{L}_{\va}(w_{\va})=-\mathcal{L}_{\va}(v)\quad\text{ in }\om\quad\text{and}\quad w_{\va}=0\quad\text{on }\pa\om.\nonumber
\end{align}
Using the normalization, we can assume that $ \left\|g\right\|_{C^{0,\sigma}(\om)}=1 $. According to the representation theorem $ \eqref{Representation formula} $, we obtain
\begin{align}
w_{\va}(x)&=-\int_{\om}\nabla_y G_{\va}(x,y)[A(y/\va)\nabla v+V(y/\va)v]dy-\int_{\om}G_{\va}(x,y)[B(y/\va)\nabla v+(c(y/\va)+\lambda)v]dy,\nonumber
\end{align}
which implies that
\begin{align}
|w_{\va}(x)|&\leq  C\int_{\om}|\nabla_yG_{\va}(x,y)|[\delta(y)]^{\sigma-1}dy
+C\int_{\om}|G_{\va}(x,y)|[\delta(y)]^{\sigma-1}dy\nonumber\\
&\quad+C\int_{\om}(|\nabla_yG_{\va}(x,y)|+|G_{\va}(x,y)|)dy\nonumber\\
&\leq I_1+I_2+I_3,\nonumber
\end{align}
where we use the fact that $ v $ is bounded in $ \om $ as a result of the maximum principle. In view of $\eqref{Green function integration estimates} $, we have $ I_1\leq C[\delta(x)]^{\sigma} $. Next, we will estimate $ I_2 $. In fact, by using $ \eqref{ww26} $, we get
\begin{align}
I_2&\leq \left(\int_{\om\cap\left\{\delta(y)<\frac{1}{4}|x-y|\right\}}+\int_{\om\cap\left\{\delta(y)\geq\frac{1}{4}|x-y|\right\}}\right)|G_{\va}(x,y)|[\delta(y)]^{\sigma-1}dy\nonumber\\
&\leq\int_{\om\cap\left\{\delta(y)<\frac{1}{4}|x-y|\right\}}\frac{C[\delta(x)]^{\sigma}[\delta(y)]^{\sigma_1}}{|x-y|^{\sigma+\sigma_1}}[\delta(y)]^{\sigma-1}dy+\int_{\om\cap\left\{\delta(y)\geq\frac{1}{4}|x-y|\right\}}\frac{C[\delta(x)]^{\sigma}}{|x-y|^{\sigma+\sigma_2}}[\delta(y)]^{\sigma-1}dy\nonumber\\
&\leq \int_{\om\cap\left\{\delta(y)<\frac{1}{4}|x-y|\right\}}\frac{C[\delta(x)]^{\sigma}}{|x-y|}dy+\int_{\om\cap\left\{\delta(y)\geq\frac{1}{4}|x-y|\right\}}\frac{C[\delta(x)]^{\sigma}}{|x-y|^{\sigma_2+1}}dy\leq C[\delta(x)]^{\sigma}.\nonumber
\end{align}
Here, we choose $ \sigma_1\in(0,1) $ such that $ \sigma_1+\sigma>1 $ and $ \sigma_2\in(0,1) $. Then, we have $ I_2\leq C[\delta(x)]^{\sigma} $. Moreover, as $ I_3\leq C(I_1+I_2) $, we have $ I_3\leq C[\delta(x)]^{\sigma} $. Therefore,
\begin{align}
|w_{\va}(x)|\leq C[\delta(x)]^{\sigma}.\label{ww27}
\end{align}
For $ u_{\va} $, we can choose $ u_{\va,2}=u_{\va}-u_{\va,1} $, then, $ \mathcal{L}_{\va}(u_{\va,2})=\operatorname{div}(f)+F $ in $ \om $ and $ u_{\va,2}=0 $ on $ \pa\om $. From the H\"{o}lder estimate $ \eqref{Holder estimates elementary} $, we get
\begin{align}
\left\|u_{\va,2}\right\|_{C^{0,\sigma}(\om)}\leq C\left\{\left\|F\right\|_{L^q(\om)}+\left\|f\right\|_{L^p(\om)}\right\}.\nonumber
\end{align}
So the problem is reduced down to the proof of $ \left\|u_{\va,1}\right\|_{C^{0,\sigma}(\om)}\leq C $. By the definition of $ v $, we only need to show $ \left\|w_{\va}\right\|_{C^{0,\sigma}(\om)}\leq C $. We divide it into the following three cases: (1) $ |x-y|\leq \frac{1}{4}\delta(x) $; (2) $ |x-y|\leq \frac{1}{4}\delta(y) $; (3) $ |x-y|> \frac{1}{4}\delta(x) $ and $ |x-y|\leq \frac{1}{4}\delta(y) $. For the first case, choose $ r=\delta(x) $ and in view of $ \mathcal{L}_{\va}(w_{\va})=-\mathcal{L}_{\va}(v) $ with $ \eqref{Holder interior estimates} $, we have
\begin{align}
|w_{\va}(x)-w_{\va}(y)|&\leq [w_{\va}]_{C^{0,\sigma}(B(x,\frac{r}{4}))}|x-y|^{\sigma}\nonumber\\
&\leq  C|x-y|^{\sigma}\left\{r^{-\sigma}\left(\dashint_{B(x,\frac{r}{2})}|w_{\va}|^2dz\right)^{\frac{1}{2}}+r^{1-\sigma}\left(\dashint_{B(x,\frac{r}{2})}(|\nabla v|^p+|v|^p)dz\right)^{\frac{1}{p}}\right\}\nonumber\\
&\quad+C|x-y|^{\sigma}r^{2-\sigma}\left(\dashint_{B(x,\frac{r}{2})}(|\nabla v|^q+|v|^q)dz\right)^{\frac{1}{q}}\nonumber\\
&\leq  C|x-y|^{\sigma}\nonumber
\end{align}
where $ q=\frac{2p}{p+2} $ and $ \sigma=1-\frac{2}{p} $. The case (2) is similar to the case (1) and for the case (3), we obtain
\begin{align}
|w_{\va}(x)-w_{\va}(y)|\leq|w_{\va}(x)|+|w_{\va}(y)|\leq C|x-y|^{\sigma}.\nonumber
\end{align}
This completes the proof of Theorem \ref{Holder}.
\end{proof}

Using standard localization arguments in \cite{Xu1}, we can obtain the following result.

\begin{thm}[localization of Lipschitz estimates]\label{localization Lipschitz estimates}
Suppose that $ A\in\Lambda(\mu,\tau,\kappa),V $ satisfies $ \eqref{Periodicity} $, $ \eqref{regularity} $, $ B $ and $ c $ satisfy $ \eqref{Periodicity} $, $ \eqref{Boundedness} $, $ \lambda\geq\lambda_{0} $ and $ \om $ is a bounded $ C^{1,\eta} $ domain in $ \mathbb{R}^d $ with $ (0<\eta<1) $, $ d\geq 2 $. Let $ x_0\in\om $ and $ 0<r<\operatorname{diam}(\om) $. If $ \pa\om\cap B(x_0,2r)\neq\emptyset $, assume that $ u_{\va}\in W^{1,2}(\om(x_0,2r);\mathbb{R}^m) $ is the weak solution to
\begin{align}
\left\{\begin{matrix}
\mathcal{L}_{\va}(u_{\va})=F&\text{in}&\om(x_0,2r),\\
u_{\va}=0&\text{on}&\Delta(x_0,2r),
\end{matrix}\right.\nonumber
\end{align}
with $ F\in L^p(\om(x_0,2r);\mathbb{R}^m) $ and $ p>d $. If $ \pa\om\cap B(x_0,2r)=\emptyset $, assume that $ u_{\va}\in W^{1,2}(B(x_0,2r);\mathbb{R}^m) $ is a weak solution to
\begin{align}
\mathcal{L}_{\va}(u_{\va})=F\quad\text{in}\quad B(x_0,2r),\nonumber
\end{align}
with the same data $ F $ and $ p $, then
\begin{align}
\left\|\nabla u_{\va}\right\|_{L^{\infty}(\om(x_0,r))}\leq\frac{C}{r}\left\{\left(\dashint_{\om(x_0,2r)}|u_{\va}|^2\right)^{\frac{1}{2}}+r^2\left(\dashint_{\om(x_0,2r)}|F|^p\right)^{\frac{1}{p}}\right\},\label{localization Lipschitz estimates e}
\end{align}
where $ C $ depends only on $ \mu,\tau,\kappa,\lambda,p,d,m, \sigma,\eta $ and $ \om $.
\end{thm}

\begin{proof}[Proof of Theorem \ref{Lipschitz estimates Green}]
For $ \eqref{Pointwise estimates for Green functions 11} $, $ \eqref{Pointwise estimates for Green functions 21} $ and $ \eqref{Pointwise estimates for Green functions 31} $, the proof is trivial since we just need to change localization of the H\"{o}lder estimates $ \eqref{Holder interior estimates} $ to the localization of $ \eqref{localization Lipschitz estimates e} $. Also, we note that the constant $ \frac{1}{4} $ in $ \eqref{Pointwise estimates for Green functions 11} $-$ \eqref{Pointwise estimates for Green functions 31} $ can be replaced by $ \frac{1}{3} $. Then, we only need to show $ \eqref{Lipschitz estimates Green e 1} $-$ \eqref{Lipschitz estimates Green e 3} $. For $ x_0,y_0\in\om $, we can set $ r=|x_0-y_0| $. If $ \delta(x_0)<\frac{1}{4}r $, we have, $ \mathcal{L}_{\va}(G_{\va}^{\cdot\gamma}(\cdot,y_0))=0 $ in $ \om(x_0,\frac{1}{4}r) $ and $ G_{\va}^{\cdot\gamma}(\cdot,y_0)=0 $ on $ \pa\om\cap B(x_0,\frac{1}{4}r) $. Then, by using $ \eqref{Boundedness estimates} $ and $ \eqref{localization Lipschitz estimates e} $, we can obtain that
\begin{align}
|\nabla_x G_{\va}(x_0,y_0)|\leq\frac{C}{r}\left(\dashint_{B(x_0,\frac{1}{4}r)}|G_{\va}(x,y_0)|^2dx\right)^{\frac{1}{2}}\leq\frac{C}{|x_0-y_0|}.\label{ww44}
\end{align}
If $ \delta(y_0)<\frac{1}{4}r $, for any $ x\in B(x_0,\frac{1}{8}r) $, we have $ |x-y_0|\geq |x_0-y_0|-\frac{1}{8}r=\frac{7}{8}r $ and $ \delta(y_0)<\frac{2}{7}|x-y_0|\leq \frac{1}{3}|x-y_0| $. From \eqref{Pointwise estimates for Green functions 21}  and $ \eqref{localization Lipschitz estimates e} $, it can be obtained that
\begin{align}
|\nabla_x G_{\va}(x_0,y_0)|\leq\frac{C}{r}\left(\dashint_{B(x_0,\frac{1}{8}r)}|G_{\va}(x,y_0)|^2dx\right)^{\frac{1}{2}}\leq\frac{C\delta(y_0)}{|x_0-y_0|^2}.\label{4.35-2}
\end{align}
If $ \frac{1}{4}r\leq\delta(x_0) $ and $ \frac{1}{4}r\leq\delta(y_0) $, we have
\begin{align}
\frac{1}{4}|x_0-y_0|\leq \frac{3}{8}|x_0-y_0|=\frac{1}{2}\left(r-\frac{1}{4}r\right)\leq \frac{1}{2}|x-y_0|\quad\text{ for any }x\in B\left(x_0,\frac{1}{4}r\right).\label{ww39}
\end{align}
At first, there exists a point $ \overline{y}\in\pa\om $ (see Figure 2), such that
\begin{equation}
(x_0-\overline{y})//(x_0-y_0)\text{ and }(x_0-\overline{y})\cdot(x_0-y_0)>0.\label{condition}
\end{equation}
Then according to the fact that $ \frac{1}{4}|x_0-y_0|=\frac{1}{4}r\leq |y_0-\overline{y}| $, there always exists a positive integer $ N\in\mathbb{N}_+ $ and a sequence of points $ \{y_j\}_{j=1}^N $ such that
\begin{align}
&y_{j}=x_0+\frac{5}{4}(y_{j-1}-x_0),\ j=1,\ldots, N,\nonumber\\
&\frac{1}{4}|x_0-y_{N-1}|\leq |y_{N-1}-\overline{y}|\quad\text{and}\quad \frac{1}{4}|x_0-y_{N}|>|y_{N}-\overline{y}|.\nonumber
\end{align}

\centerline{
\begin{tikzpicture}
\filldraw[color=black, fill=gray!50] (-5,0) circle (0.5);
\filldraw[color=black, fill=black](-5,0) circle (0.05);
\filldraw[color=black, fill=black](-3,0) circle (0.05);
\filldraw[color=black, fill=black](-2.5,0) circle (0.05);
\filldraw[color=black, fill=black](-1.875,0) circle (0.05);
\filldraw[color=black, fill=black](-1.093,0) circle (0.05);
\filldraw[color=black, fill=black](5.125,0) circle (0.05);
\filldraw[color=black, fill=black](8,0) circle (0.05);
\node[color=black] at (-5,0.2) {$ x_0 $};
\node[color=black] at (-3,0.2) {$ y_0 $};
\node[color=black] at (-2.5,0.2) {$ y_1 $};
\node[color=black] at (-1.875,0.2) {$ y_2 $};
\node[color=black] at (-1.093,0.2) {$ y_3 $};
\node[color=black] at (5.125,0.2) {$ y_N $};
\node[color=black] at (8.2,0.2) {$ \overline{y} $};
\draw [color=black, thick][-](-5,0) -- (8,0);
\draw [color=black, thick][dotted](3,0.2) -- (4,0.2);
\draw [color=black, thick][dashed](8,1.5) -- (8,-0.5);
\node[color=black] at (8.4,0.9) {$ \partial\om $};
\end{tikzpicture}
}
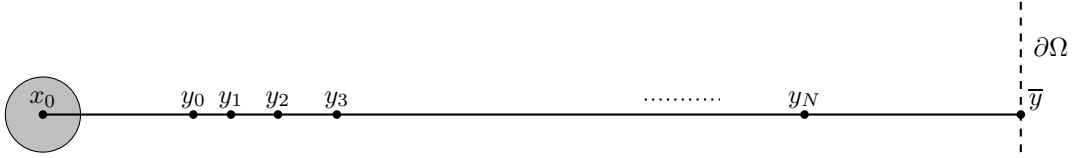
\centerline{Figure 2.}
\noindent
Moreover, we have
\begin{align}
\mathcal{L}_{\va}(G_{\va}(x,y_0)-G_{\va}(x,y_1))=0\quad\text{ in }B\left(x_0,\frac{1}{4}r\right).\label{ww41}
\end{align}
From $ \eqref{Pointwise estimates for Green functions 6} $ and $ \eqref{ww39} $, since $ |y_1-y_0|=\frac{1}{4}|x_0-y_0|\leq \frac{1}{2}|x-y_0| $, we can obtain that for $ \sigma\in(0,1) $,
\begin{align}
|G_{\va}(x,y_0)-G_{\va}(x,y_1)|\leq \frac{C|y_1-y_0|^{\sigma}}{|x-y_0|^{\sigma}}\leq C\quad\text{ for any }x\in B\left(x_0,\frac{1}{4}r\right).\label{ww40}
\end{align}
By applying $ \eqref{localization Lipschitz estimates e} $ to $ \eqref{ww41} $, we can obtain
\begin{align}
|\nabla_x (G_{\va}(x_0,y_0)-G_{\va}(x_0,y_1))|\leq\frac{C}{r}
\left(\dashint_{B(x_0,\frac{1}{4}r)}|G_{\va}(x,y_0)-G_{\va}(x,y_1)|^2\right)^{\frac{1}{2}}\leq\frac{C}{|x_0-y_0|}.\label{ww43}
\end{align}
Similarly, owing to the fact that $ |y_{j+1}-y_j|\leq\frac{1}{2}|x_0-y_{j}| $, we have
\begin{equation}
|\nabla_x (G_{\va}(x_0,y_i)-G_{\va}(x_0,y_{i+1}))|\leq\frac{C}{|x_0-y_i|}
\left(\dashint_{B(x_0,\frac{1}{4}|x_0-y_i|)}|G_{\va}(x,y_i)-G_{\va}(x,y_{i+1})|^2\right)^{\frac{1}{2}}\leq\frac{C}{|x_0-y_i|},\label{ww43-i}
\end{equation}
where $ i=1,2,...,N-1 $. Finally, because of the simple observation that
\begin{equation}
\frac{1}{4}|x_0-y_N|>|y_N-\overline{y}|\geq\delta(y_N), \nonumber
\end{equation}
we can deduce from $ \eqref{4.35-2} $ that
\begin{equation}
|\nabla_x G_{\va}(x_0,y_N)|\leq\frac{C\delta(y_N)}{|x_0-y_N|^2}\leq\frac{C}{|x_0-y_N|}.
\label{ww43-N}
\end{equation}
From \eqref{ww43}-\eqref{ww43-N}, we get
    \begin{equation}
      |\nabla_xG_{\va}(x_0,y_0)|\leq \sum_{i=0}^N\frac{C}{|x_0-y_i|}\leq \frac{C}{r}\sum_{i=0}^N\left(\frac{4}{5}\right)^i\leq \frac{C}{r}\leq \frac{C}{|x_0-y_0|}.\label{4.40}
    \end{equation}
Therefore,
  \begin{equation}
      |\nabla_xG_{\va}(x_0,y_0)|\leq   \frac{C}{r},\quad \textrm{ when } \frac{1}{4}r\leq\delta(x_0) \textrm{ and } \frac{1}{4}r\leq\delta(y_0).\nonumber
    \end{equation}
This, together with  \eqref{ww44}  and \eqref{4.35-2}, gives the proof of $ \eqref{Lipschitz estimates Green e 1} $.
Then, $ \eqref{Lipschitz estimates Green e 2} $ follows directly from $ \eqref{Lipschitz estimates Green e 1} $ by considering the adjoint Green function $ G_{\va}^*(x,y) $. Finally, by applying $ \eqref{localization Lipschitz estimates e} $ to $ \nabla_yG_{\va}^{\cdot\gamma}(\cdot,y_0) $ and using $ \eqref{ww44} $, we can prove $ \eqref{Lipschitz estimates Green e 3} $.
\end{proof}

\section{$ L^p $ convergence rates}

In this section, we will consider the convergence of the Green functions for $ \mathcal{L}_{\va} $.

To handle the convergence rates of $ \mathcal{L}_{\va} $, we define some auxiliary functions via
\begin{align}
b_{ik}^{\al\gamma}(y)=\widehat{a}_{i k}^{\al\gamma}-a_{ik}^{\al \gamma}(y)-a_{ij}^{\al\beta}(y)\pa_j\chi_{k}^{\beta \gamma}(y),b_{i0}^{\al \gamma}(y)=\widehat{V}_{i}^{\al \gamma}-V_{i}^{\al \gamma}(y)-a_{i j}^{\al\beta}(y) \pa_j\chi_{0}^{\beta \gamma}(y),\label{b function}
\end{align}
and
\begin{align}
&\Delta\Theta_{i}^{\al \gamma}=\widehat{B}_{i}^{\al \gamma}-B_{i}^{\al \gamma}(y)-B_{j}^{\al \beta}(y)\pa_j\chi_{i}^{\beta \gamma}(y)\quad \text{in }\mathbb{R}^{d},\quad\int_{Y}\Theta_{i}^{\al \beta}(y)dy=0,\label{theta function 1}\\
&\Delta\Theta_{0}^{\al\gamma}=\widehat{c}^{\al\gamma}-c^{\al\gamma}(y)-B_{i}^{\al\beta}(y)\pa_i\chi_{0}^{\beta \gamma}(y)\quad\text{in }\mathbb{R}^{d},\quad\int_{Y}\Theta_{0}^{\al\beta}(y) dy=0,\label{theta function 2}
\end{align}
with $ 1\leq i\leq d $. We mention that the existence of $ \Theta_{k} $ is given by Theorem 4.28 in \cite{Cioranescu} on account of $\int_{Y} \Theta_{k}^{\al\gamma}(y)dy=0 $ for $ k=0,1,\ldots,d $. Furthermore it is not hard to see that $ \Theta_{k}^{\al\gamma} $ is periodic and belongs to $ W_{\operatorname{loc}}^{1,2}\left(\mathbb{R}^{d}\right) $.

Suppose that $ A\in\Lambda(\mu,\tau,\kappa) $, and $ V $ satisfies $ \eqref{Periodicity} $ and $ \eqref{regularity} $. From the interior Schauder estimate (see \cite{MG}), we obtain that
\begin{align}
\max_{0\leq k\leq d}\{\|\chi_k\|_{L^{\infty}(Y)},\|\nabla\chi_k\|_{L^{\infty}(Y)},[\nabla\chi_k]_{C^{0,\tau}(Y)}\}\leq C(\mu,\tau,\kappa,m,d).\label{boundedness of chi}
\end{align}

By the same argument, it follows from $ \eqref{boundedness of chi} $ that
\begin{align}
\max_{0\leq k\leq d}\{\left\|\nabla\Theta_k^{\al\gamma}\right\|_{L^{\infty}(Y)}\}\leq C(\mu,\tau,\kappa,m,d).\label{boundedness of Theta}
\end{align}

\begin{thm}[Convergence of Green functions for $ \mathcal{L}_{\va} $]\label{Green functions convergence}
Suppose that $ A\in\Lambda(\mu,\tau,\kappa),V,B $ satisfy $ \eqref{Periodicity} $ and $ \eqref{regularity} $, $ c $ satisfies $ \eqref{Periodicity} $ and $ \eqref{Boundedness} $, $\lambda \geq \lambda_{0} $ and $ \om $ is a bounded $ C^{1,1} $ domain in $ \mathbb{R}^d $ with $ d\geq 2 $.  Then, for $ G_{\va}(x,y) $ and $ G_0(x,y) $ being the Green functions for $ \mathcal{L}_{\va} $ and $ \mathcal{L}_{0} $, we have
\begin{align}
|G_{\va}(x,y)-G_{0}(x,y)|\leq\frac{C\va}{|x-y|^{d-1}},\label{Convergence of Green functions}
\end{align}
where $ C $ depends only on $ \mu,\tau,\kappa,\lambda,d,m $ and $ \om $.
\end{thm}

\begin{lem}[\cite{Xu2}, Lemma 2.9]\label{Lemma 2.8}
There exist $ E_{jik}^{\alpha\gamma}\in W_{\operatorname{per}}^{1,2}(Y) $ with $ k=0,1,\ldots,d $, such that
\begin{align}
b_{i k}^{\alpha \gamma}=\pa_j\left\{E_{jik}^{\alpha \gamma}\right\}\quad\text{and }\quad E_{jik}^{\alpha \gamma}=-E_{ijk}^{\alpha \gamma}\label{E function}
\end{align}
where $1 \leq i, j \leq d $ and $1 \leq \alpha, \gamma \leq m$. Moreover, if $\chi_{k} $ is H\"{o}lder continuous, then $ E_{jik}^{\alpha\gamma}\in L^{\infty}(Y) $.
\end{lem}

\begin{lem}[\cite{Xu2}, Lemma 5.1]\label{Lemma 5.1}
Suppose that $ u_{\va},u_{0}\in H^{1}(\om;\mathbb{R}^{m}) $ satisfy $ \mathcal{L}_{\va}(u_{\va})=\mathcal{L}_{0}(u_{0}) $ in $ \om $. Let
\begin{align}
w_{\va}^{\beta}=u_{\va}^{\beta}-u_{0}^{\beta}-\va \sum_{k=0}^{d}\chi_{k}^{\beta\gamma}(x/\va)\pa_{k}u_0^{\gamma},\label{ww46}
\end{align}
Then we have
\begin{align}
[\mathcal{L}_{\va}\left(w_{\va}\right)]^{\al}(x)&=-\pa_i\left\{\mathcal{K}_{i}^{\al}(x)-\va\left(\mathcal{I}_{i}^{\al}(x)+\mathcal{J}_{i}^{\al}(x)\right)\right\}-\va(\mathcal{M}^{\al}(x)+\mathcal{N}^{\al}(x)),\label{ww47}
\end{align}
where
\begin{align}
\mathcal{I}_{i}^{\al}(x)&=a_{ij}^{\al\beta}(x/\va)\sum_{k=0}^{d}\chi_{k}^{\beta \gamma}(x/\va)\pa_{jk}^2u_0^{\gamma}+V_{i}^{\al\beta}(x/\va)\sum_{k=0}^{d}\chi_{k}^{\beta \gamma}(x/\va)\pa_ku_0^{\gamma},\nonumber\\
 \quad \mathcal{J}_{i}^{\al}(x)&=\sum_{k=0}^{d} \pa_i\Theta_{k}^{\al\gamma}(x/\va) \pa_ku_0^{\gamma}, \quad \mathcal{K}_{i}^{\al}(x)=\sum_{j=0}^{d}b_{ij}^{\al \gamma}(x/\va)\pa_{j}u_0^{\gamma},\nonumber\\
\mathcal{M}^{\al}(x)&=\sum_{k=0}^{d}\left[\pa_i\Theta_{k}^{\al\gamma}(x/\va)+B_{i}^{\al\beta}(x/\va)\chi_{k}^{\beta\gamma}(x/\va)\right]\pa_{ik}^2u_0^{\gamma},\nonumber\\\mathcal{N}^{\al}(x)&=\left[c^{\al\beta}(x/\va)+\lambda\delta^{\al\beta}\right]\sum_{k=0}^{d} \chi_{k}^{\beta\gamma}(x/\va)\pa_ku_0^{\gamma}.\nonumber
\end{align}
\end{lem}

\begin{rem}
From Lemma \ref{Lemma 5.1} and $ \eqref{E function} $, we have
\begin{align}
\pa_i(\mathcal{K}_{i}^{\al}(x))=\pa_i\left(\sum_{j=0}^db_{ij}^{\al\gamma}(x/\va)\pa_ju_0^{\gamma}\right)=\va\pa_j\left(\sum_{k=0}^d E_{jik}^{\al\gamma}(x/\va)\pa_ku_0^{\gamma}\right).\label{E formula}
\end{align}
\end{rem}

\begin{lem}\label{infty b}
Suppose that $ A\in\Lambda(\mu,\tau,\kappa),V,B $ satisfy $ \eqref{Periodicity} $ and $ \eqref{regularity} $, $ c $ satisfies $ \eqref{Periodicity} $ and $ \eqref{Boundedness} $, $\lambda \geq \lambda_{0} $ and $ \om $ is a bounded $ C^{1,\eta} $ $ (0<\eta<1) $ domain in $ \mathbb{R}^d $ with $ d\geq 2 $.  Then, we have
\begin{align}
\|u_{\va}\|_{L^{\infty}(\om(x_0,r))}\leq C\|f\|_{L^{\infty}(\Delta(x_0,3r))}+C\dashint_{\om(x_0,3r)}|u_{\va}|,\label{infty b e}
\end{align}
where $ \mathcal{L}_{\va}(u_{\va})=0 $ in $ \om(x_0,3r) $,  $ u_{\va}=f $ on $ \Delta(x_0,3r) $, $ x_0\in\om $, $ 0<r<\diam(\om) $ and $ C $ depends only on $ \mu,\tau,\kappa,\lambda,d,m,\eta,\om $.
\end{lem}
\begin{proof}
The proof is almost the same as Lemma 6.3.2 in \cite{Shen1}. We only need to change the operator $ L_{\va} $ in \cite{Shen1} to the operator $ \mathcal{L}_{\va} $ with lower order terms and use $ \eqref{ww32} $ and $ \eqref{infty interior estimates} $.
\end{proof}

\begin{lem}\label{ww33}
Suppose that $ A\in\Lambda(\mu,\tau,\kappa),V,B $ satisfy $ \eqref{Periodicity} $ and $ \eqref{regularity} $, $ c $ satisfies $ \eqref{Periodicity} $ and $ \eqref{Boundedness} $, $\lambda \geq \lambda_{0} $ and $ \om $ is a bounded $ C^{1,\eta} $ $ (0<\eta<1) $ domain in $ \mathbb{R}^d $ with $ d\geq 2 $.  Let $ u_{\va}\in W^{1,2}(\om(x_0,4r);\mathbb{R}^m) $ and $ u_0\in W^{2,p}(\om(x_0,4r);\mathbb{R}^m) $ for some $ d<p<\infty $. Suppose that
\begin{align}
\mathcal{L}_{\va}(u_{\va})=\mathcal{L}_0(u_0)\quad\text{in }\om(x_0,4r)\quad\text{and}\quad u_{\va}=u_0\quad\text{on }\Delta(x_0,4r).\nonumber
\end{align}
Then, we have
\begin{align}
&\|u_{\va}-u_0\|_{L^{\infty}(\om(x_0,r))}\nonumber\\
&\quad\quad\leq C\dashint_{\om(x_0,4r)}|u_{\va}-u_0|+C\va\|\nabla u_0\|_{L^{\infty}(\om(x_0,4r))}+C\va r^{1-\frac{d}{p}}\left\|\nabla^2u_0\right\|_{L^p(\om(x_0,4r))},\label{ww34}
\end{align}
where $ x_0\in\om $, $ 0<r<\diam(\om) $ and $ C $ depends only on $ \mu,\tau,\kappa,\lambda,d,m,\eta,\om $.
\end{lem}

\begin{proof}
By rescaling and translation, we can assume that $ r=1 $. Choose a domain $ \widetilde{\om} $, which is $ C^{1,\eta} $, such that $ \om(0,3)\subset\widetilde{\om}\subset\om(0,4) $. Consider
\begin{align}
w_{\va}=u_{\va}-u_{0}-\va \sum_{k=0}^{d}\chi_{k}^{\gamma}(x/\va)\pa_{k}u_0^{\gamma}=w_{\va,1}+w_{\va,2}\quad\text{in }\widetilde{\om},\nonumber
\end{align}
where
\begin{align}
\mathcal{L}_{\va}(w_{\va,1})&=\mathcal{L}_{\va}(w_{\va})\quad\text{in }\widetilde{\om}\quad w_{\va,1}\in W_0^{1,2}(\widetilde{\om};\mathbb{R}^m),\label{ww48}\\
\mathcal{L}_{\va}(w_{\va,2})&=0\quad\text{in }\widetilde{\om}\quad w_{\va,2}=w_{\va},\quad\text{on }\pa\widetilde{\om}.\label{ww49}
\end{align}
Since $ w_{\va,2}=w_{\va}=-\va \sum_{k=0}^{d}\chi_{k}^{\beta\gamma}(x/\va)\pa_{k}u_0^{\gamma} $ on $ \Delta(0,3) $ and $ \left\|\chi\right\|_{L^{\infty}}\leq C $ (by using $ \eqref{boundedness of chi} $), it follows from $ \eqref{infty b e} $ that
\begin{align}
\left\|w_{\va,2}\right\|_{L^{\infty}(\om(0,1))}&\leq C\va\|\nabla u_0\|_{L^{\infty}(\Delta(0,3))}+C\dashint_{\om(0,3)}|w_{\va,2}|\nonumber\\
&\leq C\va\|\nabla u_0\|_{L^{\infty}(\Delta(0,3))}+C\dashint_{\om(0,3)}|w_{\va}|+C\dashint_{\om(0,3)}|w_{\va,1}|\nonumber\\
&\leq C\dashint_{\om(0,3)}|u_{\va}-u_0|+C\va\left\|\nabla u_0\right\|_{L^{\infty}(\om(0,3))}+C\left\|w_{\va,1}\right\|_{L^{\infty}(\om(0,3))}.\nonumber
\end{align}
This gives
\begin{align}
\|u_{\va}-u_0\|_{L^{\infty}(\om(0,1))}\leq C\dashint_{\om(0,3)}|u_{\va}-u_0|+C\va\left\|\nabla u_0\right\|_{L^{\infty}(\om(0,3))}+C\left\|w_{\va,1}\right\|_{L^{\infty}(\om(0,3))}.\label{ww50}
\end{align}
To estimate $ w_{\va,1} $ on $ \om_3 $, we use the Green function representation
\begin{align}
w_{\va,1}(x)=\int_{\widetilde{\om}}\widetilde{G}_{\va}(x,y)\mathcal{L}_{\va}(w_{\va})dy,\nonumber
\end{align}
where $ \widetilde{G}_{\va}(x,y) $ denotes the matrix of the Green function for $ \mathcal{L}_{\va} $ in $ \widetilde{\om} $. From $ \eqref{ww47} $, we obtain
\begin{align}
[w_{\va,1}(x)]^{\delta}&=\va\int_{\widetilde{\om}}\pa_{y_j}\widetilde{G}_{\va}^{\delta\al}(x,y)\sum_{k=0}^{d}E_{jik}^{\al\gamma}(y/\va)\pa_{jk}^2u_0^{\gamma}(y)dy\nonumber\\
&\quad\quad-\va\int_{\widetilde{\om}}\pa_{y_j}\widetilde{G}_{\va}^{\delta\al}(x,y)(\mathcal{I}_{i}^{\al}(y)+\mathcal{J}_{i}^{\al}(y))-\va\int_{\widetilde{\om}}\widetilde{G}_{\va}^{\delta\al}(x,y)(\mathcal{M}^{\al}(y)+\mathcal{N}^{\al}(y))dy.\nonumber
\end{align}
From Lemma \ref{Lemma 2.8}, we have $ \|E\|_{L^{\infty}(\widetilde{\om})}\leq C $ and it follows that
\begin{align}
|w_{\va,1}(x)|&\leq C\va\int_{\widetilde{\om}}|\nabla_y\widetilde{G}_{\va}(x,y)||\nabla^2u_0(y)|dy+C\va\int_{\widetilde{\om}}|\nabla_y\widetilde{G}_{\va}(x,y)||\nabla u_0(y)|dy\nonumber\\
&\quad\quad+C\va\int_{\widetilde{\om}}|\widetilde{G}_{\va}(x,y)||\nabla^2u_0(y)|dy+C\va\int_{\widetilde{\om}}|\widetilde{G}_{\va}(x,y)||\nabla u_0(y)|dy.\nonumber
\end{align}
Since under the assumption of $ \mathcal{L}_{\va} $, from \eqref{preliminary}  and \eqref{Lipschitz estimates Green e 2}, we have, for $ p>d $,
\begin{align}
&C\va\int_{\widetilde{\om}}|\nabla_y\widetilde{G}_{\va}(x,y)||\nabla^2u_0(y)|dy+C\va\int_{\widetilde{\om}}|\widetilde{G}_{\va}(x,y)||\nabla^2u_0(y)|dy\nonumber\\
&\quad\quad\leq C\va\left\|\nabla^2u_0\right\|_{L^p(\om(0,4))}\left\{\left(\int_{\widetilde{\om}}|\nabla_y\widetilde{G}_{\va}(x,y)|^{p'}dy\right)^{\frac{1}{p'}}+\left(\int_{\widetilde{\om}}|\widetilde{G}_{\va}(x,y)|^{p'}dy\right)^{\frac{1}{p'}}\right\}\nonumber\\
&\quad\quad\leq C\va\left\|\nabla^2u_0\right\|_{L^p(\om(0,4))},\nonumber
\end{align}
and
\begin{align}
&C\va\int_{\widetilde{\om}}|\nabla_y\widetilde{G}_{\va}(x,y)||\nabla u_0(y)|dy+C\va\int_{\widetilde{\om}}|\widetilde{G}_{\va}(x,y)||\nabla u_0(y)|dy\nonumber\\
&\quad\quad\leq C\va\left\|\nabla u_0\right\|_{L^{\infty}(\om(0,4))}\left(\int_{\widetilde{\om}}|\nabla_y\widetilde{G}_{\va}(x,y)|dy+\int_{\widetilde{\om}}|\widetilde{G}_{\va}(x,y)|dy\right)\leq C\va\left\|\nabla u_0\right\|_{L^{\infty}(\om(0,4))}.\nonumber
\end{align}
\end{proof}

\begin{proof}[Proof of Theorem \ref{Green functions convergence}]
We first note that under the assumption of $ \mathcal{L}_{\va} $, $ \mathcal{L}_0 $ and $ \om $ in the theorem, the size estimate  and $ |\nabla_xG_0(x,y)|\leq C|x-y|^{1-d} $ hold for any $ x,y\in\om $ and $ x\neq y $. We now fix $ x_0,y_0\in\om $ and $ r=\frac{1}{8}|x_0-y_0|>0 $. For $ F\in C_0^{\infty}(\om(y_0,r);\mathbb{R}^m) $, let
\begin{align}
u_{\va}(x)=\int_{\widetilde{\om}}G_{\va}(x,y)F(y)dy\quad\text{and}\quad u_0(x)=\int_{\widetilde{\om}}G_{0}(x,y)F(y)dy.\nonumber
\end{align}
Then $ \mathcal{L}_{\va}(u_{\va})=\mathcal{L}_0(u_0)=F $ in $ \om $ and $ u_{\va}=u_0=0 $ on $ \pa\om $. Note that since $ \om $ is $ C^{1,1} $,
\begin{align}
\left\|\nabla^2u_0\right\|_{L^p(\om)}&\leq C\left\|F\right\|_{L^p(\om)}\quad\text{for }1<p<\infty,\label{ww51}\\
\left\|\nabla u_0\right\|_{L^{\infty}(\om)}&\leq Cr^{1-\frac{d}{p}}\left\|F\right\|_{L^p(\om(y_0,r))}\quad\text{for }p>d.\label{ww52}
\end{align}
The inequality $ \eqref{ww51} $ is the $ W^{2,p} $ estimate for the constant second order elliptic operator $ \mathcal{L}_0 $ in $ C^{1,1} $ domains (see \cite{Tr}) and $ \eqref{ww52} $ follows from the estimate $ |\nabla_yG_0(x,y)|\leq C|x-y|^{1-d} $ and H\"{o}lder's inequality.

Next, let
\begin{align}
w_{\va}=u_{\va}-u_0-\va \sum_{k=0}^{d}\chi_{k}^{\gamma}(x/\va)\pa_{k}u_0^{\gamma}=v_{\va,1}+v_{\va,2},\nonumber
\end{align}
where $ v_{\va,1}\in W_0^{1,2}(\om;\mathbb{R}^m) $ and $ \mathcal{L}_{\va}(v_{\va,1})=\mathcal{L}_{\va}(w_{\va}) $ in $ \om $. Observe that by the formula $ \eqref{ww47} $ for $ \mathcal{L}_{\va} $,
\begin{align}
\left\|\nabla v_{\va,1}\right\|_{L^2(\om)}&\leq C\va\left\|\nabla^2u_0\right\|_{L^2(\om)}+C\va\left\|\nabla u_0\right\|_{L^2(\om)}\leq C\va\left\|F\right\|_{L^2(\om(y_0,r))},\nonumber
\end{align}
where we have used the fact that $ \chi_k $ and $ E_{kij} $ are bounded, $ \eqref{Energy estimates} $ and $ \eqref{ww51} $. By H\"{o}lder's inequality and Sobolev inequalities, this implies that if $ d\geq 3 $,
\begin{align}
\left\|v_{\va,1}\right\|_{L^2(\om(x_0,r))}&\leq \left\|v_{\va,1}\right\|_{L^2(\om)}\leq Cr\left\|v_{\va,1}\right\|_{L^{\frac{2d}{d-2}}(\om)} 
\leq Cr\left\|\nabla v_{\va,1}\right\|_{L^{2}(\om)}\leq C\va r^{1+\frac{d}{2}-\frac{d}{p}}\left\|F\right\|_{L^p(\om(y_0,r))},\label{ww53}
\end{align}
for $ p>d $. We point out that if $ d=2 $, one has
\begin{align}
\left\|v_{\va,1}\right\|_{L^2(\om(x_0,r))}\leq C\va r\left\|F\right\|_{L^2(\om(y_0,r))},\nonumber
\end{align}
in place of $ \eqref{ww53} $. To see this we use the $ W^{1,p} $ estimates $ \eqref{W^{1,p} estimates} $. Thus there exists some $ \overline{p}<2 $ such that
\begin{align}
\left\|\nabla v_{\va,1}\right\|_{L^{\overline{p}}(\om)}\leq C\va r\|\nabla^2u_0\|_{L^{\overline{p}}(\om)}+C\va r\|\nabla u_0\|_{L^{\overline{p}}(\om)}\leq C\va\|F\|_{L^{\overline{p}}(\om(y_0,r))},\label{ww54}
\end{align}
which, by H\"{o}lder's inequality and Sobolev inequality, leads to
\begin{align}
\|v_{\va,1}\|_{L^2(\om(x_0,r))}&\leq Cr^{1-\frac{2}{q}}\|v_{\va,1}\|_{L^q(\om(x_0,r))}\leq Cr^{1-\frac{2}{q}}\|v_{\va,1}\|_{L^q(\om)}\leq Cr^{1-\frac{2}{q}}\|\nabla v_{\va,1}\|_{L^{\overline{p}}(\om)}\nonumber\\
&\leq Cr^{2-\frac{2}{\overline{p}}}\|\nabla v_{\va,1}\|_{L^{\overline{p}}(\om)}\leq C\va r^{2-\frac{2}{\overline{p}}}\|F\|_{L^{\overline{p}}(\om(y_0,r))}\leq C\va r\left\|F\right\|_{L^2(\om(y_0,r))},\label{ww55}
\end{align}
where $ \frac{1}{q}=\frac{1}{\overline{p}}-\frac{1}{2} $.

Observe that since $ \mathcal{L}_{\va}(v_{\va,2})=0 $ in $ \om $ and $ v_{\va,2}=w_{\va} $ on $ \pa\om $, by the maximum principle $ \eqref{ww32} $, we have
\begin{align}
\|v_{\va,2}\|_{L^{\infty}(\om)}\leq C\left\|v_{\va,2}\right\|_{L^{\infty}(\pa\om)}\leq C\va\left\|\nabla u_0\right\|_{L^{\infty}(\pa\om)}.\label{ww56}
\end{align}
From $ \eqref{ww51} $-$ \eqref{ww56} $, we obtain
\begin{align}
\left\|u_{\va}-u_{0}\right\|_{L^{2}(\om(x_{0},r))}
&\leq\left\|v_{\va,1}\right\|_{L^{2}(\om(x_{0},r))}+\left\|v_{\va,2}\right\|_{L^{2}(\om(x_{0},r))}+C\left\|\nabla u_{0}\right\|_{L^{2}(\om(x_{0},r))}\nonumber\\
&\leq\left\|v_{\va,1}\right\|_{L^{2}(\om(x_{0},r))}+\left\|v_{\va,2}\right\|_{L^{2}(\om(x_{0},r))}+C\va r^{\frac{d}{2}}\left\|\nabla u_{0}\right\|_{L^{\infty}(\om)}\nonumber\\
&\leq\left\|v_{\va,1}\right\|_{L^{2}(\om(x_{0}, r))}+C\va r^{\frac{d}{2}}\left\|\nabla u_{0}\right\|_{L^{\infty}(\om)}\leq C \va r^{1+\frac{d}{2}-\frac{d}{p}}\|F\|_{L^{p}(\om(y_{0},r))},\nonumber
\end{align}
where $ p>d $. This, together with Lemma \ref{ww33} and $ \eqref{ww51} $, gives
\begin{align}
|u_{\va}(x_{0})-u_{0}(x_{0})| \leq C \va r^{1-\frac{d}{p}}\|F\|_{L^p(\om(y_{0},r))}.\label{ww57}
\end{align}
Then, it follows by the duality arguments that
\begin{align}
\left(\int_{\om(y_{0},r)}|G_{\va}(x_{0}, y)-G_{0}(x_{0},y)|^{p^{\prime}}dy\right)^{\frac{1}{p'}}\leq C\va r^{1-\frac{d}{p}}\quad\text{for any }p>d.\label{ww58}
\end{align}
Finally, since $ \mathcal{L}_{\va}^{*}(G_{\va}(x_{0}, \cdot))=\mathcal{L}_{0}^{*}(G_{0}(x_{0},\cdot))=0 $ in $ \om(y_{0}, r)$, we may invoke Lemma \ref{ww33} again to conclude that
\begin{align}
|G_{\va}(x_{0}, y_{0})-G_{0}(x_{0}, y_{0})|&\leq \dashint_{\om(y_{0},r)}|G_{\va}(x_{0}, y)-G_{0}(x_{0}, y)| dy+C \va\left\|\nabla_{y} G_{0}\left(x_{0}, \cdot\right)\right\|_{L^{\infty}(\om(y_{0},r))} \nonumber\\
&\quad\quad+C_{p} \va r^{1-\frac{d}{p}}\left\|\nabla_{y}^{2} G_{0}(x_{0},\cdot)\right\|_{L^p(\om(y_{0},r))}\leq C \va r^{1-d},\nonumber
\end{align}
where we have used
\begin{align}
\left(\dashint_{\om(y_0, r)}\left|\nabla_{y}^{2} G_{0}(x_{0}, y)\right|^{p}dy\right)^{\frac{1}{p}}&\leq Cr^{-2}\left\|G_{0}(x_{0},\cdot)\right\|_{L^{\infty}(\om(x_0,2 r))}\leq Cr^{-d},\nonumber
\end{align}
obtained by using $ W^{2,p} $ estimates on $ C^{1,1} $ domains for $ \mathcal{L}_{0}^* $, which is the localization of the estimates $ \eqref{ww51} $.
\end{proof}

\begin{rem}
$ \eqref{Convergence rates Lp} $ follows directly from $ \eqref{Convergence of Green functions} $ by using the standard arguments of Theorem 3.4 in \cite{Kenig2}.
\end{rem}

\section*{Acknowledgments}
We are grateful to Professor Qiang Xu of Lanzhou University for warm guidance on the topics of the homogenization theory for elliptic operators with lower order terms. We are also grateful to Professor Zhongwei Shen of the University of Kentucky for some essential hints on the Homogenization theory. We sincerely thank the anonymous reviewers for their constructive revision suggestions.


\begin{thebibliography}{99}

\bibitem{RA}
\newblock R.A. Adams and J.J.F. Fournier,
\newblock Sobolev spaces,
\newblock Elsevier/Academic Press, Amsterdam, 2003.


\bibitem{Av}
\newblock M. Avellaneda and F. Lin,
\newblock Compactness methods in the theory of homogenization,
\newblock \emph{Communications on Pure and Applied Mathematics}, \textbf{42} (1987), 139–172.


\bibitem{Av2}
\newblock M. Avellaneda and F. Lin,
\newblock Compactness methods in the theory of homogenization II: Equations in non‐divergence form,
\newblock \emph{Communications on pure and applied mathematics}, \textbf{42} (1989), 139-172.


\bibitem{Cioranescu}
\newblock D. Cioranescu and P. Donato,
\newblock An introduction to homogenization,
\newblock The Clarendon Press, Oxford University Press, New York, 1999.

\bibitem{Cho1}
\newblock S. Cho, H. Dong and S. Kim,
\newblock On the Green's matrices of strongly parabolic systems of second order,
\newblock \emph{Indiana University Mathematics Journal}, \textbf{57} (2008), 1633–1677.


\bibitem{Dong2}
\newblock H. Dong and S. Kim,
\newblock Green's matrices of second order elliptic systems with measurable coefficients in two dimensional domains,
\newblock \emph{Transactions of the American Mathematical Society}, \textbf{361} (2009), 3303-3323.


\bibitem{Dong1}
\newblock H. Dong and S. Kim,
\newblock Green's function for nondivergence elliptic operators in two dimensions,
\newblock \emph{SIAM Journal on Mathematical Analysis}, \textbf{53} (2021), 4637-4656.


\bibitem{Geng1}
\newblock J. Geng and Z. Shen,
\newblock Uniform regularity estimates in parabolic homogenization,
\newblock \emph{Indiana University Mathematics Journal}, \textbf{64} (2015), 697–733.


\bibitem{Geng2}
\newblock J. Geng and B. Shi,
\newblock Green's matrices and boundary estimates in parabolic homogenization,
\newblock \emph{Journal of Differential Equations}, \textbf{269} (2020), 3031–3066.


\bibitem{MG}
\newblock M. Giaquinta and L. Martinazzi,
\newblock An introduction to the regularity theory for elliptic systems, harmonic maps and minimal graphs,
\newblock Edizioni della Normale, Pisa, 2005.

\bibitem{Tr}
\newblock D. Gilbarg and N.S. Trudinger,
\newblock Periodic homogenization of Green and Neumann functions,
\newblock Springer-Verlag, Berlin-New York, 1977.

\bibitem{Kenig2}
\newblock C. E. Kenig, F. Lin and Z. Shen,
\newblock Periodic homogenization of Green and Neumann functions,
\newblock \emph{Communications on Pure and Applied Mathematics}, \textbf{67} (2014), 1219–1262.

\bibitem{Kenig3}
\newblock C. E. Kenig, F. Lin and Z. Shen,
\newblock Estimates of eigenvalues and eigenfunctions in periodic homogenization,
\newblock \emph{Journal of the European Mathematical Society}, \textbf{15} (2013), 1901-1925.

\bibitem{Shen2}
\newblock Z. Shen,
\newblock Boundary value problems in Morrey spaces for elliptic systems on Lipschitz domains,
\newblock \emph{American Journal of Mathematics}, \textbf{125} (2003), 1079-1115.

\bibitem{Shen3}
\newblock Z. Shen,
\newblock Bounds of Riesz transforms on $ L^p $ spaces for second order elliptic operators,
\newblock \emph{Universit\'{e} de Grenoble. Annales de l'Institut Fourier}, \textbf{55} (2005), 173–197.

\bibitem{Shen4}
\newblock Z. Shen,
\newblock Necessary and sufficient conditions for the solvability of the $ L^p $ Dirichlet problem on Lipschitz domains,
\newblock \emph{Mathematische Annalen}, \textbf{336} (2006), 697–725.

\bibitem{Shen5}
\newblock Z. Shen,
\newblock The $ L^p $ boundary value problems on Lipschitz domains,
\newblock \emph{Advances in Mathematics}, \textbf{216} (2007), 212–254.

\bibitem{Shen6}
\newblock Z. Shen,
\newblock $ W^{1,p} $ estimates for elliptic homogenization problems in nonsmooth domains,
\newblock \emph{Indiana University Mathematics Journal}, \textbf{57} (2008), 2283-2298.

\bibitem{Shen1}
\newblock Z. Shen,
\newblock Periodic homogenization of elliptic systems,
\newblock Birkhäuser/Springer, Cham, 2018.

\bibitem{EM}
\newblock E. M. Stein,
\newblock Singular integrals and differentiability properties of functions,
\newblock  Princeton university press, 1970.

\bibitem{Taylor}
\newblock J.L. Taylor, S. Kim and R.M. Brown,
\newblock The Green function for elliptic systems in two dimensions,
\newblock \emph{Communications in Partial Differential Equations}, \textbf{38} (2013), 1574-1600.

\bibitem{Xu1}
\newblock Q. Xu,
\newblock Uniform regularity estimates in homogenization theory of elliptic system with lower order terms,
\newblock \emph{Journal of Mathematical Analysis and Applications}, \textbf{438} (2016), 1574-1600.

\bibitem{Xu2}
\newblock Q. Xu,
\newblock Uniform regularity estimates in homogenization theory of elliptic systems with lower order terms on the Neumann boundary problem,
\newblock \emph{Journal of Differential Equations}, \textbf{261} (2016), 4368-4423.

\end{thebibliography}
\end{document}